\documentclass[a4paper,12pt]{amsart}
\usepackage{amsmath,amssymb,amsfonts,amsthm}

\usepackage[alphabetic]{amsrefs}
\usepackage{bm}

\usepackage{graphicx}
\usepackage{ascmac}
\usepackage[all]{xy}
\usepackage{amsthm}
\usepackage[top=25mm,bottom=25mm,left=30mm,right=30mm]{geometry}
\usepackage{tikz}
\usetikzlibrary{cd,arrows,matrix,backgrounds,shapes,positioning,calc,decorations.markings,decorations.pathmorphing,decorations.pathreplacing}
\tikzset{blackv/.style={circle,fill=black,inner sep=3pt,outer sep=3pt},
         whitev/.style={circle,fill=white,draw=black,inner sep=3pt,outer sep=3pt},
         blabel/.style={circle,draw=black,inner sep=1.5pt,outer sep=0pt},
         redv/.style={circle,fill=red,inner sep=3pt,outer sep=3pt},
         block/.style={draw,rectangle split,rectangle split horizontal,rectangle split parts=#1},
         symbol/.style={
           draw=none,
           every to/.append style={
             edge node={node [sloped, allow upside down, auto=false]{$#1$}}}}
}
\usepackage{pxpgfmark}
\usepackage{multirow}
\usetikzlibrary{intersections,arrows,calc,decorations.markings}
\newtheorem{theorem}{Theorem}[section]

\newtheorem{proposition}[theorem]{Proposition}
\newtheorem{defprop}[theorem]{Definition \textbf{\&} Proposition}

\newtheorem{corollary}[theorem]{Corollary}
\newtheorem{lemma}[theorem]{Lemma}

\theoremstyle{definition}
\newtheorem{remark}[theorem]{Remark}
\newtheorem{example}[theorem]{Example}
\newtheorem{definition}[theorem]{Definition}

\makeatletter
 
 \@addtoreset{equation}{section}
\makeatother

\def\xx{\mathbf{x}}

\def\TT{\mathbb{T}}

\def\ZZ{\mathbb{Z}}

\def\Ccal{\mathcal{C}}
\def\Dcal{\mathcal{D}}
\def\Fcal{\mathcal{F}}
\def\Xcal{\mathcal{X}}

\def\mod{\opname{mod}\nolimits}
\def\proj{\opname{proj}\nolimits}

\newcommand{\opname}[1]{\operatorname{\mathsf{#1}}}

\newcommand{\Hom}{\opname{Hom}}
\newcommand{\End}{\opname{End}}

\newcommand{\Ext}{\opname{Ext}}
\newcommand{\add}{\opname{add}\nolimits}

\newcommand{\sst}[1]{\substack{#1}}

\DeclareMathOperator{\st}{\mathsf{st}}
\DeclareMathOperator{\cone}{\mathsf{cone}}
\DeclareMathOperator{\lk}{\mathsf{lk}}
\DeclareMathOperator{\join}{\mathsf{join}}
\title[Positive cluster complexes and $\tau$-tilting simplicial complexes]{Positive cluster complexes and $\tau$-tilting simplicial complexes of cluster-tilted algebras of finite type}
\author{Yasuaki Gyoda}
\address{Graduate School of Mathematical Sciences, The University of Tokyo, 3-8-1 Komaba Meguro-ku Tokyo 153-8914, Japan}
\email{gyoda-yasuaki@g.ecc.u-tokyo.ac.jp}
\begin{document}
\begin{abstract}
In this study, we consider the positive cluster complex, a full subcomplex of a cluster complex the vertices of which are all non-initial cluster variables.
In particular, we provide a formula for the difference in face vectors of positive cluster complexes caused by a mutation for finite type. 
Moreover, we explicitly describe specific positive cluster complexes of finite type and calculate their face vectors. We also provide a method to compute the face vector of an arbitrary positive cluster complex of finite type using these results.
Furthermore, we apply our results to the $\tau$-tilting theory of cluster-tilted algebras of finite representation type using the correspondence between clusters and support $\tau$-tilting modules.
\end{abstract}

\maketitle
\tableofcontents

\section{Introduction}
\emph{Cluster algebras} are a class of commutative algebras generated by some elements called \emph{cluster variables}, which are given by transformations called \emph{mutations}. The combinatorial structure of cluster variables and mutations is a subject of active research owing to its applications in various mathematical fields.
In this study, we consider cluster complexes. These simplicial complexes are combinatorial structures comprising clusters and mutations introduced by Fomin-Zelevinsky \cite{fzii}. More precisely, a cluster complex $\Delta(\xx,B)$ is a simplicial complex the vertex set of which comprises cluster variables and the simplicial set of which comprises subsets of clusters in a cluster algebra (Definition \ref{def:clustercomplex}). 
When a cluster algebra arises from a triangulation of a marked surface, the cluster complex is now known to coincide with the \emph{arc complex}, the maximal simplicial set of which consists of triangulations on the marked surface (see \cites{fst,ft}). Further, when a cluster algebra arises from a non-valued quiver, its cluster complex coincides with the \emph{support $\tau$-tilting simplicial complex}, the simplicial set of which comprises special modules of the cluster-tilted algebra of the quiver (see \cites{bmrrt,air}). Cluster algebras have been applied to hyperbolic geometry and representation theory of algebras, respectively, through these correspondences.

In this study, we focus on the positive cluster complex and the full subcomplex of the cluster complex, the vertices of which are non-initial cluster variables. This work is motivated by the representation theory of algebras. The positive cluster complex appears naturally as a \emph{$\tau$-tilting simplicial complex}, in which maximal simplices are $\tau$-tilting modules. When $\Lambda$ is hereditary, the $\tau$-tilting $\Lambda$-module coincides with the tilting $\Lambda$-module. The tilting module is a generalization of the progenerator in the Morita theory and an essential element in the tilting theory. A tilting module $T$ gives rise to an equivalence between a certain torsion class (resp. torsion-free class) of $\Lambda$ and a certain torsion-free class (resp. torsion class) of $(\End_{\Lambda}T)^{\mathrm{op}}$ (see \cite{brenbutl}). The tilting simplicial complex was studied by Riedtmann-Schofield and Unger \cites{resc, ung}. In particular, Riedtmann-Schofield introduced the mutation of tilting modules.

The primary purpose of this paper is to investigate the shape of the positive cluster complex of finite type and apply it to the $\tau$-tilting simplicial complex by using an isomorphism between the cluster and support $\tau$-tilting modules given by Adachi-Iyama-Reiten \cite{air}. 

We begin with the theory of cluster complexes, by investigating enumerative properties of face vectors of positive cluster complexes of finite type. The face vector is essential in understanding the structure of a simplicial complex.
Let $(\xx,B)$ be a cluster seed and $\Delta(\xx,B)$ a cluster complex associated with $(\xx,B)$. We denote by $\Delta^+(\xx,B)$ a positive cluster complex with respect to the initial seed $(\xx,B)$ (Definition \ref{def:positiveclustercomplex}). For $v\in\ZZ^n$, we denote by $[v]_k$ the vector the $i$th component of which is $(i-k)$th component of $v$ (and if $i\leq k$, it is $0$).
\begin{theorem}[Theorem \ref{thm:mutation-positivecomplex}]\label{main-1positive}
We provide the formula for the difference of the face vectors of positive cluster complexes $\Delta^+(\xx,B)$ and $\Delta^+(\xx',B')$ when $(\xx',B')$ is obtained from $(\xx,B)$ by a mutation. More precisely, if $x\in\xx$ is exchanged with $x'\in\xx'$ by the mutation, we have 
\begin{align}\label{eq:difference}
   f(\Delta^+(\xx,B))-f(\Delta^+(\xx',B'))
   =\left[f(\Delta^+(\xx'-\{x'\},B_{\backslash\{x'\}}))\right]_1-\left[f(\Delta^+(\xx-\{x\},B_{\backslash\{x\}}))\right]_1,
\end{align}
where $f(\Delta^+(\xx,B))$ is the face vector of $\Delta^+(\xx,B)$.
\end{theorem}
 Equation \eqref{eq:difference} implies that the difference between two face vectors of positive cluster complexes obtained by a mutation equals that of lower ranks. Moreover, by using Theorem \ref{main-1positive}, we provide a sufficient condition for the invariance of face vectors under a mutation.

\begin{corollary}[Corollary \ref{thm:invariant-sink/source}]\label{main-2positive}
If $B'$ is obtained from $B$ by a sink or source mutation, then the face vectors of $\Delta^+(\xx,B)$ and $\Delta^+(\xx',B')$ coincide.
\end{corollary}
Moreover, the following corollary follows immediately from Corollary \ref{main-2positive}.

\begin{corollary}[Corollary \ref{cor:independent-orientation}]\label{main-3positive}
Let $B$ and $B'$ be skew-symmetrizable matrices. If corresponding valued quivers $Q_B$ and $Q_{B'}$ have the same Dynkin diagram as the underlying graph, then the face vectors of $\Delta^+ (\xx,B)$ and $\Delta^+ (\xx',B')$ coincide.
\end{corollary}
 An ``initial condition'' is required along with the difference between two face vectors by a mutation to calculate the face vector of a positive cluster complex. Therefore, we need to know the face vectors of specific positive cluster complexes. For the classical finite type, we solve this problem by providing explicit descriptions of positive cluster complexes in a particular case.
\begin{theorem}[Theorems \ref{thm:positive-simplex-A_n}, \ref{thm:positive-simplex-B_n}, \ref{thm:positive-simplex-D_n}] \label{main-4positive}
We give an explicit description of a positive cluster complex $\Delta^+(\xx,B)$ when the corresponding valued quiver $Q_B$ of $B$ is a linearly oriented Dynkin quiver of type $A,B,C,D$. Specifically, if $Q_B$ is a linearly oriented $A_n$ type, $\Delta^+(\xx,B)$ is isomorphic to a cone of the (dual) associahedron with the rank $n-1$.
\end{theorem}

We calculate face vectors of positive cluster complexes of classical type from the explicit descriptions in Theorem \ref{main-4positive}.

\begin{theorem}[Theorem \ref{thm:count1}]\label{thm:intro-count1}
Suppose that the underlying graph of the valued quiver $Q_B$ is $A_n,B_n,C_n\text{ or }D_n$ type. Then, the $k$th entry of $f^+(\Delta(\xx,B))$ is given by the following list:
\begin{itemize}
    \item[($A_n$)] $\dfrac{1}{k+2}\begin{pmatrix}n\\ k+1\end{pmatrix}\begin{pmatrix}n+k+1\\ k+1\end{pmatrix}$,
    \vspace{3mm}
    \item[($B_n$)] $ \begin{pmatrix}n\\ k+1\end{pmatrix}\begin{pmatrix}n+k\\ k+1\end{pmatrix}$,
      \vspace{3mm}
    \item[($C_n$)] same as $B_n$ type,
      \vspace{3mm}
    \item[($D_n$)] $\begin{pmatrix}n\\ k+1\end{pmatrix}\begin{pmatrix}n+k-1\\ k+1\end{pmatrix}+\begin{pmatrix}n-1\\ k\end{pmatrix}\begin{pmatrix}n-k-2\\ k\end{pmatrix}-\dfrac{1}{n-1}\begin{pmatrix}n-1\\ k\end{pmatrix}\begin{pmatrix}n+k-1\\ k+1\end{pmatrix}$.
  \end{itemize}
\end{theorem}

For exceptional finite type, we calculate face vectors directly by using a calculator.

\begin{theorem}[Theorem \ref{thm:count2}]\label{thm:intro-count2}
Suppose that the underlying graph of the valued quiver $Q_B$ is $E_6,E_7,E_8,F_4,\text{ or }G_2$ type. Then, $f^+(\Delta(\xx,B))$ is given by the following list:
  \begin{itemize}
    \item[($E_6$)] $(1,36,300,1035,1720,1368,418)$,
    \item[($E_7$)] $(1,63,777,3927,9933,13299,9009,2431)$,
    \item[($E_8$)] $(1,120,2135,15120,54327,108360,121555,71760,17342)$,
    \item[($F_4$)] $(1, 24, 101, 144, 66)$,
    \item[($G_2$)] $(1, 6, 5)$.
  \end{itemize}
\end{theorem}

By Theorems \ref{thm:intro-count1}, \ref{thm:intro-count2} and Theorem \ref{main-1positive}, we can obtain the face vector of any positive cluster complex of finite type by summing the differences by mutations to the face vector of a computable such complex. We explain this point precisely in Section 5 with a concrete example. 

By applying these results to positive cluster complexes of skew-symmetric finite type and identifying positive cluster complexes with $\tau$-tilting simplicial complexes, we provide some results on $\tau$-tilting theory. First, we have the counterpart of Theorem \ref{main-1positive}.

\begin{theorem}[Corollary \ref{cor:mutation-tau-tilting-complex}]\label{main-1tilting}
  Let $\Lambda$ and $\Lambda'$ be cluster-tilted algebras such that $Q_{\Lambda'}$ is obtained from $Q_\Lambda$ by a quiver mutation. Then, we provide a formula for the difference of the numbers of basic $\tau$-rigid modules of $\Lambda$ and $\Lambda'$ with $k$ direct summands.
\end{theorem}
As the counterpart of Corollary \ref{main-2positive}, we have the following corollary.
\begin{corollary}[Corollary \ref{cor:invariant-sink/source-tilting}]\label{main-2tilting}
 Let $\Lambda$ and $\Lambda'$ be cluster-tilted algebras such that $Q_\Lambda'$ is obtained from $Q_\Lambda$ by a sink or source mutation. Then, $\Lambda$ and $\Lambda'$ have the same number of basic $\tau$-rigid modules with $k$ direct summands for any $0 \leq k\leq |Q_\Lambda|$.
\end{corollary}
The following corollary is a special case of Corollary \ref{main-2tilting}, which has been given in several works \cites{mrz,eno}.
\begin{corollary}[Corollary \ref{cor:independent-orientation-tilting}]\label{main-3tilting}
Let $Q$ be a simply laced Dynkin quiver. For any $0\leq k\leq |Q|$, the number of basic rigid $KQ$-modules with $k$ direct summands depends only on the underlying graph of $Q$.
\end{corollary}
 
Next, by applying Theorem \ref{main-4positive} to $A_n$ or $D_n$ type, we have the following.
\begin{theorem}[Corollaries \ref{cor:explicit-descriptiton-An}, \ref{cor:explicit-descriptiton-Dn}]
  We describe a $\tau$-tilting simplicial complex of $k Q$ when $Q$ is a linearly oriented quiver of Dynkin $A_n$ or $D_n$ type. Specifically, if $Q$ is linearly oriented Dynkin $A_n$ type, then the $\tau$-tilting simplicial complex of $KQ$ is isomorphic to a cone of the support $\tau$-tilting simplicial complex of $KQ'$, where $Q'$ is of the Dynkin $A_{n-1}$ type.
\end{theorem}
Moreover, by using the $h$-vector and the shellability of the $\tau$-tilting complex, we prove a theorem on the Hasse quiver $\mathsf{Hasse}(\tau\textrm{-}\mathrm{tilt}\Lambda)$ of a $\tau$-tilting module along with an associated corollary. For an algebra $\Lambda$, let $T(M)$ be the number of arrows that end at $M\in\mathsf{Hasse}(\tau\textrm{-}\mathrm{tilt}\Lambda)$.
Set 
\begin{align}\label{anotherreph-vector-intro}
    h_i(\Lambda)=\#\{M\in\tau\textrm{-}\mathrm{tilt}\Lambda\mid T(M)=i\}.
\end{align}
\begin{theorem}[Corollary 6.21] 
 Let $\Lambda$ and $\Lambda'$ be cluster-tilted algebras of finite representation type. If $Q_{\Lambda'}$ is obtained from $Q_{\Lambda}$ by a sink or source quiver mutation, then $h_i(\Lambda)=h_i(\Lambda')$ holds for any $i$.
\end{theorem}
\begin{corollary}[Corollary 6.22]
  Let $\Lambda$ and $\Lambda'$ be finite dimensional hereditary algebras. If $Q_\Lambda$ and $Q_{\Lambda'}$ have the same underlying graph, then $h_i(\Lambda) = h_i(\Lambda')$ holds for any $i$.
\end{corollary}

In Appendix \ref{anotherproof}, we provide another proof of the $A_n$ case of Theorem \ref{main-4positive} directly, without identifying a positive cluster complex with a $\tau$-tilting simplicial complex. 

\subsection*{Organization}
The remainder of this work is organized as follows. In Section 2, we describe the basics of cluster algebra. In Section 3, we discuss the properties of positive cluster complexes. In Section 4, we provide specific descriptions of some special positive cluster complexes of finite type. Building on these results, we provide a method for computing the face vectors of general finite type positive cluster complexes in Section 5. In Section 6, we apply the results of Sections 3, 4, and 5 to the representation theory of algebras.

\subsection*{Acknowledgments}
The author thanks Haruhisa Enomoto for the discussion on positive complexes. The author also thanks Osamu Iyama, Sota Asai, and Aaron Chan for their helpful comments at the Tokyo-Nagoya Algebra Seminar. Iyama also checked the paper and commented on its structure. The author would also like to thank Tomoki Nakanishi. This work was supported by JSPS KAKENHI Grant number JP20J12675.

\section{Cluster structures and their properties}
\subsection{Cluster patterns and cluster complexes}\label{section:cluster pattern}
We start by recalling the definitions of seed mutation and cluster patterns according to \cite{fziv}.
Let $n\in \ZZ_{\geq0}$ and $\Fcal$ be a rational function field of $n$ indeterminates.
A \emph{labeled seed} is a pair $(\mathbf{x},B)$, where
\begin{itemize}
\item $\mathbf{x}=(x_1, \dots, x_n)$ is an $n$-tuple of elements of $\mathcal F$ forming a free generating set of $\mathcal F$,
\item $B=(b_{ij})$ is a \emph{skew-symmetrizable} $n \times n$ integer matrix which is. Then, there exists a positive integer diagonal matrix $S$ such that $SB$ is skew-symmetric. Also, we call $S$ a \emph{skew-symmetrizer} of $B$.
\end{itemize}

We say that $\xx$ is a \emph{cluster}, and we refer to $x_i$ and $B$ as the \emph{cluster variables} and the \emph{exchange matrix}, respectively.

We adopt the notation $[b]_+=\max(b,0)$ for an integer $b$. 
Let $(\mathbf{x}, B)$ be a labeled seed, and let $k \in\{1,\dots, n\}$. The \emph{seed mutation $\mu_k$ in direction $k$} transforms $(\mathbf{x},B)$ into another labeled seed $\mu_k(\mathbf{x}, B)=(\mathbf{x'}, B')$ defined as follows.
\begin{itemize}
\item The entries of $B'=(b'_{ij})$ are given by
\begin{align} \label{eq:matrix-mutation}
b'_{ij}=\begin{cases}-b_{ij} &\text{if $i=k$ or $j=k$,} \\
b_{ij}+\left[ b_{ik}\right] _{+}b_{kj}+b_{ik}\left[ -b_{kj}\right]_+ &\text{otherwise.}
\end{cases}
\end{align}
\item The cluster variables $\mathbf{x'}=(x'_1, \dots, x'_n)$ are given by
\begin{align}\label{eq:x-mutation}
x'_j=\begin{cases}\dfrac{\mathop{\prod}\limits_{i=1}^{n} x_i^{[b_{ik}]_+}+\mathop{\prod}\limits_{i=1}^{n} x_i^{[-b_{ik}]_+}}{x_k} &\text{if $j=k$,}\\
x_j &\text{otherwise.}
\end{cases}
\end{align}
\end{itemize}
Specially, if $b_{jk}\geq0$ (resp. $b_{jk}\leq0$) for all $j$, then we say that $\mu_k$ is a \emph{sink mutation} (resp. \emph {source mutation}).
Let $\mathbb{T}_n$ be the \emph{$n$-regular tree} in which edges are labeled by the numbers $1, \dots, n$ such that the $n$ edges emanating from each vertex have different labels. We write
\begin{xy}(0,1)*+{t}="A",(10,1)*+{t'}="B",\ar@{-}^k"A";"B" \end{xy}
to indicate that vertices $t,t'\in \mathbb{T}_n$ are joined by an edge labeled by $k$. We fix an arbitrary vertex $t_0\in \TT_n$, which is referred to as the \emph{rooted vertex}.

A \emph{cluster pattern} is an assignment of a labeled seed $\Sigma_t=(\xx_t,B_t)$ to every vertex $t\in \mathbb{T}_n$ such that the labeled seeds $\Sigma_t$ and $\Sigma_{t'}$ assigned to the endpoints of any edge
\begin{xy}(0,1)*+{t}="A",(10,1)*+{t'}="B",\ar@{-}^k"A";"B" \end{xy}
are obtained from each other by the seed mutation in direction $k$. We denote by $P\colon t\mapsto \Sigma_t$ in this assignment. The elements of $\Sigma_t$ are denoted as follows:
\begin{align} \label{den:seed_at_t}
\mathbf{x}_t=(x_{1;t},\dots,x_{n;t}),\ B_t=(b_{ij;t}).
\end{align}
In particular, at $t_0$, we denote
\begin{align} \label{initialseed}
\mathbf{x}=\mathbf{x}_{t_0}=(x_1,\dots,x_n),\ B=B_{t_0}=(b_{ij}).
\end{align}

The degree $n$ of the regular tree $\TT_n$ is called the \emph{rank} of a cluster pattern $P$.
We also denote by $P(\xx,B)$ the cluster pattern with the initial seed $(\xx,B)$.
\begin{example}\label{A2}
An example of mutations in the case of $A_2$ is provided below.
Let $n = 2$. We consider a tree $\TT_2$ with edges labeled as follows:
\begin{align}\label{A2tree}
\begin{xy}
(-10,0)*+{\dots}="a",(0,0)*+{t_0}="A",(10,0)*+{t_1}="B",(20,0)*+{t_2}="C", (30,0)*+{t_3}="D",(40,0)*+{t_4}="E",(50,0)*+{t_5}="F", (60,0)*+{\dots}="f"
\ar@{-}^{1}"a";"A"
\ar@{-}^{2}"A";"B"
\ar@{-}^{1}"B";"C"
\ar@{-}^{2}"C";"D"
\ar@{-}^{1}"D";"E"
\ar@{-}^{2}"E";"F"
\ar@{-}^{1}"F";"f"
\end{xy}.
\end{align}
We set $B=\begin{bmatrix}
 0 & 1 \\
 -1 & 0
\end{bmatrix}
$ as the initial exchange matrix at $t_0$.
Then, coefficients and cluster variables are given as in Table \ref{A2seed} \cite{fziv}*{Example 2.10}.
\begin{table}[ht]
\begin{equation*}
\begin{array}{|c|c|cc|}
\hline
&&&\\[-4mm]
t& B_t && \xx_t \hspace{30mm}\\
\hline
&&&\\[-3mm]
0 &\begin{bmatrix}0&1\\-1&0\end{bmatrix}& x_1& x_2 \\[3mm]
\hline
&&&\\[-3mm]
1& \begin{bmatrix}0&-1\\1&0\end{bmatrix}& x_1& \dfrac{x_1+1}{x_2} \\[3mm]
\hline
&&&\\[-3mm]
2&\begin{bmatrix}0&1\\-1&0\end{bmatrix}&\dfrac{x_1 + x_2 + 1}{x_1x_2} & \dfrac{x_1+1}{x_2} \\[3mm]
\hline
&&&\\[-3mm]
3&\begin{bmatrix}0&-1\\1&0\end{bmatrix}& \dfrac{x_1+x_2+1}{x_1x_2} & \dfrac{x_2+1}{x_1} \\[3mm]
\hline
&&&\\[-2mm]
4&\begin{bmatrix}0&1\\-1&0\end{bmatrix}& x_2 & \dfrac{x_2+1}{x_1} \\[3mm]
\hline
&&&\\[-2mm]
5&\begin{bmatrix}0&-1\\1&0\end{bmatrix}& x_2 & x_1\\[3mm]
\hline
\end{array}
\end{equation*}
\caption{Exchange matrices and cluster variables in type~$A_2$\label{A2seed}}
\end{table}
\end{example}

To define the class of cluster patterns of finite type, we define the non-labeled seeds according to \cite{fziv}. For a cluster pattern $P$, we introduce the following equivalence relations of labeled seeds. We say that \begin{align*}
\Sigma_t=(\xx_t, B_t),\quad \xx_t=(x_{1;t,}\dots,x_{n;t}),\quad B_t=(b_{ij;t})
\end{align*}
and
\begin{align*}
\Sigma_s=(\xx_{s}, B_{s}),\quad \xx_s=(x_{1;s},\dots,x_{n;s}),\quad B_s=(b_{ij;s})
\end{align*}
are equivalent if there exists a
permutation~$\sigma$ of indices~$1, \dots, n$ such that
\begin{align}\label{eq:order-identification}
x_{i;s} = x_{\sigma(i);t}, \quad
b_{ij;s} = b_{\sigma(i), \sigma(j);t}
\end{align}
for all~$i$ and~$j$.
We denote by $[\Sigma]$ the equivalence classes represented by a labeled seed
$\Sigma$ and refer to this as the \emph{non-labeled seed}. We define a \emph{(non-labeled) clusters} $[\xx]$ as the set ignoring the order of a labeled cluster. Abusing notation, we abbreviate $[\xx]$ to $\xx$. We can define the \emph{mutation}, the \emph{sink mutation}, and the \emph{source mutation $\mu_x$ of a non-labeled seed in the direction of a cluster variable $x$} by using a mutation of a labeled seed.
\begin{remark}
By \cite{cl2}*{Proposition 3}, if we assume that the first equation in \eqref{eq:order-identification} holds, then we have the second equation in \eqref{eq:order-identification}. Therefore, non-labeled clusters can be regarded as clusters of non-labeled seeds. Furthermore, ``an exchange matrix $B$ associated with $\xx$'' is well-defined.
\end{remark}

\begin{definition}
The \emph{exchange graph} of a cluster pattern is the connected regular graph the vertices of which are the non-labeled seeds of the cluster pattern and the edges of which connect the non-labeled seeds related by a single mutation.
\end{definition}
Using the exchange graph, we define cluster patterns of finite type.
\begin{definition}
A cluster pattern $P$ is of \emph{finite type} if the exchange graph of $P$ is finite.
\end{definition}

\begin{remark}\label{rem:finite-charactrization}
If $P$ is of finite type, the initial non-labeled seed $[(\xx,B)]$ is  mutation equivalent to $[(\xx',B')]$ such that $A(B')$ is a finite Cartan matrix (\cite{fzii}*{Theorem 1.4}), where $A(B')=(a_{ij})$ is defined as
\begin{align*}
a_{ij}=\begin{cases}
2\quad &\text{if } i=j;\\
-|b'_{ij}| & \text{if } i\neq j.
\end{cases}
\end{align*}
We call $B'$ an exchange matrix \emph{of tree finite type}.
\end{remark}
Let us introduce cluster complexes as defined in \cite{fzii}.
\begin{definition}\label{def:clustercomplex}
Let $P(\xx,B)$ be a cluster pattern. We define a \emph{cluster complex} $\Delta(\xx,B)$ as a simplicial complex the vertex set of which consists of cluster variables and in which simplices are subsets of each cluster.
\end{definition}
If $P(\xx,B)$ is of finite type, then we say that the cluster complex $\Delta(\xx,B)$ is of \emph{finite type}. A cluster complex of finite type has finitely many simplices. We say that $x$ is \emph{compatible} with $x'$ if $\{x,x'\}$ is a simplex of $\Delta(\xx,B)$.
\begin{example}\label{clustercomplexofA2}
We consider the cluster pattern in Example \ref{A2} and provide a cluster complex corresponding to this cluster pattern in Figure \ref{A2complex}.

\begin{figure}[ht]
\caption{Cluster complex of $A_2$ type}\label{A2complex}
\[
\begin{tikzpicture}
 \coordinate (0) at (0,0);
 \coordinate (1) at (30:2);
 \coordinate (2) at (90:2);
 \coordinate (3) at (150:2);
 \coordinate (4) at (270:2);
 \coordinate (5) at (330:2);
 \node (6) at (30:3) {$\dfrac{x_1 + x_2+1}{x_1x_2}$};
 \node (7) at (90:2.7){$\dfrac{x_2+1}{x_1}$};
 \node (8) at (150:2.5){$x_2$};
 \node (9) at (270:2.5){$x_1$};
 \node (10) at (330:3){$\dfrac{x_1+1}{x_2}$};
 \fill(1) circle (0.7mm);
 \fill(2) circle (0.7mm);
 \fill(3) circle (0.7mm);
 \fill(4) circle (0.7mm);
 \fill(5) circle (0.7mm);
 \draw(1) to (2);
 \draw(2) to (3);
 \draw(3) to (4);
 \draw(4) to (5);
 \draw(5) to (1);
\end{tikzpicture}
\]
\end{figure}
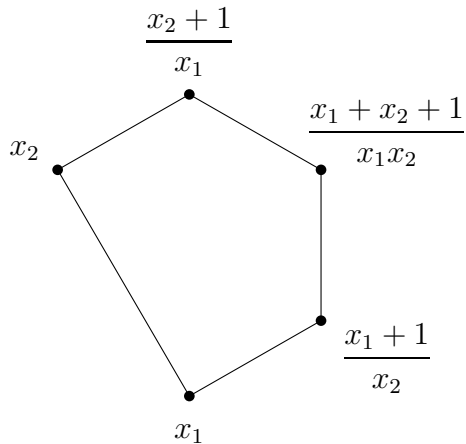
\end{example}
We note that, if $[(\xx,B)]$ is mutation equivalent to $[(\xx',B')]$ as non-labeled seeds, then we have $\Delta(\xx,B)\cong\Delta(\xx',B')$. Hence, we can define the \emph{cluster complex of $X_n$ type} as a cluster complex $\Delta(\xx,B)$, where $B$ is mutation equivalent to an exchange matrix of tree $X_n$ type under an appropriate permutation of indices. We denote by $\Delta(X_n)$ the cluster complex of $X_n$ type.
\subsection{Quiver mutations}
A cluster pattern can be constructed using a quiver $Q$ instead of an exchange matrix $B$; we explain this construction along the lines of \cite{kel10} and \cite{kel12}. First, we consider the case in which $B$ is skew-symmetric.
For an $n\times n$ skew-symmetric integer matrix $B = (b_{ij})$, we define a quiver $Q_B$ as follows.

\begin{itemize}
    \item The cardinality of the set of vertices for $ Q_B $ is $ n $, and these vertices are numbered from $ 1 $ to $ n $.
    \item The number of arrows from $ i $ to $ j $ in $ Q_B $ is $ b_ {ij} $. If negative, we assume that there are $ -b_ {ij} $ arrows from $ j $ to $ i $.
\end{itemize}
Let $Q'=\mu_k(Q)$ be a quiver obtained from $Q$ by the following steps.

\begin{itemize}
  \item[(1)] For any path $i \rightarrow k \rightarrow j$, add an arrow $i \rightarrow j$,
  \item[(2)] reverse all arrows incident to $k$,
  \item[(3)] remove a maximal set of disjoint $2$-cycles.
 \end{itemize}
We refer to this transformation from $Q$ to $Q'$ as the \emph{quiver mutation in direction $k$}.
Then, we have $Q_{\mu_k(B)} = \mu_k(Q_B)$. Therefore, quivers and quiver mutations can be used rather than of skew-symmetric exchange matrices and matrix mutations.
\begin{example}
 A quiver $Q_B$ corresponding to $B=\begin{bmatrix}
 0&-1&-2\\
 1&0&4\\
 2&-4&0
 \end{bmatrix}$ is
 
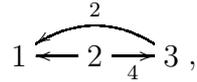
\[
 \begin{xy}(0,0)*+{1}="I",(10,0)*+{2}="J", (20,0)*+{{3}}="K" \ar@{->}"J";"I" \ar@{->}_4"J";"K" \ar@/_4mm/_2"K";"I" \end{xy},
 \]
 where the numbers on the arrows represent the number of arrows; values of 1 arrow are omitted for simplicity. Then, we obtain the following quiver from $Q_B$ by mutating in direction $3$:
 
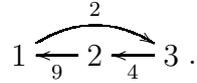
\[
 \begin{xy}(0,0)*+{1}="I",(10,0)*+{2}="J", (20,0)*+{{3}}="K" \ar@{->}^9"J";"I" \ar@{<-}_4"J";"K" \ar@/^4mm/^2"I";"K" \end{xy}.
 \]
\end{example}

Next, we generalize quivers and quiver mutations to a form that can be used even when the corresponding matrix $B$ is skew-symmetrizable. Let $Q_0$ be the set of vertices of $Q$ and $Q_1$ the set of arrows of $Q$. For an $n\times n$ skew-symmetrizable integer matrix $B$, we define a \emph{valued quiver} $Q_B = (Q_B,v,d)$ as follows.
\begin{itemize}
    \item The cardinality of the set of vertices for $ Q_B $ is $ n $, and these vertices are numbered from $ 1 $ to $ n $.
    \item The quiver $Q_B$ has an arrow from $i$ to $j$ if $b_{ij}> 0$.
    \item A function $v\colon (Q_B)_1\to \ZZ^2$ transfers $(\alpha\colon i\to j) \mapsto
    (b_{ij},-b_{ji})$.
    \item A function $d\colon (Q_B)_0\to \ZZ_{> 0}$ satisfies the following condition: for each arrow $\alpha\colon i\to j$, we have
    \[d(i)v(\alpha)_1=v(\alpha)_2d(j),\]
    where $v(\alpha) = (v(\alpha)_1, v(\alpha)_2)$.
\end{itemize}
We define $\mu_k(Q_{B}):=Q_{\mu_k(B)}$ and refer to this construction as the \emph{valued quiver mutation in direction $k$}.
\begin{example}
  A valued quiver $Q_B$ corresponding to $B=\begin{bmatrix}
 0&-2&-4\\
 1&0&4\\
 2&-4&0
 \end{bmatrix}$ is
 
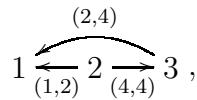
\[
 \begin{xy}(0,0)*+{1}="I",(10,0)*+{2}="J", (20,0)*+{{3}}="K" \ar@{->}^{(1,2)}"J";"I" \ar@{->}_{(4,4)}"J";"K" \ar@/_4mm/_{(2,4)}"K";"I" \end{xy},
 \]
 where the pair of integers on the arrow $\alpha$ is $v(\alpha)$, and $d(1)=1,d(2)=2,d(3)=2$. Then, we obtain the following quiver from $Q_B$ by mutating in direction $3$.
 \[
 \begin{xy}(0,0)*+{1}="I",(10,0)*+{2}="J", (20,0)*+{{3}}="K" \ar@{->}^{(9,18)}"J";"I" \ar@{<-}_{(4,4)}"J";"K" \ar@/^4mm/^{(4,2)}"I";"K" \end{xy}.
 \]
\end{example}
By definition, $d(i)$ of $Q_B$ corresponds with the $(i,i)$ entry of the skew-symmetrizer $S$ of $B$. Moreover, by Remark \ref{rem:finite-charactrization}, if a cluster pattern $P(\xx,B)$ is of finite type, then $Q_{B}$ is mutation equivalent to a valued quiver in which the underlying graph is that shown in Figure \ref{fig:valued-quiver} (we omit labels (1,1)).
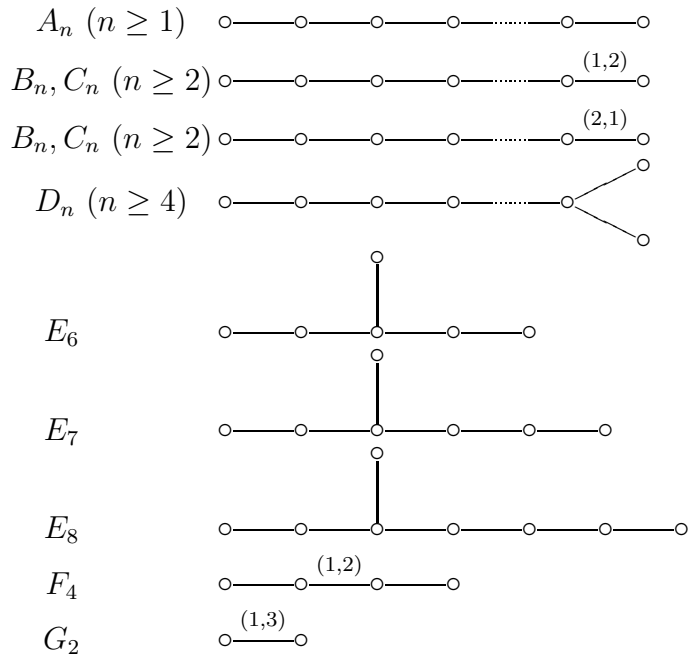
\begin{figure}[ht]
    \caption{Underlying graph of finite type}
    \label{fig:valued-quiver}
\begin{center}
\begin{xy}
    (0,0)*{},(35,-5)*{A_n\ (n \geq 1)},(50,-5)*{\circ}="A",(60,-5)*{\circ}="B",(70,-5)*{\circ}="C",(80,-5)*{\circ}="D",(85,-5)*{}="E",(90,-5)*{}="F",(95,-5)*{\circ}="G",(105,-5)*{\circ}="H"\ar@{-} "A";"B", \ar@{-} "B";"C", \ar@{-} "C";"D", \ar@{-} "D";"E", \ar@{.} "E";"F",\ar@{-} "F";"G",\ar@{-} "G";"H"
\end{xy}
\begin{xy}
    (0,0)*{},(35,-5)*{B_n,C_n\ (n \geq 2)},(50,-5)*{\circ}="A",(60,-5)*{\circ}="B",(70,-5)*{\circ}="C",(80,-5)*{\circ}="D",(85,-5)*{}="E",(90,-5)*{}="F",(95,-5)*{\circ}="G",(105,-5)*{\circ}="H", \ar@{-} "A";"B", \ar@{-} "B";"C", \ar@{-} "C";"D", \ar@{-} "D";"E", \ar@{.} "E";"F",\ar@{-} "F";"G",\ar@{-}^{(1,2)} "G";"H"
\end{xy}
\begin{xy}
    (0,0)*{},(35,-5)*{B_n,C_n\ (n \geq 2)},(50,-5)*{\circ}="A",(60,-5)*{\circ}="B",(70,-5)*{\circ}="C",(80,-5)*{\circ}="D",(85,-5)*{}="E",(90,-5)*{}="F",(95,-5)*{\circ}="G",(105,-5)*{\circ}="H", \ar@{-} "A";"B", \ar@{-} "B";"C", \ar@{-} "C";"D", \ar@{-} "D";"E", \ar@{.} "E";"F",\ar@{-} "F";"G",\ar@{-}^{(2,1)} "G";"H"
\end{xy}
\begin{xy}
    (0,0)*{},(35,-5)*{D_n\ (n \geq 4)},(50,-5)*{\circ}="A",(60,-5)*{\circ}="B",(70,-5)*{\circ}="C",(80,-5)*{\circ}="D",(85,-5)*{}="E",(90,-5)*{}="F",(95,-5)*{\circ}="G",(105,0)*{\circ}="H",(105,-10)*{\circ}="I",\ar@{-} "A";"B", \ar@{-} "B";"C", \ar@{-} "C";"D", \ar@{-} "D";"E", \ar@{.} "E";"F",\ar@{-} "F";"G",\ar@{-} "G";"H",\ar@{-} "G";"I"
\end{xy}
\begin{xy}
    (0,0)*{},(28.5,-5)*{E_6},(50,-5)*{\circ}="A",(60,-5)*{\circ}="B",(70,-5)*{\circ}="C",(80,-5)*{\circ}="D",(90,-5)*{\circ}="E",(70,5)*{\circ}="F"\ar@{-} "A";"B", \ar@{-} "B";"C", \ar@{-} "C";"D", \ar@{-} "D";"E",\ar@{-} "C";"F"
\end{xy}
\begin{xy}
    (0,0)*{},(28.5,-5)*{E_7},(50,-5)*{\circ}="A",(60,-5)*{\circ}="B",(70,-5)*{\circ}="C",(80,-5)*{\circ}="D",(90,-5)*{\circ}="E",(70,5)*{\circ}="F",(100,-5)*{\circ}="G"\ar@{-} "A";"B", \ar@{-} "B";"C", \ar@{-} "C";"D", \ar@{-} "D";"E",\ar@{-} "E";"G",\ar@{-} "C";"F"
\end{xy}
\begin{xy}
    (0,0)*{},(28.5,-5)*{E_8},(50,-5)*{\circ}="A",(60,-5)*{\circ}="B",(70,-5)*{\circ}="C",(80,-5)*{\circ}="D",(90,-5)*{\circ}="E",(70,5)*{\circ}="F",(100,-5)*{\circ}="G",(110,-5)*{\circ}="H"\ar@{-} "A";"B", \ar@{-} "B";"C", \ar@{-} "C";"D", \ar@{-} "D";"E",\ar@{-} "E";"G",\ar@{-} "G";"H",\ar@{-} "C";"F"
\end{xy}
\begin{xy}
    (0,0)*{},(28.5,-5)*{F_4},(50,-5)*{\circ}="A",(60,-5)*{\circ}="B",(70,-5)*{\circ}="C",(80,-5)*{\circ}="D"\ar@{-} "A";"B", \ar@{-}^{(1,2)} "B";"C", \ar@{-} "C";"D"
\end{xy}
\begin{xy}
    (0,0)*{},(28.5,-5)*{G_2},(50,-5)*{\circ}="A",(60,-5)*{\circ}="B",\ar@{-}^{(1,3)} "A";"B"
\end{xy}
\end{center}
\end{figure}
We note that a sink (resp. source) mutation, which is defined in Subsection \ref{section:cluster pattern}, is a mutation in direction $i$, which is a sink (resp. source) vertex in $Q_B$.
\subsection{Basic properties of cluster complexes}
Here, we introduce some basic and useful properties of cluster complexes. 

First, we recall the definition of closed stars and links and cones.
\begin{definition}\label{def:cone-star-link}
  Let $K$ be a simplicial complex and $v$ a vertex of $K$.
  \begin{enumerate}
    \item The \emph{closed star of $v$} is a (not necessarily full) subcomplex $\st_K(v)$ of $K$ the simplices of which are all the simplices $\sigma$ of $K$ such that $\sigma \cup \{v\}$ is a simplex of $K$.
    \item The \emph{link of $v$} is a (not necessarily full) subcomplex $\lk_K(v)$ of $K$ defined by $\lk_K(v) = \{\sigma\in \st_K(v)\mid v\notin \sigma \}$; that is, simplices of $\lk_K(v)$ are all the simplices $\sigma$ of $K$ with $v \not \in \sigma$ such that $\sigma \cup \{v\}$ is a simplex of $K$.
    \item Let $K$ be a simplicial complex and $v$ a vertex of $K$. Let $*$ be a vertex disjoint from $K$. The \emph{cone of $K$} is a simplicial complex $\cone(K)$ in which simplices are $\sigma$ and $\sigma \cup \{ * \}$ for each simplex $\sigma$ of $K$.
  \end{enumerate}
\end{definition}
The following property can be verified straightforwardly from Definition \ref{def:cone-star-link}.
\begin{lemma}\label{lem.st-lk-relation}
  Let $K$ be a simplicial complex and $v$ a vertex of $K$. Then, $\st_K(v)$ is isomorphic to $\cone(\lk_K(v))$ as simplicial complexes.
\end{lemma}
The following property is fundamental:
\begin{lemma}\label{lem:fullsub}
Let $\Delta(\xx,B)$ be a cluster complex and $v$ a vertex of $\Delta(\xx,B)$. Then, $\st_{\Delta(\xx,B)}(v)$ and $\lk_{\Delta(\xx,B)}(v)$ are full subcomplexes of $\Delta(\xx,B)$.
\end{lemma}
This lemma is proven using the following lemma.
\begin{lemma}[\cite{cl2}*{Corollary 5}]\label{lem:existance-cluster}
Let $X$ be a set of cluster variables satisfying the following condition.
\begin{align}\label{assumption-compatible}
\text{For any pair $x,x'\in X$, there exists a cluster $\xx$ such that $x,x'\in\xx$.}
\end{align}
Then, there exists a cluster such that it contains $X$.
\end{lemma}
\begin{proof}[Proof of Lemma \ref{lem:fullsub}]
Here, we prove that $\st_{\Delta(\xx,B)}(v)$ is a full subcomplex of $\Delta(\xx,B)$. It suffices to show that if $X$ and $X'$ are simplices of $\st_{\Delta(\xx,B)}(v)$, and $X\cup X'$ is a simplex of $\Delta(\xx,B)$, then $X\cup X'$ is a simplex of $\st_{\Delta(\xx,B)}(v)$. By definition of the closed star, $X\cup X'\cup \{v\}$ satisfies \eqref{assumption-compatible}. Therefore, by Lemma \ref{lem:existance-cluster}, $X\cup X'\cup \{v\}$ is a simplex of $\Delta(\xx,B)$ and $X\cup X'$ is a simplex of $\st_{\Delta(\xx,B)}(v)$. Similarly, we can prove the fullness of $\lk_{\Delta(\xx,B)}(v)$.
\end{proof}
Now, we generalize the definition of the cone.
\begin{definition}\label{def:set-of-join}
Let $K,L$ be simplicial complexes with a disjoint vertex set. The \emph{join of $K$ and $L$} is a simplicial complex $\join(K,L)$ the vertex set of which his $K\cup L$ and the simplices of which are all disjoint unions of a simplex of $K$ and that of $L$.
\end{definition}
We note that $\join(K,L,M)$ is well-defined by the associativity of joining.
We remark that, if $B=B_1\oplus B_2$ (i.e., $Q_B=Q_{B_1}\cup Q_{B_2}$), then we have
\begin{align}
\Delta(\xx,B)\cong \join (\Delta(\xx_1,B_1),\Delta(\xx_2,B_2)),
\end{align}
where $\xx_1,\xx_2$ are subsets of $\xx$ corresponding to $B_1,B_2$, respectively.
\section{Positive cluster complexes}
\subsection{Positive cluster complexes and basic property}
In this section, we define the special subcomplex of a cluster complex called the \emph{positive cluster complex}.
\begin{definition}\label{def:positiveclustercomplex}
Let $\Delta(\xx,B)$ be a cluster complex. We define a \emph{positive cluster complex} $\Delta^+(\xx,B)$ as a full subcomplex of $\Delta(\xx,B)$ the vertex set of which consists of all non-initial cluster variables. 
\end{definition}
The term ``positive'' comes from the positive roots of a finite root system. If $B$ is a bipartite matrix of tree finite type, then there is a canonical bijection between $\Delta^+(\xx,B)$ and simplices consisting of positive roots of a  \emph{generalized associahedron} \cite{fzii}*{Theorem 1.9}. In fact, \cites{arm,fkt} use the term the ``positive cluster complex'' in this sense. In this paper, we use this term more generally, even if $B$ is not bipartite or not of tree finite type.
By the definition of the positive cluster complex, if $B=B_1\oplus B_2$, then we have
\begin{align}
\Delta^+(\xx,B)\cong \join (\Delta^+(\xx_1,B_1),\Delta^+(\xx_2,B_2)).\label{eq:positive-join}
\end{align}
\begin{example}
In Example \ref{clustercomplexofA2}, the set of positive simplices associated with the initial cluster $\xx$ is given as follows.
\begin{align*}
&\left\{\emptyset,\left\{\dfrac{x_2+1}{x_1}\right\},\left\{\dfrac{x_1+x_2+1}{x_1x_2}\right\}, \left\{\dfrac{x_1+1}{x_2}\right\},\right.\\
&\ \left.\left\{\dfrac{x_2+1}{x_1}, \dfrac{x_1+x_2+1}{x_1x_2}\right\},
\left\{\dfrac{x_1+1}{x_2},\dfrac{x_1+x_2+1}{x_1x_2}\right\}\right\}
\end{align*}
\end{example}
We recall the \emph{face vector}\footnote{In cluster algebra theory, the term ``$f$-vector'' is used to refer to the exponential vector of the principal term of a $F$-polynomial. Therefore, we use this name throughout to avoid confusion.}, $f(\Delta) = (f_{-1},f_0,\dots,f_n)$. This vector is associated with a simplicial complex $\Delta$, and $f_i$ represents the number of $i$-dimensional simplices. We also consider the face vector associated with a set of simplices (necessarily not a simplicial complex) herein. In this case, since the empty set may not be included in a set of simplices, $f_{-1}$ is not always equal to $1$. This vector can be defined when a simplicial complex or a set of simplices is finite.

We remark that if
\begin{align}
\Delta\cong \join (\Delta',\Delta''),
\end{align}
$f(\Delta')=(1,a_0,a_1,\dots,a_{n})$ and $f(\Delta'')=(1,b_0,b_1,\dots,b_{m})$, then we have
\begin{align}
f(\Delta)=\left(1,\sum_{i+j=-1}a_ib_j,\sum_{i+j=0}a_ib_j,\dots,\sum_{i+j=m+n}a_ib_j\right)\label{eq:join-facevector},
\end{align}
where $-1\leq i\leq n$, $-1\leq j\leq m$, and $a_{-1}=b_{-1}=1$.

A proof of a fundamental property of positive cluster complexes is given below.

\begin{proposition}\label{pure}
Positive cluster complexes are pure.
\end{proposition}

This property is proven by the following lemma:

\begin{lemma}\label{initial-mutate}
Let $(\xx_t,B_t)$ be a seed in which cluster contains the initial cluster variable $x_i$. By a seed mutation of $(\xx_t,B_t)$ into a direction corresponding to $x_i$, $x_i$ is exchanged with a non-initial cluster variable.
\end{lemma}

\begin{proof}
We use the $d$-compatibility degree or $f$-compatibility degree, which are introduced in \cite{cl2} and \cite{fugy}, respectively. As there is no significant difference in the proofs between the two, we prove the statement using the $f$-compatibility degree. Let $x$ be a cluster variable exchanged with $x_i$ by the mutation in the statement. We assume $x$ is an initial cluster variable. Then, since both $x_i$ and $x$ are contained in the initial cluster, the compatibility degree $(x_i\parallel x)=0$ by \cite{fugy}*{Theorem 4.18}. However, since $x_i$ is exchanged with $x$ by a mutation, we have $(x_i\parallel x)=1$ by \cite{fugy}*{Theorem 4.22}, which is a contradiction.
\end{proof}

\begin{proof}[Proof of Proposition \ref{pure}]
It suffices to show that for any subset $X$ consisting of non-initial cluster variables of a cluster $\xx_t$, there exists a cluster $\xx_s$ such that $X\subset\xx_s$ and $\xx_s$ do not contain initial cluster variables. Lemma \ref{initial-mutate} shows that $\xx_s$ satisfying this condition can be obtained by replacing all the initial cluster variables in $\xx_t$ with mutations.   
\end{proof}

\subsection{Difference of face vectors by mutation}

 We use the following notation to describe the main theorem of this section. Let $(\xx,B)$ be a seed, and $X$ a subset of $\xx$. Then, we denote by $B_{\setminus X}$ the submatrix of $B$ obtained by removing all entries in positions corresponding to $X$ from $B$.
The main theorem of this subsection is the following one.
  \begin{theorem}\label{thm:mutation-positivecomplex}
Let $\Delta (\xx,B)$ be a cluster complex of finite type. We assume that $[(\xx',B')]$ is obtained from $[(\xx,B)]$ by a mutation in direction $x$ (and $x\mapsto x'$). Then, we have
\begin{align}
   f(\Delta^+(\xx,B))-f(\Delta^+(\xx',B'))
   =\left[f(\Delta^+(\xx'-\{x'\},B_{\backslash\{x'\}}))\right]_1-\left[f(\Delta^+(\xx-\{x\},B_{\backslash\{x\}}))\right]_1.
\end{align}
 \end{theorem}

This theorem indicates an amount of increase of positive faces by a mutation.

The complement of the set of simplices of a positive cluster complex is the set of simplices such that each simplex contains at least one element in the initial cluster. We cite some useful lemmas below.

\begin{lemma}[\cite{fzy}*{Proposition 3.5.3}]\label{lem:bijection-between-B-and-Bx}
For any cluster complex $\Delta(\xx,B)$ of finite type and $x\in \xx$, $\st_{\Delta(\xx,B)}(x)$ is isomorphic to a cone of $\Delta(\xx-\{x\}, B_{\setminus \{x\}})$. In particular, we have a bijection between the set of simplices of $\Delta(\xx,B)$ including $\{x\}$ and that of $\Delta(\xx-\{x\}, B_{\setminus \{x\}})$. Furthermore, if the face vector of the former can be defined and is $f$, then that of the latter is $[f]_{-1}$. \end{lemma}

Furthermore, Lemma \ref{lem:bijection-between-B-and-Bx} can be generalized as follows.

\begin{lemma}[\cite{fzy}*{Proposition 3.6}]\label{cor:generalized-bijection-between-B-and-Bx}
For any cluster complex $\Delta(\xx,B)$ of finite type and $X\subset\xx$ such that $|X|=k$, we have a bijection between the set of simplices of $\Delta(\xx,B)$ including $X$ and that of $\Delta(\xx-X, B_{\setminus X})$. Furthermore, if the face vector of the former can be defined and is $f$, then that of the latter is $[f]_{-k}$.
\end{lemma}

This allows us to prove Theorem \ref{thm:mutation-positivecomplex}.

\begin{proof}[Proof of Theorem \ref{thm:mutation-positivecomplex}]
We may assume that $\xx=X\sqcup{x}, \xx'=X\sqcup{x'}$, and $X=\{x_1,\dots,x_{n-1}\}$. For any set $Y\subset\xx$, let $\widehat{\Delta_Y}$ be the set of simplices of $\Delta(\xx,B)$ such that each simplex contains at least one element in $Y$. Moreover, for any subset $Y'$ of $Y$, let $\underline{\Delta_{Y'\subset Y}}$ be the set of simplices of $\Delta(\xx,B)$ such that each simplex contains all elements in $Y'$ and does not contain any elements in $Y\setminus Y'$. Since $(\xx,B)$ is mutation equivalent to $(\xx',B')$, we have $\Delta(\xx,B)\cong\Delta(\xx',B')$. Since $\Delta^+(\xx,B)=\Delta(\xx,B)-\widehat{\Delta_\xx}$ and $\Delta^+(\xx',B')=\Delta(\xx',B')-\widehat{\Delta_{\xx'}}$ and
\begin{align}
\widehat{\Delta_\xx}&=\widehat{\Delta_X}\sqcup \underline{\Delta_{\{x\}\subset \xx}},\label{eq:Deltax}\\
\widehat{\Delta_{\xx'}}&=\widehat{\Delta_X}\sqcup \underline{\Delta_{\{x'\}\subset \xx'}},\label{eq:Deltax'}
\end{align}
the difference of face vectors of $\Delta^+(\xx,B)$ and $\Delta^+(\xx',B')$ is
\begin{align*}
f(\underline{\Delta_{\{x'\}\subset \xx'}})-f(\underline{\Delta_{\{x\}\subset \xx}}).
\end{align*}
Therefore, by Lemma \ref{lem:bijection-between-B-and-Bx}, we have a bijection between simplices of $\underline{\Delta_{\{x\}\subset \xx}}$ and that of $ \cone(\Delta^+(\xx-\{x\},B_{\{x\}}))$. This suffices to complete the proof.
\end{proof}
By Theorem \ref{thm:mutation-positivecomplex}, we have the following corollary.
 \begin{corollary}\label{thm:invariant-sink/source}
Let $\Delta (\xx,B)$ be a cluster complex of finite type. If $[\Sigma_{t'}]$ is obtained from $[\Sigma_{t}]$ by a sink or source mutation, then for any $-1\leq k\leq n-1$, $\Delta^+ (\xx_t,B_t)$ and $\Delta^+ (\xx_{t'},B_{t'})$ have the same number of $k$-simplices.
 \end{corollary}
We also have the following corollary.
\begin{corollary}\label{cor:independent-orientation}
Let $B$ and $B'$ be skew-symmetrizable matrices. If $Q_B$ and $Q_{B'}$ have the same underlying graph in Figure \ref{fig:valued-quiver}, then, for any $-1\leq k\leq n-1$, $\Delta^+ (\xx_t,B_t)$ and $\Delta^+ (\xx_{t'},B_{t'})$ have the same number of $k$-simplices.
\end{corollary}
\begin{proof}
We can obtain $B'$ from $B$ by a sequence of sink or source mutations and a permutation of indices (\cite{fzii}*{Proposition 9.2}). Therefore, we have the statement using Corollary \ref{thm:invariant-sink/source}.
\end{proof}
\begin{example}\label{ex:not-iso-positive-complex}
We set
\[B=\begin{bmatrix}
0&-1&1\\
1&0&0\\
-1&0&0
\end{bmatrix},\quad Q_B= \begin{xy}(0,0)*+{1}="I",(10,0)*+{2}="J", (20,0)*+{{3}}="K" \ar@{->}"I";"J"  \ar@/_4mm/"K";"I" \end{xy},
\]
and consider a cluster pattern with the  initial seed$(\xx=(x_1,x_2,x_3),Q_B)$. This is of $A_3$ type. Figure \ref{hexagoncomplex} provides a cluster complex of this cluster pattern.
\begin{figure}[ht]
\caption{Cluster complex of $A_3$ type\label{hexagoncomplex}}
\vspace{5pt}
\begin{center}
\scalebox{0.8}{
\begin{tikzpicture}
\coordinate (0) at (0,0);
\coordinate (u*) at (90:5.33);
\coordinate (u) at (90:4);
\coordinate (ul) at (150:4);
\coordinate (ur) at (30:4);
\coordinate (uml) at (150:2);
\coordinate (umr) at (30:2);
\coordinate (dmc) at (-90:2);
\coordinate (dl) at (-150:4);
 \coordinate (dl*) at (-150:5.33);
\coordinate (dr) at (-30:4);
\coordinate (dr*) at (-30:5.33);
\coordinate (d) at (-90:4);
\draw (u) to (ul);
\draw (u) to (ur);
\draw (ul) to (uml);
\draw (ur) to (umr);
\draw (umr) to (uml);
\draw (u) to (uml);
\draw (u) to (umr);
\draw (uml) to (dl);
\draw (umr) to (dr);
\draw (ul) to (dl);
\draw (ur) to (dr);
\draw (uml) to (dmc);
\draw (umr) to (dmc);
\draw (dl) to (dmc);
\draw (dr) to (dmc);
\draw (dl) to (d);
\draw (dmc) to (d);
\draw (dr) to (d);
\draw(ul) [out=90,in=180]to (u*);
\draw(ur) [out=90,in=0]to (u*);
\draw(ul) [out=210,in=120]to (dl*);
\draw(d) [out=210,in=300]to (dl*);
\draw(d) [out=330,in=240]to (dr*);
\draw(ur) [out=330,in=60]to (dr*);
\fill [white](u) circle (0.9cm);
\fill[white] (ul) circle (0.9cm);
\fill[white] (ur) circle (0.9cm);
\fill[white] (uml) circle (0.9cm);
\fill[white] (umr) circle (0.9cm);
\fill[white] (dmc) circle (0.9cm);
\fill[white] (dl) circle (0.9cm);
\fill[white] (dr) circle (0.9cm);
\fill[white] (d) circle (0.9cm);
\draw (u)++(210:0.7cm) to ++(90:0.7cm);
\draw (u)++(150:0.7cm) to ++(30:0.7cm);
\draw (u)++(90:0.7cm) to ++(330:0.7cm);
\draw (u)++(30:0.7cm) to ++(270:0.7cm);
\draw (u)++(330:0.7cm) to ++(210:0.7cm);
\draw (u)++(270:0.7cm) to ++(150:0.7cm);
\draw (ur)++(210:0.7cm) to ++(90:0.7cm);
\draw (ur)++(150:0.7cm) to ++(30:0.7cm);
\draw (ur)++(90:0.7cm) to ++(330:0.7cm);
\draw (ur)++(30:0.7cm) to ++(270:0.7cm);
\draw (ur)++(330:0.7cm) to ++(210:0.7cm);
\draw (ur)++(270:0.7cm) to ++(150:0.7cm);
\draw (ul)++(210:0.7cm) to ++(90:0.7cm);
\draw (ul)++(150:0.7cm) to ++(30:0.7cm);
\draw (ul)++(90:0.7cm) to ++(330:0.7cm);
\draw (ul)++(30:0.7cm) to ++(270:0.7cm);
\draw (ul)++(330:0.7cm) to ++(210:0.7cm);
\draw (ul)++(270:0.7cm) to ++(150:0.7cm);
\draw (umr)++(210:0.7cm) to ++(90:0.7cm);
\draw (umr)++(150:0.7cm) to ++(30:0.7cm);
\draw (umr)++(90:0.7cm) to ++(330:0.7cm);
\draw (umr)++(30:0.7cm) to ++(270:0.7cm);
\draw (umr)++(330:0.7cm) to ++(210:0.7cm);
\draw (umr)++(270:0.7cm) to ++(150:0.7cm);
\draw (uml)++(210:0.7cm) to ++(90:0.7cm);
\draw (uml)++(150:0.7cm) to ++(30:0.7cm);
\draw (uml)++(90:0.7cm) to ++(330:0.7cm);
\draw (uml)++(30:0.7cm) to ++(270:0.7cm);
\draw (uml)++(330:0.7cm) to ++(210:0.7cm);
\draw (uml)++(270:0.7cm) to ++(150:0.7cm);
\draw (dl)++(210:0.7cm) to ++(90:0.7cm);
\draw (dl)++(150:0.7cm) to ++(30:0.7cm);
\draw (dl)++(90:0.7cm) to ++(330:0.7cm);
\draw (dl)++(30:0.7cm) to ++(270:0.7cm);
\draw (dl)++(330:0.7cm) to ++(210:0.7cm);
\draw (dl)++(270:0.7cm) to ++(150:0.7cm);
\draw (dr)++(210:0.7cm) to ++(90:0.7cm);
\draw (dr)++(150:0.7cm) to ++(30:0.7cm);
\draw (dr)++(90:0.7cm) to ++(330:0.7cm);
\draw (dr)++(30:0.7cm) to ++(270:0.7cm);
\draw (dr)++(330:0.7cm) to ++(210:0.7cm);
\draw (dr)++(270:0.7cm) to ++(150:0.7cm);
\draw (dmc)++(210:0.7cm) to ++(90:0.7cm);
\draw (dmc)++(150:0.7cm) to ++(30:0.7cm);
\draw (dmc)++(90:0.7cm) to ++(330:0.7cm);
\draw (dmc)++(30:0.7cm) to ++(270:0.7cm);
\draw (dmc)++(330:0.7cm) to ++(210:0.7cm);
\draw (dmc)++(270:0.7cm) to ++(150:0.7cm);
\draw (d)++(210:0.7cm) to ++(90:0.7cm);
\draw (d)++(150:0.7cm) to ++(30:0.7cm);
\draw (d)++(90:0.7cm) to ++(330:0.7cm);
\draw (d)++(30:0.7cm) to ++(270:0.7cm);
\draw (d)++(330:0.7cm) to ++(210:0.7cm);
\draw (d)++(270:0.7cm) to ++(150:0.7cm);
\draw [thick, blue](uml)++(90:0.7cm) to ++(240:1.21cm);
\draw [thick, blue](umr)++(90:0.7cm) to ++(300:1.21cm);
\draw [thick, blue](dmc)++(210:0.7cm) to ++(0:1.21cm);
\draw [thick, blue](dl)++(210:0.7cm) to ++(30:1.39cm);
\draw [thick, blue](dr)++(150:0.7cm) to ++(330:1.39cm);
\draw [thick, blue](u)++(90:0.7cm) to ++(270:1.39cm);
\draw [thick, blue](ul)++(270:0.7cm) to ++(60:1.21cm);
\draw [thick, blue](d)++(150:0.7cm) to ++(0:1.21cm);
\draw [thick, blue](ur)++(150:0.7cm) to ++(300:1.21cm);
\end{tikzpicture}}
\end{center}
\end{figure}
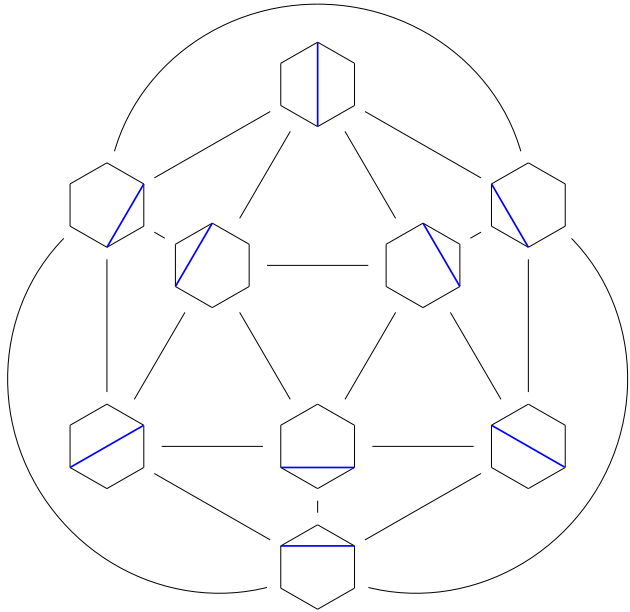
In Figure \ref{hexagoncomplex}, vertices are represented by the diagonals of a hexagon. This is an identification of a cluster complex with the corresponding \emph{arc complex} (see Subsection \ref{subsec:focus-An}).
We consider a mutation of $(\xx = (x_1,x_2,x_3),Q_B)$ in direction $3$. A non-labeled seed obtained by this mutation is given as
\[
\left(\xx'=\left(x_1,x_2,\dfrac{x_1+1}{x_3}\right),Q_B= \begin{xy}(0,0)*+{1}="I",(10,0)*+{2}="J", (20,0)*+{{3}}="K" \ar@{->}"I";"J"  \ar@/^4mm/"I";"K" \end{xy},\right).
\]
 In Figure \ref{hexagoncomplex}, a simplex corresponding to $\xx$ is a triangle that is the right-hand neighbor of the central triangle, and $\xx'$ is the right-hand neighbor of $\xx$. In Figure \ref{fig-positive-simplices}, the left-hand side is a complex $\Delta^+(\xx,B)$ and the right-hand side is $\Delta^+(\xx',B')$. Since mutation from $(\xx,B)$ to $(\xx',B')$ is a source mutation, we have $f(\Delta^+(\xx,B)) = f(\Delta^+(\xx',B'))$ by Corollary \ref{thm:invariant-sink/source}. Indeed, from Figure \ref{fig-positive-simplices}, both $f(\Delta^+(\xx,B))$ and $f(\Delta^+(\xx',B'))$ are (1,6,10,5).
\begin{figure}[ht]
\caption{Complexes $\Delta^+(\xx,B)$ and $\Delta^+(\xx',B')$}\label{fig-positive-simplices}
\[\scalebox{0.8}{
\begin{tikzpicture}
 \coordinate (0) at (0,0);
 \coordinate (1) at (18:3);
 \coordinate (2) at (90:3);
 \coordinate (3) at (162:3);
 \coordinate (4) at (234:3);
 \coordinate (5) at (306:3);
 \draw(0) to (1);
 \draw(0) to (2);
 \draw[ultra thick](0) to (3);
 \draw(0) to (4);
 \draw(0) to (5);
 \draw(1) to (2);
 \draw[ultra thick](2) to (3);
 \draw[ultra thick](3) to (4);
 \draw(4) to (5);
 \draw(5) to (1);
\fill[white](0) circle (0.9cm);
\fill[white] (1) circle (0.9cm);
\fill[white] (2) circle (0.9cm);
\fill[white] (3) circle (0.9cm);
\fill[white] (4) circle (0.9cm);
\fill[white] (5) circle (0.9cm);
\draw (0)++(210:0.7cm) to ++(90:0.7cm);
\draw (0)++(150:0.7cm) to ++(30:0.7cm);
\draw (0)++(90:0.7cm) to ++(330:0.7cm);
\draw (0)++(30:0.7cm) to ++(270:0.7cm);
\draw (0)++(330:0.7cm) to ++(210:0.7cm);
\draw (0)++(270:0.7cm) to ++(150:0.7cm);
\draw (1)++(210:0.7cm) to ++(90:0.7cm);
\draw (1)++(150:0.7cm) to ++(30:0.7cm);
\draw (1)++(90:0.7cm) to ++(330:0.7cm);
\draw (1)++(30:0.7cm) to ++(270:0.7cm);
\draw (1)++(330:0.7cm) to ++(210:0.7cm);
\draw (1)++(270:0.7cm) to ++(150:0.7cm);
\draw (2)++(210:0.7cm) to ++(90:0.7cm);
\draw (2)++(150:0.7cm) to ++(30:0.7cm);
\draw (2)++(90:0.7cm) to ++(330:0.7cm);
\draw (2)++(30:0.7cm) to ++(270:0.7cm);
\draw (2)++(330:0.7cm) to ++(210:0.7cm);
\draw (2)++(270:0.7cm) to ++(150:0.7cm);
\draw (3)++(210:0.7cm) to ++(90:0.7cm);
\draw (3)++(150:0.7cm) to ++(30:0.7cm);
\draw (3)++(90:0.7cm) to ++(330:0.7cm);
\draw (3)++(30:0.7cm) to ++(270:0.7cm);
\draw (3)++(330:0.7cm) to ++(210:0.7cm);
\draw (3)++(270:0.7cm) to ++(150:0.7cm);
\draw (4)++(210:0.7cm) to ++(90:0.7cm);
\draw (4)++(150:0.7cm) to ++(30:0.7cm);
\draw (4)++(90:0.7cm) to ++(330:0.7cm);
\draw (4)++(30:0.7cm) to ++(270:0.7cm);
\draw (4)++(330:0.7cm) to ++(210:0.7cm);
\draw (4)++(270:0.7cm) to ++(150:0.7cm);
\draw (5)++(210:0.7cm) to ++(90:0.7cm);
\draw (5)++(150:0.7cm) to ++(30:0.7cm);
\draw (5)++(90:0.7cm) to ++(330:0.7cm);
\draw (5)++(30:0.7cm) to ++(270:0.7cm);
\draw (5)++(330:0.7cm) to ++(210:0.7cm);
\draw (5)++(270:0.7cm) to ++(150:0.7cm);
\draw [thick, blue](0)++(270:0.7cm) to ++(60:1.21cm);
\draw [thick, blue](1)++(90:0.7cm) to ++(240:1.21cm);
\draw [thick, blue](2)++(90:0.7cm) to ++(270:1.39cm);
\draw [thick, red](3)++(150:0.7cm) to ++(300:1.21cm);
\draw [thick, blue](4)++(150:0.7cm) to ++(0:1.21cm);
\draw [thick, blue](5)++(210:0.7cm) to ++(30:1.39cm);
\end{tikzpicture}
\begin{tikzpicture}[baseline=-33mm]
 \coordinate (0) at (306:3);
 \coordinate (1) at (18:3);
 \coordinate (2) at (90:3);
 \coordinate (3) at (0,0);
 \coordinate (4) at (234:3);
 \coordinate (5) at (306:6);
 \draw(0) to (1);
 \draw(1) to (3);
 \draw(0) to (3);
 \draw(0) to (4);
 \draw[ultra thick](0) to (5);
 \draw(1) to (2);
 \draw(2) to (3);
 \draw(3) to (4);
 \draw[ultra thick](4) to (5);
 \draw[ultra thick](5) to (1);
\fill[white](0) circle (0.9cm);
\fill[white] (1) circle (0.9cm);
\fill[white] (2) circle (0.9cm);
\fill[white] (3) circle (0.9cm);
\fill[white] (4) circle (0.9cm);
\fill[white] (5) circle (0.9cm);
\draw (0)++(210:0.7cm) to ++(90:0.7cm);
\draw (0)++(150:0.7cm) to ++(30:0.7cm);
\draw (0)++(90:0.7cm) to ++(330:0.7cm);
\draw (0)++(30:0.7cm) to ++(270:0.7cm);
\draw (0)++(330:0.7cm) to ++(210:0.7cm);
\draw (0)++(270:0.7cm) to ++(150:0.7cm);
\draw (1)++(210:0.7cm) to ++(90:0.7cm);
\draw (1)++(150:0.7cm) to ++(30:0.7cm);
\draw (1)++(90:0.7cm) to ++(330:0.7cm);
\draw (1)++(30:0.7cm) to ++(270:0.7cm);
\draw (1)++(330:0.7cm) to ++(210:0.7cm);
\draw (1)++(270:0.7cm) to ++(150:0.7cm);
\draw (2)++(210:0.7cm) to ++(90:0.7cm);
\draw (2)++(150:0.7cm) to ++(30:0.7cm);
\draw (2)++(90:0.7cm) to ++(330:0.7cm);
\draw (2)++(30:0.7cm) to ++(270:0.7cm);
\draw (2)++(330:0.7cm) to ++(210:0.7cm);
\draw (2)++(270:0.7cm) to ++(150:0.7cm);
\draw (3)++(210:0.7cm) to ++(90:0.7cm);
\draw (3)++(150:0.7cm) to ++(30:0.7cm);
\draw (3)++(90:0.7cm) to ++(330:0.7cm);
\draw (3)++(30:0.7cm) to ++(270:0.7cm);
\draw (3)++(330:0.7cm) to ++(210:0.7cm);
\draw (3)++(270:0.7cm) to ++(150:0.7cm);
\draw (4)++(210:0.7cm) to ++(90:0.7cm);
\draw (4)++(150:0.7cm) to ++(30:0.7cm);
\draw (4)++(90:0.7cm) to ++(330:0.7cm);
\draw (4)++(30:0.7cm) to ++(270:0.7cm);
\draw (4)++(330:0.7cm) to ++(210:0.7cm);
\draw (4)++(270:0.7cm) to ++(150:0.7cm);
\draw (5)++(210:0.7cm) to ++(90:0.7cm);
\draw (5)++(150:0.7cm) to ++(30:0.7cm);
\draw (5)++(90:0.7cm) to ++(330:0.7cm);
\draw (5)++(30:0.7cm) to ++(270:0.7cm);
\draw (5)++(330:0.7cm) to ++(210:0.7cm);
\draw (5)++(270:0.7cm) to ++(150:0.7cm);
\draw [thick, blue](0)++(210:0.7cm) to ++(30:1.39cm);
\draw [thick, blue](1)++(90:0.7cm) to ++(240:1.21cm);
\draw [thick, blue](2)++(90:0.7cm) to ++(270:1.39cm);
\draw [thick, blue](3)++(270:0.7cm) to ++(60:1.21cm);
\draw [thick, blue](4)++(150:0.7cm) to ++(0:1.21cm);
\draw [thick, red](5)++(210:0.7cm) to ++(0:1.21cm);
\end{tikzpicture}}
\]
\end{figure}
\end{example}
\begin{remark}
As in Example \ref{ex:not-iso-positive-complex}, although $f(\Delta^+(\xx,B))=f(\Delta^+(\xx',B'))$, $\Delta^+(\xx,B)\cong\Delta^+(\xx',B')$ does not necessarily hold.
\end{remark}
We also consider the case in which the face vector of $\Delta^+(\xx,B)$ does not coincide with that of $\Delta^+(\xx',B')$.
\begin{example}\label{ex:not-equal-simplex}
We consider the same situation in Example \ref{ex:not-iso-positive-complex} and a mutation of $(\xx=(x_1,x_2,x_3),B)$ in direction $1$. A non-labeled seed obtained by this mutation is
\[
\left(\xx''=\left(\dfrac{x_2+x_3}{x_1},x_2,x_3\right),Q_B''=\begin{xy}(0,0)*+{1}="I",(10,0)*+{2}="J", (20,0)*+{{3}}="K" \ar@{->}"J";"I" \ar@{->}"K";"J" \ar@/^4mm/"I";"K" \end{xy}\right).
\]
This is neither a sink nor a source mutation. In Figure \ref{hexagoncomplex}, a simplex corresponding to $\xx''$ is a triangle of the center. In Figure \ref{fig:not-equal-simplex}, the left-hand side is a complex $\Delta^+(\xx,B)$, and the right-hand side is $\Delta^+(\xx'',B'')$.
\begin{figure}[ht]
\caption{Complexes $\Delta^+(\xx,B)$ and $\Delta^+(\xx'',B'')$}\label{fig:not-equal-simplex}
\[\scalebox{0.8}{
\begin{tikzpicture}
 \coordinate (0) at (0,0);
 \coordinate (1) at (18:3);
 \coordinate (2) at (90:3);
 \coordinate (3) at (162:3);
 \coordinate (4) at (234:3);
 \coordinate (5) at (306:3);
 \draw[ultra thick](0) to (1);
 \draw(0) to (2);
 \draw(0) to (3);
 \draw(0) to (4);
 \draw(0) to (5);
 \draw[ultra thick](1) to (2);
 \draw(2) to (3);
 \draw(3) to (4);
 \draw(4) to (5);
 \draw[ultra thick](5) to (1);
\fill[white](0) circle (0.9cm);
\fill[white] (1) circle (0.9cm);
\fill[white] (2) circle (0.9cm);
\fill[white] (3) circle (0.9cm);
\fill[white] (4) circle (0.9cm);
\fill[white] (5) circle (0.9cm);
\draw (0)++(210:0.7cm) to ++(90:0.7cm);
\draw (0)++(150:0.7cm) to ++(30:0.7cm);
\draw (0)++(90:0.7cm) to ++(330:0.7cm);
\draw (0)++(30:0.7cm) to ++(270:0.7cm);
\draw (0)++(330:0.7cm) to ++(210:0.7cm);
\draw (0)++(270:0.7cm) to ++(150:0.7cm);
\draw (1)++(210:0.7cm) to ++(90:0.7cm);
\draw (1)++(150:0.7cm) to ++(30:0.7cm);
\draw (1)++(90:0.7cm) to ++(330:0.7cm);
\draw (1)++(30:0.7cm) to ++(270:0.7cm);
\draw (1)++(330:0.7cm) to ++(210:0.7cm);
\draw (1)++(270:0.7cm) to ++(150:0.7cm);
\draw (2)++(210:0.7cm) to ++(90:0.7cm);
\draw (2)++(150:0.7cm) to ++(30:0.7cm);
\draw (2)++(90:0.7cm) to ++(330:0.7cm);
\draw (2)++(30:0.7cm) to ++(270:0.7cm);
\draw (2)++(330:0.7cm) to ++(210:0.7cm);
\draw (2)++(270:0.7cm) to ++(150:0.7cm);
\draw (3)++(210:0.7cm) to ++(90:0.7cm);
\draw (3)++(150:0.7cm) to ++(30:0.7cm);
\draw (3)++(90:0.7cm) to ++(330:0.7cm);
\draw (3)++(30:0.7cm) to ++(270:0.7cm);
\draw (3)++(330:0.7cm) to ++(210:0.7cm);
\draw (3)++(270:0.7cm) to ++(150:0.7cm);
\draw (4)++(210:0.7cm) to ++(90:0.7cm);
\draw (4)++(150:0.7cm) to ++(30:0.7cm);
\draw (4)++(90:0.7cm) to ++(330:0.7cm);
\draw (4)++(30:0.7cm) to ++(270:0.7cm);
\draw (4)++(330:0.7cm) to ++(210:0.7cm);
\draw (4)++(270:0.7cm) to ++(150:0.7cm);
\draw (5)++(210:0.7cm) to ++(90:0.7cm);
\draw (5)++(150:0.7cm) to ++(30:0.7cm);
\draw (5)++(90:0.7cm) to ++(330:0.7cm);
\draw (5)++(30:0.7cm) to ++(270:0.7cm);
\draw (5)++(330:0.7cm) to ++(210:0.7cm);
\draw (5)++(270:0.7cm) to ++(150:0.7cm);
\draw [thick, blue](0)++(270:0.7cm) to ++(60:1.21cm);
\draw [thick, red](1)++(90:0.7cm) to ++(240:1.21cm);
\draw [thick, blue](2)++(90:0.7cm) to ++(270:1.39cm);
\draw [thick, blue](3)++(150:0.7cm) to ++(300:1.21cm);
\draw [thick, blue](4)++(150:0.7cm) to ++(0:1.21cm);
\draw [thick, blue](5)++(210:0.7cm) to ++(30:1.39cm);
\end{tikzpicture}
\hspace{1cm}
\begin{tikzpicture}[baseline=-33mm]
 \coordinate (0) at (306:3);
 \coordinate (1) at (198:5);
 \coordinate (2) at (90:3);
 \coordinate (3) at (0,0);
 \coordinate (4) at (234:3);
 \coordinate (5) at (162:3);
 \draw(0) to (3);
 \draw(0) to (4);
 \draw(2) to (3);
 \draw(3) to (4);
 \draw(3) to (5);
 \draw(2) to (5);
 \draw(4) to (5);
 \draw[ultra thick](1) to (4);
 \draw[ultra thick](1) to (5);
\fill[white](0) circle (0.9cm);
\fill[white] (1) circle (0.9cm);
\fill[white] (2) circle (0.9cm);
\fill[white] (3) circle (0.9cm);
\fill[white] (4) circle (0.9cm);
\fill[white] (5) circle (0.9cm);
\draw (0)++(210:0.7cm) to ++(90:0.7cm);
\draw (0)++(150:0.7cm) to ++(30:0.7cm);
\draw (0)++(90:0.7cm) to ++(330:0.7cm);
\draw (0)++(30:0.7cm) to ++(270:0.7cm);
\draw (0)++(330:0.7cm) to ++(210:0.7cm);
\draw (0)++(270:0.7cm) to ++(150:0.7cm);
\draw (1)++(210:0.7cm) to ++(90:0.7cm);
\draw (1)++(150:0.7cm) to ++(30:0.7cm);
\draw (1)++(90:0.7cm) to ++(330:0.7cm);
\draw (1)++(30:0.7cm) to ++(270:0.7cm);
\draw (1)++(330:0.7cm) to ++(210:0.7cm);
\draw (1)++(270:0.7cm) to ++(150:0.7cm);
\draw (2)++(210:0.7cm) to ++(90:0.7cm);
\draw (2)++(150:0.7cm) to ++(30:0.7cm);
\draw (2)++(90:0.7cm) to ++(330:0.7cm);
\draw (2)++(30:0.7cm) to ++(270:0.7cm);
\draw (2)++(330:0.7cm) to ++(210:0.7cm);
\draw (2)++(270:0.7cm) to ++(150:0.7cm);
\draw (3)++(210:0.7cm) to ++(90:0.7cm);
\draw (3)++(150:0.7cm) to ++(30:0.7cm);
\draw (3)++(90:0.7cm) to ++(330:0.7cm);
\draw (3)++(30:0.7cm) to ++(270:0.7cm);
\draw (3)++(330:0.7cm) to ++(210:0.7cm);
\draw (3)++(270:0.7cm) to ++(150:0.7cm);
\draw (4)++(210:0.7cm) to ++(90:0.7cm);
\draw (4)++(150:0.7cm) to ++(30:0.7cm);
\draw (4)++(90:0.7cm) to ++(330:0.7cm);
\draw (4)++(30:0.7cm) to ++(270:0.7cm);
\draw (4)++(330:0.7cm) to ++(210:0.7cm);
\draw (4)++(270:0.7cm) to ++(150:0.7cm);
\draw (5)++(210:0.7cm) to ++(90:0.7cm);
\draw (5)++(150:0.7cm) to ++(30:0.7cm);
\draw (5)++(90:0.7cm) to ++(330:0.7cm);
\draw (5)++(30:0.7cm) to ++(270:0.7cm);
\draw (5)++(330:0.7cm) to ++(210:0.7cm);
\draw (5)++(270:0.7cm) to ++(150:0.7cm);
\draw [thick, blue](0)++(210:0.7cm) to ++(30:1.39cm);
\draw [thick, red](1)++(150:0.7cm) to ++(330:1.39cm);
\draw [thick, blue](2)++(90:0.7cm) to ++(270:1.39cm);
\draw [thick, blue](3)++(270:0.7cm) to ++(60:1.21cm);
\draw [thick, blue](4)++(150:0.7cm) to ++(0:1.21cm);
\draw [thick, blue](5)++(150:0.7cm) to ++(300:1.21cm);
\end{tikzpicture}}
\]
\end{figure}

Then, we have the face vectors $f({\Delta^+(\xx,B)})$ and $f(\Delta^+(\xx'',B''))$ as follows:
\[
f({\Delta^+(\xx,B)})=(1,6,10,5),\quad f({\Delta^+(\xx'',B'')})=(1,6,9,4).
\]
Thus, the difference between the two is $(0, 0, 1, 1)$. We can also calculate this value using Theorem \ref{thm:mutation-positivecomplex}. However, we need face vectors of positive cluster complexes of subquivers to do so. In Section 5, we provide face vectors of positive cluster complex in special cases and thereby calculate this example.
\end{example}
\begin{remark}
 We consider an example such that $f(\Delta^+ (\xx_t,B_t)) = f(\Delta^+ (\xx_{t'},B_{t'}))$ and $(\xx_{t'},B_{t'})$ are obtained from $ (\xx_t,B_t)$ by a mutation that is neither a sink nor a source mutation.
 We give $Q_{B_t}$ and $Q_{B_{t'}}$ as
 \begin{align*}
     Q_{B_t}=\begin{xy}(0,0)="N",(0,10)*+{1}="I",(20,10)*+{2}="J", (10,0)*+{{3}}="K",(0,-10)*+{4}="L",(20,-10)*+{5}="M" \ar@{<-}"I";"J"  \ar@{<-}"K";"I"\ar@{<-}"J";"K"\ar@{->}"K";"L" \ar@{->}"L";"M"\ar@{->}"M";"K"\end{xy},\quad
     Q_{B_{t'}}=\begin{xy}(0,0)="N",(0,10)*+{1}="I",(20,10)*+{2}="J", (10,0)*+{{3}}="K",(0,-10)*+{4}="L",(20,-10)*+{5}="M" \ar@{->}"I";"L"  \ar@{->}"K";"I"\ar@{->}"J";"K"\ar@{<-}"K";"L" \ar@{<-}"J";"M"\ar@{<-}"M";"K"\end{xy}.
 \end{align*}
A seed $(\xx_{t'},B_{t'})$ is obtained from $(\xx_{t},B_{t})$ by a mutation in direction $3$. We note that $B_{t}$ and $B_{t'}$ are $A_5$ type. Thus, $\Delta^+(\xx_{t},B_{t})$ and $\Delta^+(\xx_{t'},B_{t'})$ are finite type. Since full subquivers of $Q_{B_{t}}$ and $Q_{B_{t'}}$ whose vertices are $\{1,2,4,5\}$ are the same quiver, which consists of two $A_2$ type quivers, $\Delta^+(\xx_{t},B_{t})$ and $\Delta^+(\xx_{t'},B_{t'})$ have the same face vectors. In fact, since $Q_{B_t'}$ is the same shape as $Q_{B_t}$, we have $\Delta^+(\xx_{t},B_{t})\cong\Delta^+(\xx_{t'},B_{t'})$.
 \end{remark}

By using the above results of face vectors, we obtain properties of $h$-vectors. These vectors are defined, for example, in \cite{bruher}*{Section 5.1}.

\begin{definition}\label{hvector}
Let $K$ be a simplicial complex and $f(K)=(f_{-1},f_0,\dots,f_{n-1})$ the face vector of $K$. We define the \emph{$h$-vector} as $h(K)=(h_0,\dots,h_{n})$, where
\begin{align}
    \sum_{i=0}^{n}f_{i-1}(x-1)^{n-i}=\sum_{i=0}^nh_ix^{n-i}.
\end{align}
\end{definition}
By Corollaries \ref{thm:invariant-sink/source} and \ref{cor:independent-orientation}, we have the following corollaries.
 \begin{corollary}\label{thm:h-invariant-sink/source}
Let $\Delta (\xx,B)$ be a cluster complex of finite type. If $(\xx',B')$ is obtained from $(\xx,B)$ by a sink or source mutation, then we have $h(\Delta^+ (\xx,B))=h(\Delta^+(\xx',B'))$.
 \end{corollary}
\begin{corollary}\label{cor:h-independent-orientation}
Let $B$ and $B'$ be skew-symmetrizable matrices. If $Q_B$ and $Q_{B'}$ have the same underlying graph in Figure \ref{fig:valued-quiver}, then $h(\Delta^+ (\xx,B))=h(\Delta^+(\xx',B'))$. 
\end{corollary}
When a simplicial complex is pure and shellable, the $h$-vector has a special meaning; see Section \ref{applicationtotautilting2}.
\section{Explicit descriptions of positive cluster complexes of the classical finite type}
In this section, we explicitly describe some positive cluster complexes of finite type.
\subsection{Type $A_n$}\label{subsec:focus-An}
In this subsection, we focus on a cluster complex of $A_n$ type, that is, a cluster complex $\Delta(\xx,B)$ such that the initial exchange matrix $B$ is mutation equivalent to an exchange matrix of the tree $A_n$ type. Notably, we consider the case in which $B$ is an exchange matrix of tree finite type of $A_n$ (i.e., $Q_B$ has a diagram of $A_n$ type in Figure \ref{fig:valued-quiver} as its underlying diagram).

First, we explain the realization of cluster structure by marked surfaces, which is given by \cite{fzii}*{Section 12.2} and \cite{fst}.
Let $S$ be a regular $(n+3)$-gon. We say that a diagonal $\ell$ is \emph{compatible} with $\ell'$ if $\ell$ and $\ell'$ do not cross at the interior of $S$, and a set of pairwise compatible diagonals is a \emph{compatible set}. We note that a maximal compatible set of $S$ gives a triangulation of $S$; let $T$ be a maximal compatible set. We label these diagonals by $(\ell_1,\dots,\ell_n)$ and we call the result a \emph{labeled maximal compatible set}. We define a matrix $B_T$ as that obtained from $T$ as follows. For each triangle $\Delta$ in a triangulation of $S$ given by $T$, we define the $n\times n$ integer matrix $B^\Delta=(b^\Delta_{ij})$ by setting
\begin{align*}
b^\Delta_{ij} =\begin{cases}
 1 \quad&\text{if $\Delta$ has diagonals labeled $\ell_{i}$ and $\ell_{j}$, }\\
 &\text{with $\ell_j$ following $\ell_i$ in the clockwise order},\\
 -1 \quad &\text{if the same holds, with the counterclockwise order,}\\
 0 \quad &\text{otherwise}.
\end{cases}
\end{align*}
The matrix $B=B_T=(b_{ij})$ is then defined by
\begin{align*}
B=\sum_\Delta B^\Delta,
\end{align*}
where the summation runs over all triangles $\Delta$ in a triangulation of $S$ given by $T$.
\begin{example}
Let $T=(\ell_1,\ell_2,\ell_3)$ be a labeled maximal compatible set given as in Figure \ref{fig:triangulation-matrix-correspondence}. Then, we have $B_T=\begin{bmatrix}
0&-1&1\\
1&0&-1\\
-1&1&0
\end{bmatrix}$. The corresponding quiver is $
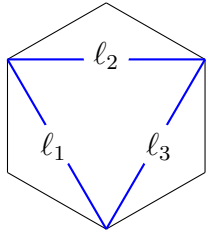
\begin{xy}(0,0)*+{1}="I",(10,0)*+{2}="J", (20,0)*+{{3}}="K" \ar@{->}"J";"I" \ar@{->}"K";"J" \ar@/^4mm/"I";"K" \end{xy}$.
\begin{figure}[ht]
\caption{Labeled maximal compatible set of a hexagon}\label{fig:triangulation-matrix-correspondence}
\[
\begin{tikzpicture}[baseline=0cm]
\coordinate (0) at (0,0);
\coordinate (1) at (30:1.5);
\coordinate (2) at (90:1.5);
\coordinate (3) at (150:1.5);
\coordinate (4) at (210:1.5);
\coordinate (5) at (270:1.5);
\coordinate (6) at (330:1.5);
\draw (1) to (2);
\draw (2) to (3);
\draw (3) to (4);
\draw (4) to (5);
\draw (5) to (6);
\draw (6) to (1);
\draw [thick, blue] (1) to (3);
\draw [thick, blue] (3) to (5);
\draw [thick, blue](1) to (5);
\fill [white](90:0.8) circle (3mm);
\fill [white](210:0.8) circle (3mm);
\fill [white](330:0.8) circle (3mm);
\node at (210:0.8) {$\ell_1$};
\node at (90:0.8) {$\ell_2$};
\node at (330:0.8) {$\ell_3$};
\end{tikzpicture}
\]
\end{figure}
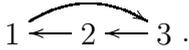
\end{example}
For $i\in \{1,\dots,n\}$ and a maximal compatible set $T$, we define a \emph{flip $\varphi_i(T)$ in direction $i$} as a transformation that obtains the unique maximal compatible set $T'$ from $T$ by removing $\ell_i$ and adding the other diagonal $\ell'_i$ to $T\setminus\{\ell_i\}$.
We explain that flips establish the same structure as seed mutations. We say that a pair $(T,B_T)$ of labeled a maximal compatible set and an exchange matrix $B_T$ given by $T$ as described above is a \emph{labeled triangulation pair}. Abusing notation, we denote $\varphi_i(T,B_T):=(\varphi_i(T),B_{\varphi_i(T)})$ and we refer to this operation as a \emph{flip of a labeled triangulation pair}. We define a \emph{triangulation pattern} as an assignment $\Xi_t=(T_t,B_{T_t})$ to every vertex $t\in \mathbb{T}_n$ such that labeled triangulation pairs $\Xi_t$ and $\Xi_{t'}$ assigned to the endpoints of any edge
\begin{xy}(0,1)*+{t}="A",(10,1)*+{t'}="B",\ar@{-}^k"A";"B" \end{xy}
are obtained from each other by a flip of a labeled triangulation pair in direction $k$. We denote by $P\colon t\mapsto \Xi_t$ this assignment. To distinguish a cluster pattern from a triangulation pattern, we use the notation $P^\Sigma$ for the former and $P^\Xi$ for the latter. We call a triangulation pair associated with the rooted vertex $t_0$ the \emph{initial triangulation pair}. We denote by $P(T,B_T)$ a triangulation pattern with the initial triangulation pair $(T,B_T)$. We use the notation $\ell_{i;t}$ as the $i$th diagonal in $T_t$.
We define a \emph{non-labeled triangulation pair} $[\Xi]$ as the equivalence classes represented by a labeled triangulation pair $\Xi$ in the same way as a non-labeled seed.
We define a \emph{flip} $\varphi_\ell$ of non-labeled triangulation pair, an \emph{arc complex} $\Delta(T,B_T)$ in the same way as a mutation of non-labeled seed and a cluster complex, respectively.
Thus, we have the following basic property.
\begin{theorem}[\cite{fst}*{Proposition 4.8, Theorem 7.11}]\label{thm:arc-variable}
Let $S$ be a regular $(n+3)$-gon and $T$ a labeled maximal compatible set. 
\begin{itemize}
\item[(1)]We have $B_{\varphi_i(T)}=\mu_i(B_T)$, where $\mu_i$ is a mutation in direction $i$ of an exchange matrix.
\item[(2)]We have the canonical bijection between the sets of all diagonals of $S$ and the set of cluster variables in $\{\xx_t\}_{t\in\TT_n}$ by the correspondence $\ell_{i;t}\mapsto x_{i;t}$. Furthermore, this bijection induces an isomorphism between the arc complex $\Delta(T,B_T)$ and the cluster complex $\Delta(\xx,B_T)$.
\end{itemize}
\end{theorem}
We note that $B_T$ is a skew-symmetric matrix which is mutation equivalent to an exchange matrix of tree $A_n$ type. Therefore, by using the canonical bijection in Theorem \ref{thm:arc-variable}, we identify diagonals of $(n+3)$-gon with cluster variables of a cluster pattern of $A_n$ type. In particular, $\ell$ is compatible with $\ell'$ if and only if $x_{\ell}$ is compatible with $x_{\ell'}$, where $x_\ell$ and $x_{\ell'}$ are corresponding cluster variables to $\ell$ and $\ell'$.
\begin{example}\label{ex:correspondence-arc-variable-A3}
 Let $B=\begin{bmatrix}
 0&-1&0\\
 1&0&-1\\
 0&1&0
 \end{bmatrix}$ of type $A_3$. The corresponding quiver is $\begin{xy}(0,0)*+{1}="I",(10,0)*+{2}="J", (20,0)*+{{3}}="K" \ar@{->}"J";"I" \ar@{->}"K";"J"  \end{xy}$. Let $T$ be a set $\left\{\ell_1=\begin{tikzpicture}[baseline=-1mm,scale=0.3]
 \coordinate (u) at(90:1); \coordinate (lu) at(150:1); \coordinate (ld) at(-150:1);
 \coordinate (ru) at(30:1); \coordinate (rd) at(-30:1); \coordinate (d) at(-90:1);
 \draw (u)--(lu)--(ld)--(d)--(rd)--(ru)--(u); \node at(0,1.3) {}; \node at(0,-1.3) {}; \draw[blue,thick] (ru)--(d);
\end{tikzpicture}, \ \ell_2=
\begin{tikzpicture}[baseline=-1mm,scale=0.3]
 \coordinate (u) at(90:1); \coordinate (lu) at(150:1); \coordinate (ld) at(-150:1);
 \coordinate (ru) at(30:1); \coordinate (rd) at(-30:1); \coordinate (d) at(-90:1);
 \draw (u)--(lu)--(ld)--(d)--(rd)--(ru)--(u); \node at(0,1.3) {}; \node at(0,-1.3) {}; \draw[blue,thick] (u)--(d);
\end{tikzpicture},\ \ell_3=
\begin{tikzpicture}[baseline=-1mm,scale=0.3]
 \coordinate (u) at(90:1); \coordinate (lu) at(150:1); \coordinate (ld) at(-150:1);
 \coordinate (ru) at(30:1); \coordinate (rd) at(-30:1); \coordinate (d) at(-90:1);
 \draw (u)--(lu)--(ld)--(d)--(rd)--(ru)--(u); \node at(0,1.3) {}; \node at(0,-1.3) {}; \draw[blue,thick] (lu)--(d);
\end{tikzpicture}
\right\}$ of diagonals of a hexagon. Then, we have $B=B_T$. We consider a cluster pattern $P^{\Sigma}(\xx,B)$ and a triangulation pattern $P^\Xi(T,B_T)$. Then, we have $\mu_{x_3}\left[\left(x_1,x_2,x_3\right)\right]=\left[\left(x_1,x_2,\dfrac{x_2+1}{x_3}\right)\right]$ by a mutation in direction $x_3$. On the other hand, we have $\varphi_{\ell_3}\left(
\begin{tikzpicture}[baseline=-1mm,scale=0.3]
 \coordinate (u) at(90:1);
 \coordinate (lu) at(150:1);
 \coordinate (ld) at(-150:1);
 \coordinate (ru) at(30:1);
 \coordinate (rd) at(-30:1);
 \coordinate (d) at(-90:1);
 \draw (u)--(lu)--(ld)--(d)--(rd)--(ru)--(u);
 \node at(0,1.3) {};
 \node at(0,-1.3) {};
 \draw[blue,thick] (lu)--(d);
 \draw[blue,thick] (u)--(d);
 \draw[blue,thick] (ru)--(d);
\end{tikzpicture}\right)=
\begin{tikzpicture}[baseline=-1mm,scale=0.3]
 \coordinate (u) at(90:1);
 \coordinate (lu) at(150:1);
 \coordinate (ld) at(-150:1);
 \coordinate (ru) at(30:1);
 \coordinate (rd) at(-30:1);
 \coordinate (d) at(-90:1);
 \draw (u)--(lu)--(ld)--(d)--(rd)--(ru)--(u);
 \node at(0,1.3) {};
 \node at(0,-1.3) {};
 \draw[blue,thick] (u)--(ld);
 \draw[blue,thick] (u)--(d);
 \draw[blue,thick] (d)--(ru);
\end{tikzpicture}
$ by a flip in direction $\ell_3$. Therefore, $\dfrac{x_2+1}{x_3}$ corresponds to $\begin{tikzpicture}[baseline=-1mm,scale=0.3]
 \coordinate (u) at(90:1);
 \coordinate (lu) at(150:1);
 \coordinate (ld) at(-150:1);
 \coordinate (ru) at(30:1);
 \coordinate (rd) at(-30:1);
 \coordinate (d) at(-90:1);
 \draw (u)--(lu)--(ld)--(d)--(rd)--(ru)--(u);
 \node at(0,1.3) {};
 \node at(0,-1.3) {};
 \draw[blue,thick] (ld)--(u);
\end{tikzpicture}$ by the canonical bijection.
The canonical bijection between the set of cluster variables and the diagonal set is given in Table \ref{examp}.
\begin{table}[ht]
\begin{tabular}{c|c}
 cluster variable $x$ & diagonal corresponding to $x$
\end{tabular}
\vspace{2mm}\\
\begin{minipage}{0.3\hsize}
\begin{center}\begin{tabular}{c|c}
 $x_1$ &
\begin{tikzpicture}[baseline=-1mm,scale=0.5]
 \coordinate (u) at(90:1); \coordinate (lu) at(150:1); \coordinate (ld) at(-150:1);
 \coordinate (ru) at(30:1); \coordinate (rd) at(-30:1); \coordinate (d) at(-90:1);
 \draw (u)--(lu)--(ld)--(d)--(rd)--(ru)--(u); \node at(0,1.3) {}; \node at(0,-1.3) {}; \draw[blue,thick] (d)--(ru);
\end{tikzpicture}
\\\hline
 $x_2$ &
\begin{tikzpicture}[baseline=-1mm,scale=0.5]
 \coordinate (u) at(90:1); \coordinate (lu) at(150:1); \coordinate (ld) at(-150:1);
 \coordinate (ru) at(30:1); \coordinate (rd) at(-30:1); \coordinate (d) at(-90:1);
 \draw (u)--(lu)--(ld)--(d)--(rd)--(ru)--(u); \node at(0,1.3) {}; \node at(0,-1.3) {}; \draw[blue,thick] (u)--(d);
\end{tikzpicture}
\\\hline
 $x_3$ &
\begin{tikzpicture}[baseline=-1mm,scale=0.5]
 \coordinate (u) at(90:1); \coordinate (lu) at(150:1); \coordinate (ld) at(-150:1);
 \coordinate (ru) at(30:1); \coordinate (rd) at(-30:1); \coordinate (d) at(-90:1);
 \draw (u)--(lu)--(ld)--(d)--(rd)--(ru)--(u); \node at(0,1.3) {}; \node at(0,-1.3) {}; \draw[blue,thick] (d)--(lu);
\end{tikzpicture}
\end{tabular}\end{center}
\end{minipage}
\begin{minipage}{0.3\hsize}
\begin{center}\begin{tabular}{c|c}
 $\cfrac{x_2+1}{x_1}$ &
\begin{tikzpicture}[baseline=-1mm,scale=0.5]
 \coordinate (u) at(90:1); \coordinate (lu) at(150:1); \coordinate (ld) at(-150:1);
 \coordinate (ru) at(30:1); \coordinate (rd) at(-30:1); \coordinate (d) at(-90:1);
 \draw (u)--(lu)--(ld)--(d)--(rd)--(ru)--(u); \node at(0,1.3) {}; \node at(0,-1.3) {}; \draw[blue,thick] (u)--(rd);
\end{tikzpicture}
\\\hline
 $\cfrac{x_1+x_3}{x_2}$ &
\begin{tikzpicture}[baseline=-1mm,scale=0.5]
 \coordinate (u) at(90:1); \coordinate (lu) at(150:1); \coordinate (ld) at(-150:1);
 \coordinate (ru) at(30:1); \coordinate (rd) at(-30:1); \coordinate (d) at(-90:1);
 \draw (u)--(lu)--(ld)--(d)--(rd)--(ru)--(u); \node at(0,1.3) {}; \node at(0,-1.3) {}; \draw[blue,thick] (lu)--(ru);
\end{tikzpicture}
\\\hline
 $\cfrac{x_2+1}{x_3}$ &
\begin{tikzpicture}[baseline=-1mm,scale=0.5]
 \coordinate (u) at(90:1); \coordinate (lu) at(150:1); \coordinate (ld) at(-150:1);
 \coordinate (ru) at(30:1); \coordinate (rd) at(-30:1); \coordinate (d) at(-90:1);
 \draw (u)--(lu)--(ld)--(d)--(rd)--(ru)--(u); \node at(0,1.3) {}; \node at(0,-1.3) {}; \draw[blue,thick] (u)--(ld);
\end{tikzpicture}
\end{tabular}\end{center}
\end{minipage}
\begin{tabular}{c|c}
 $\cfrac{x_1+x_3+x_2x_3}{x_1x_2}$ &
\begin{tikzpicture}[baseline=-1mm,scale=0.5]
 \coordinate (u) at(90:1); \coordinate (lu) at(150:1); \coordinate (ld) at(-150:1);
 \coordinate (ru) at(30:1); \coordinate (rd) at(-30:1); \coordinate (d) at(-90:1);
 \draw (u)--(lu)--(ld)--(d)--(rd)--(ru)--(u); \node at(0,1.3) {}; \node at(0,-1.3) {}; \draw[blue,thick] (lu)--(rd);
\end{tikzpicture}
\\\hline
 $\cfrac{x_1+x_3+x_1x_2}{x_2x_3}$ &
\begin{tikzpicture}[baseline=-1mm,scale=0.5]
 \coordinate (u) at(90:1); \coordinate (lu) at(150:1); \coordinate (ld) at(-150:1);
 \coordinate (ru) at(30:1); \coordinate (rd) at(-30:1); \coordinate (d) at(-90:1);
 \draw (u)--(lu)--(ld)--(d)--(rd)--(ru)--(u); \node at(0,1.3) {}; \node at(0,-1.3) {}; \draw[blue,thick] (ru)--(ld);
\end{tikzpicture}
\\\hline
 $\cfrac{x_1+x_3+x_1x_2+x_2x_3}{x_1x_2x_3}$ &
\begin{tikzpicture}[baseline=-1mm,scale=0.5]
 \coordinate (u) at(90:1); \coordinate (lu) at(150:1); \coordinate (ld) at(-150:1);
 \coordinate (ru) at(30:1); \coordinate (rd) at(-30:1); \coordinate (d) at(-90:1);
 \draw (u)--(lu)--(ld)--(d)--(rd)--(ru)--(u); \node at(0,1.3) {}; \node at(0,-1.3) {}; \draw[blue,thick] (ld)--(rd);
\end{tikzpicture}
  \end{tabular}
 \vspace{5mm}
 \caption{Canonical bijection between the set of cluster variables and diagonal set of $A_3$ type}\label{examp}
\end{table}
\end{example}
The main theorem in this subsection is the following, in which we refer to Definition \ref{def:cone-star-link} for the definition of the cone.
\begin{theorem}\label{thm:positive-simplex-A_n}
Let $\Delta(\xx,B)$ be a cluster complex of $A_n$ type. Suppose that quiver $Q_B$ has the following form.
\begin{align}\label{assumption:An}
\begin{xy}
    (50,0)*{\circ}="A",(60,0)*{\circ}="B",(70,0)*{\circ}="C",(80,0)*{\circ}="D",(85,0)*{}="E",(90,0)*{}="F",(95,0)*{\circ}="G",(105,0)*{\circ}="H"\ar@{->} "A";"B", \ar@{->} "B";"C", \ar@{->} "C";"D", \ar@{-} "D";"E", \ar@{.} "E";"F",\ar@{->} "F";"G",\ar@{->} "G";"H"
\end{xy}
\end{align}
Then $\Delta^+(\xx,B)$ is isomorphic to the cone of $\Delta(A_{n-1})$.
\end{theorem}
\begin{remark}\label{rem:equivalent-condition-An}
A matrix $B$ satisfies the assumption in Theorem \ref{thm:positive-simplex-A_n} if and only if it has the following property: for a triangulation of an $(n+3)$-gon $S$ obtained by a maximal compatible set $T$ which gives $B_T=B$, there exists a vertex $v$ of $S$ such that all diagonals in $T$ share $v$ (see Figure \ref{ex:assumption-of-nakayama}).
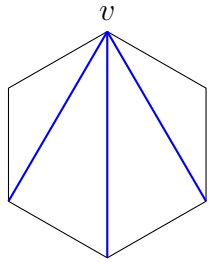
\begin{figure}[ht]
\caption{Example of triangulation corresponding to $B$ satisfying the assumption in Theorem  \ref{thm:positive-simplex-A_n}}\label{ex:assumption-of-nakayama}
\[
\begin{tikzpicture}
 \coordinate (u) at(90:1.5);
 \coordinate (lu) at(150:1.5);
 \coordinate (ld) at(-150:1.5);
 \coordinate (ru) at(30:1.5);
 \coordinate (rd) at(-30:1.5);
 \coordinate (d) at(-90:1.5);
 \draw (u)--(lu)--(ld)--(d)--(rd)--(ru)--(u);
 \draw[blue,thick] (u)--(ld);
 \draw[blue,thick] (u)--(d);
 \draw[blue,thick] (u)--(rd);
 \node at(0,1.75) {$v$};
\end{tikzpicture}
\]
\end{figure}
\end{remark}
\begin{definition}
Let $K$ be a simplicial complex and $V$ a vertex subset of $K$. For a vertex $x$ in $K$, we say that $x$ is \emph{non-linked to $V$} when $\{x,v\}$ is not a simplex of $K$ for all $v\in V$.
\end{definition}

We denote by $I$ the vertex set consisting of all cluster variables in the initial cluster $\xx$. 

The following lemma can be applied to exchange matrices of tree $A_n$ type, which contains the case \eqref{assumption:An}.

\begin{lemma}\label{lem:unique-interior}
Let $B=(b_{ij})$ be an exchange matrix of tree $A_n$ type. For a cluster complex $\Delta(\xx,B)$, there exists a unique vertex $x$ such that $x$ is non-linked to $I$.
\end{lemma}
\begin{proof}
By the canonical bijection between cluster variables and diagonals of a regular $(n+3)$-gon $S$, it suffices to show that there is a unique diagonal $\ell$ such that for all diagonals $\ell_i$ in the initial maximal compatible set $T$, $\ell_i$ intersects with $\ell$ at the interior of $S$. Since the initial exchange matrix $B$ is of the tree $A_n$ type, two vertices of $S$ are shared by only one triangle in the initial triangulation $T$. A diagonal of $S$ combining these two vertices is the desired diagonal.
\end{proof}
\begin{lemma}\label{lem:description-positivecone}
Let $\Delta(\xx,B)$ be a cluster complex satisfying the assumption in Theorem \ref{thm:positive-simplex-A_n} and $x$ the (unique) non-linked vertex to $I$. Then, $\lk_{\Delta(\xx,B)}(x)$ is isomorphic to $\Delta(A_{n-1})$. Furthermore, $\st_{\Delta(\xx,B)}(x)$ is isomorphic to $\cone(\Delta(A_{n-1}))$.
\end{lemma}
\begin{proof}
We show the first statement. By Lemma \ref{lem:fullsub} and discussion of the proof of Lemma \ref{lem:bijection-between-B-and-Bx}, it suffices to show that $x$ has the following property.
\begin{itemize}
\item We fix a non-labeled seed $[(\xx_t,B_t)]$ containing $x$. Then, $B_{\setminus\{x\}}$ (this notation is defined in Lemma \ref{lem:bijection-between-B-and-Bx}) is mutation equivalent to a matrix of tree $A_{n-1}$ type.
\end{itemize}
By the canonical bijection between cluster variables and diagonals of a regular $(n+3)$-gon $S$, it suffices to show that a diagonal $\ell_x$, which corresponds to $x$, cut out a triangle from $S$. From the assumption in $B$, a vertex $v$ of $S$ is shared by all the diagonals in the initial maximal compatible set $T$. Therefore, $\ell_x$ is a diagonal that combines both neighbor vertices of $v$ by the proof of Lemma \ref{lem:unique-interior}. Therefore, $\ell_x$ cuts a triangle out of $S$. The second statement follows from the first and Lemma \ref{lem.st-lk-relation}.
\end{proof}

\begin{proof}[Proof of Theorem \ref{thm:positive-simplex-A_n}]
From Lemma \ref{lem:description-positivecone}, it suffices to show $\st_{\Delta(\xx,B)}(x)=\Delta^+(\xx,B)$. First, we prove $\st_{\Delta(\xx,B)}(x)\subset\Delta^+(\xx,B)$. Of a non-linked vertex, all vertices of $\st_{\Delta(\xx,B)}(x)$ are in $\Delta^+(\xx,B)$. Since $\Delta^+(\xx,B)$ is a full subcomplex of $\Delta(\xx,B)$, we have the desired inclusion. Second, let us prove that $\Delta^+(\xx,B)\subset\st_{\Delta(\xx,B)}(x)$. It suffices to show that for any simplex $F$ of $\Delta^+(\xx,B)$, $F\cup\{x\}$ is a simplex of $\Delta(\xx,B)$. We assume that there exists a simplex $F$ of $\Delta^+(\xx,B)$ such that $F\cup\{x\}$ is not a simplex of $\Delta(\xx,B)$. From the canonical bijection between cluster variables and diagonals of a regular $(n+3)$-gon $S$, we can assume that a compatible set $T_F$ corresponding to $F$ contains $\ell'$ intersecting with the corresponding diagonal $\ell_x$ of $x$ at the interior of $S$. However, all diagonals intersecting with $\ell_x$ at the interior of $S$ are in a diagonal set $T$ corresponding to the initial cluster $\xx$. This contradicts the fact that $F$ is a simplex of $\Delta^+(\xx,B)$.
\end{proof}
By using Theorem \ref{thm:positive-simplex-A_n}, we can calculate the number of simplices in $\Delta^+(\xx,B)$ using  $f(\Delta(A_{n-1}))$. We provide a conceptual proof of the following corollary given by Chapoton. 
\begin{corollary}[\cite{cha04}*{(34)}]\label{cor:positive-face-from-all-face-n-1}
Let $B$ be an exchange matrix the quiver $Q_B$ of which has a diagram of $A_n$ type in Figure \ref{fig:valued-quiver} as its underlying graph. Let $f_{n,k}$ be the number of $k$-dimensional simplices of $\Delta(A_{n})$, and $f^+_{n,k}$ the number of $k$-dimensional simplices of $\Delta^+(\xx,B)$. Then, we have
\begin{align}\label{eq:An-positive-facevector}
    f^+_{n,k}=
    f_{n-1,k}+f_{n-1,k-1}=\dfrac{1}{k+2}\begin{pmatrix}n\\ k+1\end{pmatrix}\begin{pmatrix}n+k+1\\ k+1\end{pmatrix}.
\end{align}
\end{corollary}
\begin{proof}
By Theorem \ref{thm:positive-simplex-A_n} and Corollary \ref{cor:independent-orientation}, we have
\begin{align}\label{facevectorAn}
f(\Delta^+(\xx,B))=f(\Delta(A_{n-1}))+[f(\Delta(A_{n-1}))]_1.
\end{align}
By \cite{cha04}*{(33)}, we have
\begin{align}\label{eq:formula-faces-A_n}
    f_{n,k}=\dfrac{1}{k+2}\begin{pmatrix}n\\ k+1\end{pmatrix}\begin{pmatrix}n+k+3\\ k+1\end{pmatrix}.
\end{align}
By \eqref{facevectorAn} and \eqref{eq:formula-faces-A_n}, we have the desired equation \eqref{eq:An-positive-facevector}.
\end{proof}
\begin{remark}
In \cite{cha04}*{(34)}, the formula is given more generally. The equation \eqref{eq:An-positive-facevector} corresponds with \cite{cha04}*{(34)} in the case that $\ell=0$.
\end{remark}
\begin{example}
Let $(\xx,B)$ be a seed which is the same as in Example \ref{ex:correspondence-arc-variable-A3}. Then, $\Delta^+(\xx,B)$ is a domain that is shown filled with gray in Figure \ref{hexagonpositivecomplex}.
 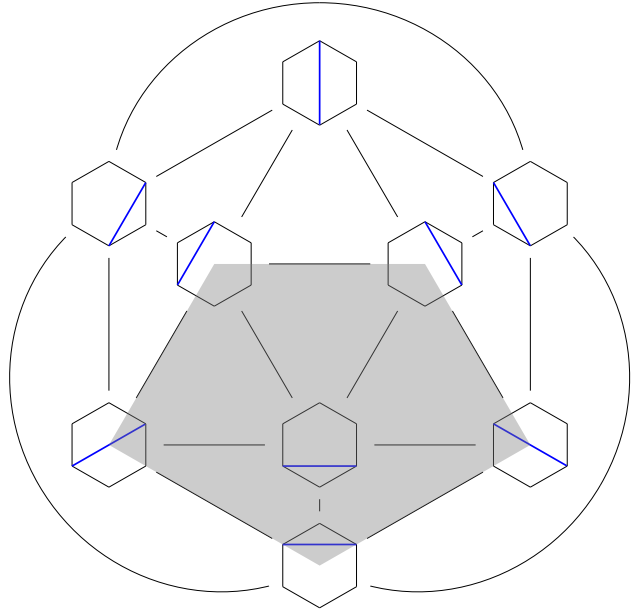
\begin{figure}[ht]
\caption{positive cluster complex of $A_3$ \label{hexagonpositivecomplex}}
\vspace{5pt}
\begin{center}
\scalebox{0.8}{
\begin{tikzpicture}
\coordinate (0) at (0,0);
\coordinate (u*) at (90:5.33);
\coordinate (u) at (90:4);
\coordinate (ul) at (150:4);
\coordinate (ur) at (30:4);
\coordinate (uml) at (150:2);
\coordinate (umr) at (30:2);
\coordinate (dmc) at (-90:2);
\coordinate (dl) at (-150:4);
 \coordinate (dl*) at (-150:5.33);
\coordinate (dr) at (-30:4);
\coordinate (dr*) at (-30:5.33);
\coordinate (d) at (-90:4);
\draw (u) to (ul);
\draw (u) to (ur);
\draw (ul) to (uml);
\draw (ur) to (umr);
\draw (umr) to (uml);
\draw (u) to (uml);
\draw (u) to (umr);
\draw (uml) to (dl);
\draw (umr) to (dr);
\draw (ul) to (dl);
\draw (ur) to (dr);
\draw (uml) to (dmc);
\draw (umr) to (dmc);
\draw (dl) to (dmc);
\draw (dr) to (dmc);
\draw (dl) to (d);
\draw (dmc) to (d);
\draw (dr) to (d);
\draw(ul) [out=90,in=180]to (u*);
\draw(ur) [out=90,in=0]to (u*);
\draw(ul) [out=210,in=120]to (dl*);
\draw(d) [out=210,in=300]to (dl*);
\draw(d) [out=330,in=240]to (dr*);
\draw(ur) [out=330,in=60]to (dr*);
\fill [white](u) circle (0.9cm);
\fill[white] (ul) circle (0.9cm);
\fill[white] (ur) circle (0.9cm);
\fill[white] (uml) circle (0.9cm);
\fill[white] (umr) circle (0.9cm);
\fill[white] (dmc) circle (0.9cm);
\fill[white] (dl) circle (0.9cm);
\fill[white] (dr) circle (0.9cm);
\fill[white] (d) circle (0.9cm);
\draw (u)++(210:0.7cm) to ++(90:0.7cm);
\draw (u)++(150:0.7cm) to ++(30:0.7cm);
\draw (u)++(90:0.7cm) to ++(330:0.7cm);
\draw (u)++(30:0.7cm) to ++(270:0.7cm);
\draw (u)++(330:0.7cm) to ++(210:0.7cm);
\draw (u)++(270:0.7cm) to ++(150:0.7cm);
\draw (ur)++(210:0.7cm) to ++(90:0.7cm);
\draw (ur)++(150:0.7cm) to ++(30:0.7cm);
\draw (ur)++(90:0.7cm) to ++(330:0.7cm);
\draw (ur)++(30:0.7cm) to ++(270:0.7cm);
\draw (ur)++(330:0.7cm) to ++(210:0.7cm);
\draw (ur)++(270:0.7cm) to ++(150:0.7cm);
\draw (ul)++(210:0.7cm) to ++(90:0.7cm);
\draw (ul)++(150:0.7cm) to ++(30:0.7cm);
\draw (ul)++(90:0.7cm) to ++(330:0.7cm);
\draw (ul)++(30:0.7cm) to ++(270:0.7cm);
\draw (ul)++(330:0.7cm) to ++(210:0.7cm);
\draw (ul)++(270:0.7cm) to ++(150:0.7cm);
\draw (umr)++(210:0.7cm) to ++(90:0.7cm);
\draw (umr)++(150:0.7cm) to ++(30:0.7cm);
\draw (umr)++(90:0.7cm) to ++(330:0.7cm);
\draw (umr)++(30:0.7cm) to ++(270:0.7cm);
\draw (umr)++(330:0.7cm) to ++(210:0.7cm);
\draw (umr)++(270:0.7cm) to ++(150:0.7cm);
\draw (uml)++(210:0.7cm) to ++(90:0.7cm);
\draw (uml)++(150:0.7cm) to ++(30:0.7cm);
\draw (uml)++(90:0.7cm) to ++(330:0.7cm);
\draw (uml)++(30:0.7cm) to ++(270:0.7cm);
\draw (uml)++(330:0.7cm) to ++(210:0.7cm);
\draw (uml)++(270:0.7cm) to ++(150:0.7cm);
\draw (dl)++(210:0.7cm) to ++(90:0.7cm);
\draw (dl)++(150:0.7cm) to ++(30:0.7cm);
\draw (dl)++(90:0.7cm) to ++(330:0.7cm);
\draw (dl)++(30:0.7cm) to ++(270:0.7cm);
\draw (dl)++(330:0.7cm) to ++(210:0.7cm);
\draw (dl)++(270:0.7cm) to ++(150:0.7cm);
\draw (dr)++(210:0.7cm) to ++(90:0.7cm);
\draw (dr)++(150:0.7cm) to ++(30:0.7cm);
\draw (dr)++(90:0.7cm) to ++(330:0.7cm);
\draw (dr)++(30:0.7cm) to ++(270:0.7cm);
\draw (dr)++(330:0.7cm) to ++(210:0.7cm);
\draw (dr)++(270:0.7cm) to ++(150:0.7cm);
\draw (dmc)++(210:0.7cm) to ++(90:0.7cm);
\draw (dmc)++(150:0.7cm) to ++(30:0.7cm);
\draw (dmc)++(90:0.7cm) to ++(330:0.7cm);
\draw (dmc)++(30:0.7cm) to ++(270:0.7cm);
\draw (dmc)++(330:0.7cm) to ++(210:0.7cm);
\draw (dmc)++(270:0.7cm) to ++(150:0.7cm);
\draw (d)++(210:0.7cm) to ++(90:0.7cm);
\draw (d)++(150:0.7cm) to ++(30:0.7cm);
\draw (d)++(90:0.7cm) to ++(330:0.7cm);
\draw (d)++(30:0.7cm) to ++(270:0.7cm);
\draw (d)++(330:0.7cm) to ++(210:0.7cm);
\draw (d)++(270:0.7cm) to ++(150:0.7cm);
\draw [thick, blue](uml)++(90:0.7cm) to ++(240:1.21cm);
\draw [thick, blue](umr)++(90:0.7cm) to ++(300:1.21cm);
\draw [thick, blue](dmc)++(210:0.7cm) to ++(0:1.21cm);
\draw [thick, blue](dl)++(210:0.7cm) to ++(30:1.39cm);
\draw [thick, blue](dr)++(150:0.7cm) to ++(330:1.39cm);
\draw [thick, blue](u)++(90:0.7cm) to ++(270:1.39cm);
\draw [thick, blue](ul)++(270:0.7cm) to ++(60:1.21cm);
\draw [thick, blue](d)++(150:0.7cm) to ++(0:1.21cm);
\draw [thick, blue](ur)++(150:0.7cm) to ++(300:1.21cm);
\fill [gray, opacity=.4] (uml)--(dl)--(d)--(dr)--(umr)--(uml);
\end{tikzpicture}}
\end{center}
\end{figure}
\end{example}
\subsection{Type $B_n$ and $C_n$}\label{subsec:focus-Bn-Cn}
We focus on cluster complexes of $B_n$ and $C_n$ type, that is, cluster complexes $\Delta(\xx,B)$ such that $B$ is mutation equivalent to an exchange matrix of tree $B_n$ or $C_n$ type in this subsection. In particular, we consider the case in which $B$ is an exchange matrix of tree $B_n$ or $C_n$ (i.e., $Q_B$ has a diagram of $B_n$ or $C_n$ type as shown in Figure \ref{fig:valued-quiver} as its underlying diagram).

First, we explain the realization of the cluster structure by marked surfaces, which is given by \cite{fzii}*{Subsection 12.3}.
Let $S$ be a regular $(2n+2)$-gon. We consider orbits of diagonals by a $\pi$ rotation $\Theta$. We say that, for any orbits $\bar{\ell}$ and $\bar{\ell}'$, $\bar{\ell}$ is \emph{compatible} with $\bar{\ell}'$ if $\bar{\ell}$ and $\bar{\ell}'$ do not cross at the interior of $S$, and the set of pairwise compatible orbits of diagonals is a \emph{compatible set}. We note that a maximal compatible set $T$ consists of $n-1$ orbits composed of two diagonals and a single diameter. We label these orbits by $(\bar{\ell}_1,\dots,\bar{\ell}_n)$. We call $T$ the \emph{labeled maximal compatible set} as in the case with $A_n$ type. We define a matrix $B_T$ as a matrix obtained from $T$ in the following manner. We assume that $\bar{\ell}_n$ is a diameter in $T$. Let $S_{\ell_n}$ be one of $(n+2)$-gon cut out by a diameter $\ell_n$. For each triangle $\Delta$ in a triangulation of $S_{\ell_n}$ given by $T$, we define the $n\times n$ integer matrix $B^\Delta=(b^\Delta_{ij})$ by setting
\begin{align*}
b^\Delta_{ij} =\begin{cases}
 1 \quad&\text{for any $i\in[1,n],j\in[1,n-1]$, if $\Delta$ has diagonals labeled $\bar{\ell}_{i}$ and  $\bar{\ell}_{j}$,  }\\
 &\text{with $\bar{\ell}_j$ following $\bar{\ell}_i$ in the clockwise order},\\
 -1 \quad &\text{if the same holds as above, with the counterclockwise order,}\\
 2 \quad&\text{for any $i\in[1,n-1]$, if $j=n$ and $\Delta$ has diagonals labeled $\bar{\ell}_{i}$ and  $\bar{\ell}_{n}$,}\\
 &\text{with $\bar{\ell}_n$ following $\bar{\ell}_i$ in the clockwise order},\\
 -2 \quad &\text{if the same holds as the above, with the counterclockwise order}\\
 0 \quad &\text{otherwise},
\end{cases}
\end{align*}
where $[1,n]$ means $\{1,\dots,n\}$.
The matrix $B=B_T=(b_{ij})$ is then defined by
\begin{align*}
B=\sum_\Delta B^\Delta,
\end{align*}
where the summation runs over all triangles $\Delta$ in a triangulation of $S_{\ell_n}$ given by $T$.
\begin{example}
Let $T=(\bar{\ell_1},\bar{\ell_2},\bar{\ell_3})$ be a labeled maximal compatible set given as in Figure \ref{fig:triangulation-matrix-correspondence-Bn}. Then, we have $B_T=\begin{bmatrix}
0&-1&2\\
1&0&-2\\
-1&1&0
\end{bmatrix}$ and the corresponding valued quiver is $\begin{xy}(0,0)*+{1}="I",(10,0)*+{2}="J", (20,0)*+{{3}}="K" \ar@{->}^{(1,1)}"J";"I" \ar@{->}^{(1,2)}"K";"J" \ar@/^4mm/^{(2,1)}"I";"K"\end{xy}$ and $d(1)=d(2)=1,d(3)=2$.
\begin{figure}[ht]
\caption{Labeled maximal compatible set of an octagon}\label{fig:triangulation-matrix-correspondence-Bn}
\[
\begin{tikzpicture}[baseline=0cm]
\coordinate (0) at (0,0);
\coordinate (1) at (22.5:1.5);
\coordinate (2) at (67.5:1.5);
\coordinate (3) at (112.5:1.5);
\coordinate (4) at (157.5:1.5);
\coordinate (5) at (202.5:1.5);
\coordinate (6) at (247.5:1.5);
\coordinate (7) at (292.5:1.5);
\coordinate (8) at (337.5:1.5);
\draw (1) to (2);
\draw (2) to (3);
\draw (3) to (4);
\draw (4) to (5);
\draw (5) to (6);
\draw (6) to (7);
\draw (7) to (8);
\draw (8) to (1);
\draw [thick, blue] (1) to (3);
\draw [thick, blue] (3) to (5);
\draw [thick, blue](1) to (5);
\draw [blue,dashed] (5) to (7);
\draw [blue,dashed] (7) to (1);
\fill [white](67.5:1) circle (3mm);
\fill [white](157.5:1) circle (3mm);
\fill [white](0,0) circle (3mm);
\node at (157.5:1) {$\bar{\ell}_1$};
\node at (67.5:1) {$\bar{\ell}_2$};
\node at (0,0) {$\bar{\ell}_3$};
\end{tikzpicture}
\]
\end{figure}
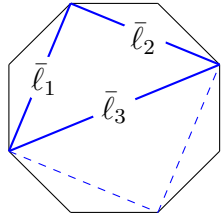
\end{example}
For $i\in \{1,\dots,n\}$ and a maximal compatible set $T$, we define a \emph{flip $\varphi_i(T)$ in direction $i$} as a transformation that obtains the unique maximal compatible set $T'$ from $T$ by removing $\bar{\ell}_i$ and adding the other diagonal $\bar{\ell}'_i$ to $T\setminus\{\bar{\ell}_i\}$.
We explain that flips establish the same structure as seed mutations of $B_n$ type, and define a \emph{labeled triangulation pair} $(T,B_T)$, a \emph{flip of labeled triangulation pair} $\varphi_i(T,B_T):=(\varphi_i(T),B_{\varphi_i(T)})$, a \emph{triangulation pattern} $P\colon t\mapsto \Xi_t$, a \emph{non-labeled triangulation pair} $[\Xi]$, and an \emph{arc complex} $\Delta(T,B_T)$ in the same manner as in Subsection \ref{subsec:focus-An}. In parallel with $A_n$ type, we have the following property by \cite{fzii}*{Subsection 12.3} and checking exchange relations \cite{fzii}*{(12.6),(12.7),(12.8)}.
\begin{theorem}\label{thm:arc-variable-Bn}
Let $S$ be a regular $(2n+2)$-gon and $T$ be a maximal compatible set. Let $P^{\Xi}_{(T,B_T)}$ be a triangulation pattern and we consider a cluster pattern $P^{\Sigma}_{(\xx,B_T)}$ the initial exchange matrix of which is $B_T$.
\begin{itemize}
\item[(1)]We have $B_{\varphi_i(T)}=\mu_i(B_T)$, where $\mu_i$ is a mutation in direction $i$ of an exchange matrix.
\item[(2)]We have the canonical bijection between the set of all orbits of diagonals of $S$ and the set of cluster variables in $\{\xx_t\}_{t\in\TT_n}$ by the correspondence $\bar{\ell}_{i;t}\mapsto x_{i;t}$. Furthermore, this bijection induces an isomorphism between an arc complex $\Delta(T,B_T)$ and a cluster complex $\Delta(\xx,B_T)$.
\end{itemize}
\end{theorem}
We note that $B_T$ is a skew-symmetrizable matrix, which is mutation equivalent to an exchange matrix of tree $B_n$ type.

To obtain the geometric realization of a cluster structure of $C_n$ type, we replace $b_{ij}^\Delta$ in the above discussion with

\begin{align*}
b^\Delta_{ij} =\begin{cases}
 1 \quad&\text{for any $i\in[1,n-1],j\in[1,n]$, if $\Delta$ has diagonals labeled $\bar{\ell}_{i}$ and  $\bar{\ell}_{j}$, }\\
 &\text{with $\bar{\ell}_j$ following $\bar{\ell}_i$ in the clockwise order},\\
 -1 \quad &\text{if the same holds as the above, with the counterclockwise order,}\\
 2 \quad&\text{for any $j\in[1,n-1]$, if $i=n$ and $\Delta$ has diagonals labeled $\bar{\ell}_{i}$ and  $\bar{\ell}_{n}$, }\\
 &\text{with $\bar{\ell}_n$ following $\bar{\ell}_i$ in the clockwise order},\\
 -2 \quad &\text{if the same holds as the above, with the counterclockwise order,}\\
 0 \quad &\text{otherwise}.
\end{cases}
\end{align*}
For a maximal compatible set $T$ of a $(2n+2)$-gon $S$, if $B$ is the corresponding matrix to $T$ of the $B_n$ type, then that of $C_n$ type is given by the transpose of $-B$.

By the construction method of arc complexes $B_n$ and $C_n$ type, we have $\Delta(B_{n})\cong \Delta(C_{n})$.

\begin{example}\label{ex:correspondence-arc-variable-B3}
 Let $B=\begin{bmatrix}
 0&-1&0\\
 1&0&-2\\
 0&1&0
 \end{bmatrix}$ of type $B_3$. The corresponding valued quiver is $\begin{xy}(0,0)*+{1}="I",(10,0)*+{2}="J", (20,0)*+{{3}}="K" \ar@{->}^{(1,1)}"J";"I" \ar@{->}^{(1,2)}"K";"J"\end{xy}$, and $d(1)=d(2)=1,d(3)=2$. 
 
 Let $T$ be a set
 $\left\{\bar{\ell}_1=\begin{tikzpicture}[baseline=-1mm,scale=0.3]
 \coordinate (u1) at(67.5:1); \coordinate (u2) at(112.5:1); \coordinate (lu) at(157.5:1);
 \coordinate (ld) at(-157.5:1); \coordinate (ru) at(22.5:1); \coordinate (rd) at(-22.5:1);  \coordinate (d1) at(-67.5:1); \coordinate (d2) at(-112.5:1);
 \draw (u1)--(u2)--(lu)--(ld)--(d2)--(d1)--(rd)--(ru)--(u1); \node at(0,1.3) {}; \node at(0,-1.3) {}; \draw[blue,thick] (u1)--(lu);
 \draw[blue,thick] (d2)--(rd);
\end{tikzpicture}, \ \bar{\ell}_2=\begin{tikzpicture}[baseline=-1mm,scale=0.3]
 \coordinate (u1) at(67.5:1); \coordinate (u2) at(112.5:1); \coordinate (lu) at(157.5:1);
 \coordinate (ld) at(-157.5:1); \coordinate (ru) at(22.5:1); \coordinate (rd) at(-22.5:1);  \coordinate (d1) at(-67.5:1); \coordinate (d2) at(-112.5:1);
 \draw (u1)--(u2)--(lu)--(ld)--(d2)--(d1)--(rd)--(ru)--(u1); \node at(0,1.3) {}; \node at(0,-1.3) {}; \draw[blue,thick] (u1)--(ld);
 \draw[blue,thick] (d2)--(ru);
\end{tikzpicture},\ \bar{\ell}_3=\begin{tikzpicture}[baseline=-1mm,scale=0.3]
 \coordinate (u1) at(67.5:1); \coordinate (u2) at(112.5:1); \coordinate (lu) at(157.5:1);
 \coordinate (ld) at(-157.5:1); \coordinate (ru) at(22.5:1); \coordinate (rd) at(-22.5:1);  \coordinate (d1) at(-67.5:1); \coordinate (d2) at(-112.5:1);
 \draw (u1)--(u2)--(lu)--(ld)--(d2)--(d1)--(rd)--(ru)--(u1); \node at(0,1.3) {}; \node at(0,-1.3) {}; \draw[blue,thick] (u1)--(d2);
\end{tikzpicture}
\right\}$ of an octagon. Then, we have $B=B_T$. For example, we consider a cluster pattern $P^{\Sigma}(\xx,B)$ and a triangulation pattern $P^\Xi(T,B_T)$. Then, we have $\mu_{x_3}\left[\left(x_1,x_2,x_3\right)\right]=\left[\left(x_1,x_2,\dfrac{x_2^2+1}{x_3}\right)\right]$ by a mutation in direction $x_3$. In contrast, we have $\varphi_{\ell_3}\left(
\begin{tikzpicture}[baseline=-1mm,scale=0.3]
 \coordinate (u1) at(67.5:1); \coordinate (u2) at(112.5:1); \coordinate (lu) at(157.5:1);
 \coordinate (ld) at(-157.5:1); \coordinate (ru) at(22.5:1); \coordinate (rd) at(-22.5:1);  \coordinate (d1) at(-67.5:1); \coordinate (d2) at(-112.5:1);
 \draw (u1)--(u2)--(lu)--(ld)--(d2)--(d1)--(rd)--(ru)--(u1); \node at(0,1.3) {}; \node at(0,-1.3) {}; \draw[blue,thick] (u1)--(d2); \draw[blue,thick] (u1)--(lu);  \draw[blue,thick] (u1)--(ld); \draw[blue,thick] (ru)--(d2); \draw[blue,thick] (rd)--(d2);
\end{tikzpicture}\right)=
\begin{tikzpicture}[baseline=-1mm,scale=0.3]
 \coordinate (u1) at(67.5:1); \coordinate (u2) at(112.5:1); \coordinate (lu) at(157.5:1);
 \coordinate (ld) at(-157.5:1); \coordinate (ru) at(22.5:1); \coordinate (rd) at(-22.5:1);  \coordinate (d1) at(-67.5:1); \coordinate (d2) at(-112.5:1);
 \draw (u1)--(u2)--(lu)--(ld)--(d2)--(d1)--(rd)--(ru)--(u1); \node at(0,1.3) {}; \node at(0,-1.3) {}; \draw[blue,thick] (ru)--(ld); \draw[blue,thick] (u1)--(lu);  \draw[blue,thick] (u1)--(ld); \draw[blue,thick] (ru)--(d2); \draw[blue,thick] (rd)--(d2);
\end{tikzpicture}
$ by a flip in direction $\ell_3$. Therefore, $\dfrac{x_2^2+1}{x_3}$ corresponds to \begin{tikzpicture}[baseline=-1mm,scale=0.3]
 \coordinate (u1) at(67.5:1); \coordinate (u2) at(112.5:1); \coordinate (lu) at(157.5:1);
 \coordinate (ld) at(-157.5:1); \coordinate (ru) at(22.5:1); \coordinate (rd) at(-22.5:1);  \coordinate (d1) at(-67.5:1); \coordinate (d2) at(-112.5:1);
 \draw (u1)--(u2)--(lu)--(ld)--(d2)--(d1)--(rd)--(ru)--(u1); \node at(0,1.3) {}; \node at(0,-1.3) {}; \draw[blue,thick] (ru)--(ld);
\end{tikzpicture} by the canonical bijection.
The canonical bijection between the set of cluster variables and the orbit set is given in Table \ref{exampB3}, and the cluster complex of $B_3$ type is given in Figure \ref{B3complex}.
\begin{table}[ht]
\begin{tabular}{c|c}
 Cluster variable $x$ & diagonal corresponding to $x$
\end{tabular}
\vspace{2mm}\\
\begin{minipage}{0.3\hsize}
\begin{center}\begin{tabular}{c|c}
 {\small $x_1$} &
\begin{tikzpicture}[baseline=-1mm,scale=0.5]
  \coordinate (u1) at(67.5:1); \coordinate (u2) at(112.5:1); \coordinate (lu) at(157.5:1);
 \coordinate (ld) at(-157.5:1); \coordinate (ru) at(22.5:1); \coordinate (rd) at(-22.5:1);  \coordinate (d1) at(-67.5:1); \coordinate (d2) at(-112.5:1);
 \draw (u1)--(u2)--(lu)--(ld)--(d2)--(d1)--(rd)--(ru)--(u1); \node at(0,1.3) {}; \node at(0,-1.3) {};  \draw[blue,thick] (u1)--(lu);  \draw[blue,thick] (rd)--(d2);
\end{tikzpicture}
\\\hline
 {\small$x_2$} &
\begin{tikzpicture}[baseline=-1mm,scale=0.5]
  \coordinate (u1) at(67.5:1); \coordinate (u2) at(112.5:1); \coordinate (lu) at(157.5:1);
 \coordinate (ld) at(-157.5:1); \coordinate (ru) at(22.5:1); \coordinate (rd) at(-22.5:1);  \coordinate (d1) at(-67.5:1); \coordinate (d2) at(-112.5:1);
 \draw (u1)--(u2)--(lu)--(ld)--(d2)--(d1)--(rd)--(ru)--(u1); \node at(0,1.3) {}; \node at(0,-1.3) {}; \draw[blue,thick] (u1)--(ld);  \draw[blue,thick] (d2)--(ru);
\end{tikzpicture}
\\\hline
 {\small$x_3$} &
 \begin{tikzpicture}[baseline=-1mm,scale=0.5]
 \coordinate (u1) at(67.5:1); \coordinate (u2) at(112.5:1); \coordinate (lu) at(157.5:1);
 \coordinate (ld) at(-157.5:1); \coordinate (ru) at(22.5:1); \coordinate (rd) at(-22.5:1);  \coordinate (d1) at(-67.5:1); \coordinate (d2) at(-112.5:1);
 \draw (u1)--(u2)--(lu)--(ld)--(d2)--(d1)--(rd)--(ru)--(u1); \node at(0,1.3) {}; \node at(0,-1.3) {}; \draw[blue,thick] (u1)--(d2);
  \end{tikzpicture}
\end{tabular}\end{center}
\end{minipage}
\begin{minipage}{0.3\hsize}
\begin{center}\begin{tabular}{c|c}
 {\small$\cfrac{x_2+1}{x_1}$} &
\begin{tikzpicture}[baseline=-1mm,scale=0.5]
  \coordinate (u1) at(67.5:1); \coordinate (u2) at(112.5:1); \coordinate (lu) at(157.5:1);
 \coordinate (ld) at(-157.5:1); \coordinate (ru) at(22.5:1); \coordinate (rd) at(-22.5:1);  \coordinate (d1) at(-67.5:1); \coordinate (d2) at(-112.5:1);
 \draw (u1)--(u2)--(lu)--(ld)--(d2)--(d1)--(rd)--(ru)--(u1); \node at(0,1.3) {}; \node at(0,-1.3) {}; \draw[blue,thick] (u2)--(ld);  \draw[blue,thick] (d1)--(ru);
\end{tikzpicture}
\\\hline
 {\small$\cfrac{x_1+x_3}{x_2}$} &
\begin{tikzpicture}[baseline=-1mm,scale=0.5]
  \coordinate (u1) at(67.5:1); \coordinate (u2) at(112.5:1); \coordinate (lu) at(157.5:1);
 \coordinate (ld) at(-157.5:1); \coordinate (ru) at(22.5:1); \coordinate (rd) at(-22.5:1);  \coordinate (d1) at(-67.5:1); \coordinate (d2) at(-112.5:1);
 \draw (u1)--(u2)--(lu)--(ld)--(d2)--(d1)--(rd)--(ru)--(u1); \node at(0,1.3) {}; \node at(0,-1.3) {};
 \draw[blue,thick] (lu)--(d2); \draw[blue,thick] (u1)--(rd);
\end{tikzpicture}
\\\hline
 {\small$\cfrac{x_2^2+1}{x_3}$} &
\begin{tikzpicture}[baseline=-1mm,scale=0.5]
  \coordinate (u1) at(67.5:1); \coordinate (u2) at(112.5:1); \coordinate (lu) at(157.5:1);
 \coordinate (ld) at(-157.5:1); \coordinate (ru) at(22.5:1); \coordinate (rd) at(-22.5:1);  \coordinate (d1) at(-67.5:1); \coordinate (d2) at(-112.5:1);
 \draw (u1)--(u2)--(lu)--(ld)--(d2)--(d1)--(rd)--(ru)--(u1); \node at(0,1.3) {}; \node at(0,-1.3) {};
 \draw[blue,thick] (ld)--(ru);
\end{tikzpicture}
\end{tabular}\end{center}
\end{minipage}
\begin{tabular}{c|c}
 {\small$\cfrac{x_1+x_3+x_2x_3}{x_1x_2}$} &
\begin{tikzpicture}[baseline=-1mm,scale=0.5]
  \coordinate (u1) at(67.5:1); \coordinate (u2) at(112.5:1); \coordinate (lu) at(157.5:1);
 \coordinate (ld) at(-157.5:1); \coordinate (ru) at(22.5:1); \coordinate (rd) at(-22.5:1);  \coordinate (d1) at(-67.5:1); \coordinate (d2) at(-112.5:1);
 \draw (u1)--(u2)--(lu)--(ld)--(d2)--(d1)--(rd)--(ru)--(u1); \node at(0,1.3) {}; \node at(0,-1.3) {}; \draw[blue,thick] (u1)--(d1); \draw[blue,thick] (u2)--(d2);
\end{tikzpicture}
\\\hline
 {\small$\cfrac{x_1+x_3+x_1x_2^2}{x_2x_3}$} &
\begin{tikzpicture}[baseline=-1mm,scale=0.5]
  \coordinate (u1) at(67.5:1); \coordinate (u2) at(112.5:1); \coordinate (lu) at(157.5:1);
 \coordinate (ld) at(-157.5:1); \coordinate (ru) at(22.5:1); \coordinate (rd) at(-22.5:1);  \coordinate (d1) at(-67.5:1); \coordinate (d2) at(-112.5:1);
 \draw (u1)--(u2)--(lu)--(ld)--(d2)--(d1)--(rd)--(ru)--(u1); \node at(0,1.3) {}; \node at(0,-1.3) {}; \draw[blue,thick] (lu)--(ru); \draw[blue,thick] (ld)--(rd);
\end{tikzpicture}
\\\hline
  {\small$\cfrac{x_1^2x_2^2 + x_1^2 + 2x_1x_3 + x_3^2}{x_2^2x_3}$} &
\begin{tikzpicture}[baseline=-1mm,scale=0.5]
  \coordinate (u1) at(67.5:1); \coordinate (u2) at(112.5:1); \coordinate (lu) at(157.5:1);
 \coordinate (ld) at(-157.5:1); \coordinate (ru) at(22.5:1); \coordinate (rd) at(-22.5:1);  \coordinate (d1) at(-67.5:1); \coordinate (d2) at(-112.5:1);
 \draw (u1)--(u2)--(lu)--(ld)--(d2)--(d1)--(rd)--(ru)--(u1); \node at(0,1.3) {}; \node at(0,-1.3) {}; \draw[blue,thick] (rd)--(lu);
\end{tikzpicture}
  \end{tabular}\\
 \vspace{5mm}
 \begin{tabular}{c|c}
 {\small$\cfrac{x_1+x_3+x_1x_2^2+x_2x_3}{x_1x_2x_3}$} &
\begin{tikzpicture}[baseline=-1mm,scale=0.5]
 \coordinate (u1) at(67.5:1); \coordinate (u2) at(112.5:1); \coordinate (lu) at(157.5:1);
 \coordinate (ld) at(-157.5:1); \coordinate (ru) at(22.5:1); \coordinate (rd) at(-22.5:1);  \coordinate (d1) at(-67.5:1); \coordinate (d2) at(-112.5:1);
 \draw (u1)--(u2)--(lu)--(ld)--(d2)--(d1)--(rd)--(ru)--(u1); \node at(0,1.3) {}; \node at(0,-1.3) {}; \draw[blue,thick] (u2)--(ru); \draw[blue,thick] (ld)--(d1);
\end{tikzpicture}
\\\hline
 {\small$\cfrac{x_1^2x_2^2 + x_1x_2x_3 + x_2x_3^2 + x_1^2 + 2x_1x_3 + x_3^2}{x_1x_2^2x_3}$} &
\begin{tikzpicture}[baseline=-1mm,scale=0.5]
  \coordinate (u1) at(67.5:1); \coordinate (u2) at(112.5:1); \coordinate (lu) at(157.5:1);
 \coordinate (ld) at(-157.5:1); \coordinate (ru) at(22.5:1); \coordinate (rd) at(-22.5:1);  \coordinate (d1) at(-67.5:1); \coordinate (d2) at(-112.5:1);
 \draw (u1)--(u2)--(lu)--(ld)--(d2)--(d1)--(rd)--(ru)--(u1); \node at(0,1.3) {}; \node at(0,-1.3) {}; \draw[blue,thick] (lu)--(d1); \draw[blue,thick] (u2)--(rd);
\end{tikzpicture}
\\\hline
 {\small$\cfrac{x_1^2x_2^2 + x_2^2x_3^2 + 2x_1x_2x_3 + 2x_2x_3^2 + x_1^2 + 2x_1x_3 + x_3^2}{x_1^2x_2^2x_3}$} &
\begin{tikzpicture}[baseline=-1mm,scale=0.5]
 \coordinate (u1) at(67.5:1); \coordinate (u2) at(112.5:1); \coordinate (lu) at(157.5:1);
 \coordinate (ld) at(-157.5:1); \coordinate (ru) at(22.5:1); \coordinate (rd) at(-22.5:1);  \coordinate (d1) at(-67.5:1); \coordinate (d2) at(-112.5:1);
 \draw (u1)--(u2)--(lu)--(ld)--(d2)--(d1)--(rd)--(ru)--(u1); \node at(0,1.3) {}; \node at(0,-1.3) {}; \draw[blue,thick] (u2)--(d1);
\end{tikzpicture}
  \end{tabular}
 \vspace{5mm}
 \caption{Canonical bijection between the set of cluster variables and the set of orbits of diagonals of $B_3$ type}\label{exampB3}
\end{table}
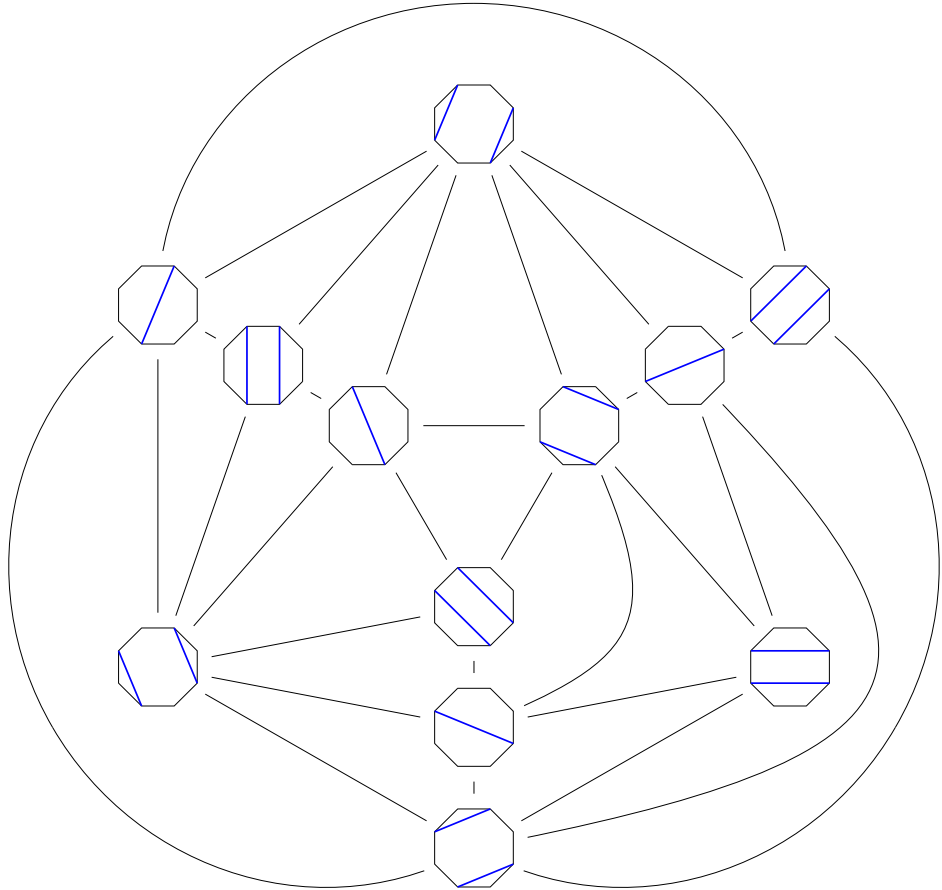
\begin{figure}[ht]
\caption{Cluster complex of $B_3$ type\label{B3complex}}
\hspace{5pt}
\begin{center}
\scalebox{0.8}{
\begin{tikzpicture}
\coordinate (0) at (0,0);
\coordinate (u*) at (90:8);
\coordinate (u) at (90:6);
\coordinate (ul) at (150:4);
\coordinate (ur) at (30:4);
\coordinate (uml) at (150:2);
\coordinate (umr) at (30:2);
\coordinate (dmc) at (-90:2);
\coordinate (dl) at (-150:6);
 \coordinate (dl*) at (-150:8);
\coordinate (dr) at (-30:6);
\coordinate (dr*) at (-30:8);
\coordinate (d) at (-90:4);
\coordinate (ll) at (150:6);
\coordinate (rr) at (30:6);
\coordinate (dd) at (-90:6);
\draw (u) to (ul);
\draw (u) to (ur);
\draw (ul) to (uml);
\draw (ur) to (umr);
\draw (umr) to (uml);
\draw (u) to (uml);
\draw (u) to (umr);
\draw (uml) to (dl);
\draw (umr) to (dr);
\draw (ul) to (dl);
\draw (ur) to (dr);
\draw (uml) to (dmc);
\draw (umr) to (dmc);
\draw (dl) to (dmc);
\draw (umr) [out=290,in=45] .. controls (-35:4) and (-45:4) .. (d);
\draw (ur) [out=290,in=45] .. controls (-20:9) and (-30:9) .. (dd);
\draw (dl) to (d);
\draw (dmc) to (d);
\draw (dr) to (d);
\draw (dd) to (d);
\draw (dd) to (dl);
\draw (dd) to (dr);
\draw (ll) to (ul);
\draw (ll) to (dl);
\draw (ll) to (u);
\draw (rr) to (ur);
\draw (rr) to (u);
\draw(ll) [out=90,in=180]to (u*);
\draw(rr) [out=90,in=0]to (u*);
\draw(ll) [out=210,in=120]to (dl*);
\draw(dd) [out=210,in=300]to (dl*);
\draw(dd) [out=330,in=240]to (dr*);
\draw(rr) [out=330,in=60]to (dr*);
\fill [white](u) circle (0.9cm);
\fill[white] (ul) circle (0.9cm);
\fill[white] (ur) circle (0.9cm);
\fill[white] (uml) circle (0.9cm);
\fill[white] (umr) circle (0.9cm);
\fill[white] (dmc) circle (0.9cm);
\fill[white] (dl) circle (0.9cm);
\fill[white] (dr) circle (0.9cm);
\fill[white] (d) circle (0.9cm);
\fill[white] (ll) circle (0.9cm);
\fill[white] (rr) circle (0.9cm);
\fill[white] (dd) circle (0.9cm);
\draw (u)++(22.5:0.7cm) to ++(135:0.54cm);
\draw (u)++(67.5:0.7cm) to ++(180:0.54cm);
\draw (u)++(112.5:0.7cm) to ++(225:0.54cm);
\draw (u)++(157.5:0.7cm) to ++(270:0.54cm);
\draw (u)++(202.5:0.7cm) to ++(315:0.54cm);
\draw (u)++(247.5:0.7cm) to ++(0:0.54cm);
\draw (u)++(292.5:0.7cm) to ++(45:0.54cm);
\draw (u)++(337.5:0.7cm) to ++(90:0.54cm);
\draw (ur)++(22.5:0.7cm) to ++(135:0.54cm);
\draw (ur)++(67.5:0.7cm) to ++(180:0.54cm);
\draw (ur)++(112.5:0.7cm) to ++(225:0.54cm);
\draw (ur)++(157.5:0.7cm) to ++(270:0.54cm);
\draw (ur)++(202.5:0.7cm) to ++(315:0.54cm);
\draw (ur)++(247.5:0.7cm) to ++(0:0.54cm);
\draw (ur)++(292.5:0.7cm) to ++(45:0.54cm);
\draw (ur)++(337.5:0.7cm) to ++(90:0.54cm);
\draw (ul)++(22.5:0.7cm) to ++(135:0.54cm);
\draw (ul)++(67.5:0.7cm) to ++(180:0.54cm);
\draw (ul)++(112.5:0.7cm) to ++(225:0.54cm);
\draw (ul)++(157.5:0.7cm) to ++(270:0.54cm);
\draw (ul)++(202.5:0.7cm) to ++(315:0.54cm);
\draw (ul)++(247.5:0.7cm) to ++(0:0.54cm);
\draw (ul)++(292.5:0.7cm) to ++(45:0.54cm);
\draw (ul)++(337.5:0.7cm) to ++(90:0.54cm);
\draw (umr)++(22.5:0.7cm) to ++(135:0.54cm);
\draw (umr)++(67.5:0.7cm) to ++(180:0.54cm);
\draw (umr)++(112.5:0.7cm) to ++(225:0.54cm);
\draw (umr)++(157.5:0.7cm) to ++(270:0.54cm);
\draw (umr)++(202.5:0.7cm) to ++(315:0.54cm);
\draw (umr)++(247.5:0.7cm) to ++(0:0.54cm);
\draw (umr)++(292.5:0.7cm) to ++(45:0.54cm);
\draw (umr)++(337.5:0.7cm) to ++(90:0.54cm);
\draw (uml)++(22.5:0.7cm) to ++(135:0.54cm);
\draw (uml)++(67.5:0.7cm) to ++(180:0.54cm);
\draw (uml)++(112.5:0.7cm) to ++(225:0.54cm);
\draw (uml)++(157.5:0.7cm) to ++(270:0.54cm);
\draw (uml)++(202.5:0.7cm) to ++(315:0.54cm);
\draw (uml)++(247.5:0.7cm) to ++(0:0.54cm);
\draw (uml)++(292.5:0.7cm) to ++(45:0.54cm);
\draw (uml)++(337.5:0.7cm) to ++(90:0.54cm);
\draw (dl)++(22.5:0.7cm) to ++(135:0.54cm);
\draw (dl)++(67.5:0.7cm) to ++(180:0.54cm);
\draw (dl)++(112.5:0.7cm) to ++(225:0.54cm);
\draw (dl)++(157.5:0.7cm) to ++(270:0.54cm);
\draw (dl)++(202.5:0.7cm) to ++(315:0.54cm);
\draw (dl)++(247.5:0.7cm) to ++(0:0.54cm);
\draw (dl)++(292.5:0.7cm) to ++(45:0.54cm);
\draw (dl)++(337.5:0.7cm) to ++(90:0.54cm);
\draw (dr)++(22.5:0.7cm) to ++(135:0.54cm);
\draw (dr)++(67.5:0.7cm) to ++(180:0.54cm);
\draw (dr)++(112.5:0.7cm) to ++(225:0.54cm);
\draw (dr)++(157.5:0.7cm) to ++(270:0.54cm);
\draw (dr)++(202.5:0.7cm) to ++(315:0.54cm);
\draw (dr)++(247.5:0.7cm) to ++(0:0.54cm);
\draw (dr)++(292.5:0.7cm) to ++(45:0.54cm);
\draw (dr)++(337.5:0.7cm) to ++(90:0.54cm);
\draw (dmc)++(22.5:0.7cm) to ++(135:0.54cm);
\draw (dmc)++(67.5:0.7cm) to ++(180:0.54cm);
\draw (dmc)++(112.5:0.7cm) to ++(225:0.54cm);
\draw (dmc)++(157.5:0.7cm) to ++(270:0.54cm);
\draw (dmc)++(202.5:0.7cm) to ++(315:0.54cm);
\draw (dmc)++(247.5:0.7cm) to ++(0:0.54cm);
\draw (dmc)++(292.5:0.7cm) to ++(45:0.54cm);
\draw (dmc)++(337.5:0.7cm) to ++(90:0.54cm);
\draw (d)++(22.5:0.7cm) to ++(135:0.54cm);
\draw (d)++(67.5:0.7cm) to ++(180:0.54cm);
\draw (d)++(112.5:0.7cm) to ++(225:0.54cm);
\draw (d)++(157.5:0.7cm) to ++(270:0.54cm);
\draw (d)++(202.5:0.7cm) to ++(315:0.54cm);
\draw (d)++(247.5:0.7cm) to ++(0:0.54cm);
\draw (d)++(292.5:0.7cm) to ++(45:0.54cm);
\draw (d)++(337.5:0.7cm) to ++(90:0.54cm);
\draw (ll)++(22.5:0.7cm) to ++(135:0.54cm);
\draw (ll)++(67.5:0.7cm) to ++(180:0.54cm);
\draw (ll)++(112.5:0.7cm) to ++(225:0.54cm);
\draw (ll)++(157.5:0.7cm) to ++(270:0.54cm);
\draw (ll)++(202.5:0.7cm) to ++(315:0.54cm);
\draw (ll)++(247.5:0.7cm) to ++(0:0.54cm);
\draw (ll)++(292.5:0.7cm) to ++(45:0.54cm);
\draw (ll)++(337.5:0.7cm) to ++(90:0.54cm);
\draw (rr)++(22.5:0.7cm) to ++(135:0.54cm);
\draw (rr)++(67.5:0.7cm) to ++(180:0.54cm);
\draw (rr)++(112.5:0.7cm) to ++(225:0.54cm);
\draw (rr)++(157.5:0.7cm) to ++(270:0.54cm);
\draw (rr)++(202.5:0.7cm) to ++(315:0.54cm);
\draw (rr)++(247.5:0.7cm) to ++(0:0.54cm);
\draw (rr)++(292.5:0.7cm) to ++(45:0.54cm);
\draw (rr)++(337.5:0.7cm) to ++(90:0.54cm);
\draw (dd)++(22.5:0.7cm) to ++(135:0.54cm);
\draw (dd)++(67.5:0.7cm) to ++(180:0.54cm);
\draw (dd)++(112.5:0.7cm) to ++(225:0.54cm);
\draw (dd)++(157.5:0.7cm) to ++(270:0.54cm);
\draw (dd)++(202.5:0.7cm) to ++(315:0.54cm);
\draw (dd)++(247.5:0.7cm) to ++(0:0.54cm);
\draw (dd)++(292.5:0.7cm) to ++(45:0.54cm);
\draw (dd)++(337.5:0.7cm) to ++(90:0.54cm);
\draw [thick, blue](ll)++(67.5:0.7cm) to ++(247.5:1.4cm);
\draw [thick, blue](dd)++(67.5:0.7cm) to ++(202.5:0.99cm);
\draw [thick, blue](dd)++(-22.5:0.7cm) to ++(202.5:0.99cm);
\draw [thick, blue](rr)++(67.5:0.7cm) to ++(225:1.29cm);
\draw [thick, blue](rr)++(22.5:0.7cm) to ++(225:1.29cm);
\draw [thick, blue](u)++(112.5:0.7cm) to ++(-112.5:0.99cm);
\draw [thick, blue](u)++(22.5:0.7cm) to ++(-112.5:0.99cm);
\draw [thick, blue](dr)++(22.5:0.7cm) to ++(180:1.29cm);
\draw [thick, blue](dr)++(-22.5:0.7cm) to ++(180:1.29cm);
\draw [thick, blue](dl)++(67.5:0.7cm) to ++(-67.5:0.99cm);
\draw [thick, blue](dl)++(157.5:0.7cm) to ++(-67.5:0.99cm);
\draw [thick, blue](ul)++(112.5:0.7cm) to ++(270:1.29cm);
\draw [thick, blue](ul)++(67.5:0.7cm) to ++(270:1.29cm);
\draw [thick, blue](ur)++(22.5:0.7cm) to ++(-157.5:1.4cm);
\draw [thick, blue](d)++(-22.5:0.7cm) to ++(157.5:1.4cm);
\draw [thick, blue](uml)++(112.5:0.7cm) to ++(-67.5:1.4cm);
\draw [thick, blue](dmc)++(112.5:0.7cm) to ++(-45:1.29cm);
\draw [thick, blue](dmc)++(157.5:0.7cm) to ++(-45:1.29cm);
\draw [thick, blue](umr)++(112.5:0.7cm) to ++(-22.5:0.99cm);
\draw [thick, blue](umr)++(202.5:0.7cm) to ++(-22.5:0.99cm);
\end{tikzpicture}}
\end{center}
\end{figure}
\end{example}
Here, we investigate a positive cluster complex of $B_n$ (and $C_n$) type. First, we give a generalization of the closed star.
 Let $K$ be a simplicial complex and $v$ a vertex of $K$. For any subset $V$ of a vertex set of $K$, we denote by
 \[
 \st_K(V)=\bigcup_{v\in V}\st_K(v).
 \]
 the union of closed stars of $v$. 

 We denote by $I$ the vertex set consisting of all cluster variables in the initial cluster $\xx$. The main theorem in this subsection is the following.
\begin{theorem}\label{thm:positive-simplex-B_n}
Let $\Delta(\xx,B)$ be a cluster complex of $B_n$ or $C_n$ type. Suppose that $Q_B$ has one of the following forms:
\begin{align}
\begin{xy}
    (50,0)*{\circ}="A",(60,0)*{\circ}="B",(70,0)*{\circ}="C",(80,0)*{\circ}="D",(85,0)*{}="E",(90,0)*{}="F",(95,0)*{\circ}="G",(105,0)*{\circ}="H"\ar@{->} "A";"B", \ar@{->} "B";"C", \ar@{->} "C";"D", \ar@{-} "D";"E", \ar@{.} "E";"F",\ar@{->} "F";"G",\ar@{->}^{(1,2)} "G";"H"
\end{xy}\label{assumption:Bn1}\\
\begin{xy}
    (50,0)*{\circ}="A",(60,0)*{\circ}="B",(70,0)*{\circ}="C",(80,0)*{\circ}="D",(85,0)*{}="E",(90,0)*{}="F",(95,0)*{\circ}="G",(105,0)*{\circ}="H"\ar@{<-} "A";"B", \ar@{<-} "B";"C", \ar@{<-} "C";"D", \ar@{<-} "D";"E", \ar@{.} "E";"F",\ar@{-} "F";"G",\ar@{<-}^{(2,1)} "G";"H"
\end{xy}\label{assumption:Bn2}\\
\begin{xy}
    (50,0)*{\circ}="A",(60,0)*{\circ}="B",(70,0)*{\circ}="C",(80,0)*{\circ}="D",(85,0)*{}="E",(90,0)*{}="F",(95,0)*{\circ}="G",(105,0)*{\circ}="H"\ar@{->} "A";"B", \ar@{->} "B";"C", \ar@{->} "C";"D", \ar@{-} "D";"E", \ar@{.} "E";"F",\ar@{->} "F";"G",\ar@{->}^{(2,1)} "G";"H"
\end{xy}\label{assumption:Cn1}\\
\begin{xy}
    (50,0)*{\circ}="A",(60,0)*{\circ}="B",(70,0)*{\circ}="C",(80,0)*{\circ}="D",(85,0)*{}="E",(90,0)*{}="F",(95,0)*{\circ}="G",(105,0)*{\circ}="H"\ar@{<-} "A";"B", \ar@{<-} "B";"C", \ar@{<-} "C";"D", \ar@{<-} "D";"E", \ar@{.} "E";"F",\ar@{-} "F";"G",\ar@{<-}^{(1,2)} "G";"H"
\end{xy}\label{assumption:Cn2}
\end{align}
Then, $\Delta^+(\xx,B)$ is isomorphic to $\st_{\Delta(\xx,B)}(I)$.
\end{theorem}
\begin{remark}\label{rem:equivalent-condition-Bn}
A matrix $B$ satisfies the assumption in Theorem \ref{thm:positive-simplex-B_n} if and only if a matrix $B$ has the following property: for a triangulation of a $(2n+2)$-gon $S$ obtained by a maximal compatible set $T$ which gives $B_T=B$ and a diameter $\bar{\ell}_n$ of $T$, there exists a vertex $v$ of $S_{\ell_n}$ such that all diagonals of $T$ in $S_{\ell_n}$ share $v$ (see Figure \ref{ex:assumption-of-Bn}). The shape shown at left in Figure \ref{ex:assumption-of-Bn} corresponds with a valued quiver \eqref{assumption:Bn1} and \eqref{assumption:Cn1}, and the shape shown at right corresponds to \eqref{assumption:Bn2} and \eqref{assumption:Cn2}. Because of symmetry, we may assume the left shape in Figure \ref{ex:assumption-of-Bn} without loss of generality.
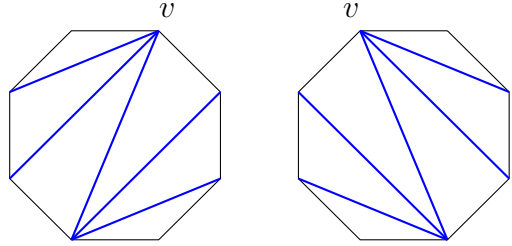
\begin{figure}[ht]
\caption{Triangulation corresponding to $B$ satisfying the assumption in Theorem \ref{thm:positive-simplex-B_n}}\label{ex:assumption-of-Bn}
\[
\begin{tikzpicture}
 \coordinate (u1) at(67.5:1.5); \coordinate (u2) at(112.5:1.5); \coordinate (lu) at(157.5:1.5);
 \coordinate (ld) at(-157.5:1.5); \coordinate (ru) at(22.5:1.5); \coordinate (rd) at(-22.5:1.5);  \coordinate (d1) at(-67.5:1.5); \coordinate (d2) at(-112.5:1.5);
 \draw (u1)--(u2)--(lu)--(ld)--(d2)--(d1)--(rd)--(ru)--(u1);  \draw[blue,thick] (u1)--(lu); \draw[blue,thick] (u1)--(ld); \draw[blue,thick] (u1)--(d2);
 \draw[blue,thick] (d2)--(rd); \draw[blue,thick] (d2)--(ru);
 \node at (67.5:1.8) {$v$};
\end{tikzpicture}
\hspace{10mm}
\begin{tikzpicture}
 \coordinate (u1) at(67.5:1.5); \coordinate (u2) at(112.5:1.5); \coordinate (lu) at(157.5:1.5);
 \coordinate (ld) at(-157.5:1.5); \coordinate (ru) at(22.5:1.5); \coordinate (rd) at(-22.5:1.5);  \coordinate (d1) at(-67.5:1.5); \coordinate (d2) at(-112.5:1.5);
 \draw (u1)--(u2)--(lu)--(ld)--(d2)--(d1)--(rd)--(ru)--(u1);  \draw[blue,thick] (u2)--(ru); \draw[blue,thick] (u2)--(rd); \draw[blue,thick] (u2)--(d1);
 \draw[blue,thick] (d1)--(ld); \draw[blue,thick] (d1)--(lu);
 \node at (112.5:1.8) {$v$};
\end{tikzpicture}
\]
\end{figure}
\end{remark}
The following lemma can be applied to an exchange matrix of tree $B_n$ or $C_n$ type, which contains the case \eqref{assumption:Bn1}--\eqref{assumption:Cn2}.
\begin{lemma}\label{lem:interior-B_n}
Let $B=(b_{ij})$ be an exchange matrix of tree $B_n$ or $C_n$ type. For a cluster complex $\Delta(\xx,B)$, there exist just $n$ vertices that are each non-linked to $I$.
\end{lemma}
\begin{proof}
By the canonical bijection between cluster variables and orbits of diagonals of a regular $(2n+2)$-gon $S$, it suffices to show that there exist $n$ orbits of diagonals such that each orbit intersects with all the orbits of diagonals corresponding to the initial cluster variables at the interior of $S$. Let $T$ be a maximal compatible orbit set corresponding to $\xx$. There are just two vertices of $S$ shared by only a single triangle in the triangulation induced by $T$ since the initial exchange matrix $B$ is of the tree $B_n$ type. We fix a vertex $w$ in these two and consider an $(n+2)$-gon $S_{\ell_n}$ cut by a diameter $\ell_n$ and not including $w$. Let $v_1,\dots,v_n$ be the vertices of an $(n+2)$-gon $S_{\ell_n}$ which are not endpoints of $\ell_n$, and $\ell_{v_1},\dots,\ell_{v_n}$ be diagonals combining $w$ with $v_1,\dots,v_n$, respectively. Then, orbits of $\ell_{v_1},\dots,\ell_{v_n}$ by $\Theta$ are the desired orbits.
\end{proof}
We denote by $I'=\{x'_1,\dots,x'_n\}$ the unique set of $n$ vertices that are non-linked to $I$. 
\begin{lemma}\label{lem:description-positiveconeBn}
For a cluster complex $\Delta(\xx,B)$ satisfying the condition in Theorem \ref{thm:positive-simplex-B_n}, $\st_{\Delta(\xx,B)}(I')$ is isomorphic to $\st_{\Delta(\xx,B)}(I)$.
\end{lemma}
\begin{proof}
We prove that $I'$ forms a cluster, and an exchange matrix $B'$ associated with this cluster equals $B$ up to the permutation of indices.
By the canonical bijection between cluster variables and orbits of diagonals of a regular $(2n+2)$-gon $S$, it suffices to show that the set of orbits corresponding to $I'$ satisfies the condition in Remark \ref{rem:equivalent-condition-Bn}. For one orbit to intersect all orbits corresponding to $I$, it must have a vertex shared by one triangle as one of its endpoints in the triangulation by $I$. Thus, if the orbit set corresponding to $I$ is that shown at left in Figure \ref{ex:assumption-of-Bn}, then that corresponding to $I'$ is that shown at right. Therefore, there exists an isomorphism and a permutation $\sigma$ such that \begin{align}
    \st_{\Delta(\xx,B)}(I')\to \st_{\Delta(\xx,B)}(I),\quad \{x'_{\sigma(i)}\} \mapsto \{x_i\},
\end{align}
where $1\leq i\leq n$.
\end{proof}
\begin{proof}[Proof of Theorem \ref{thm:positive-simplex-B_n}]
Because of Lemma \ref{lem:description-positiveconeBn}, it suffices to show $\st_{\Delta(\xx,B)}(I')=\Delta^+(\xx,B)$. 

First, we prove that $\st_{\Delta(\xx,B)}(I')\subset\Delta^+(\xx,B)$. It suffices to show that $\st_{\Delta(\xx,B)}(x'_i)\subset \Delta^+(\xx,B)$ for all $i$. This follows from the same discussion as in the proof of Theorem \ref{thm:positive-simplex-A_n}. 

Second, let us prove that $\Delta^+(\xx,B)\subset\st_{\Delta(\xx,B)}(I')$. It suffices to show that, for any simplex $F$ of $\Delta^+(\xx,B)$, there is $x'_i\in I'$ such that $F\cup\{x'_i\}$ is a simplex of $\Delta(\xx,B)$. We assume that there exists a simplex $F$ of $\Delta^+(\xx,B)$ such that $F\cup\{x'_i\}$ is not a simplex of $\Delta(\xx,B)$ for all $x'_i$. By the canonical bijection between cluster variables and orbits of diagonals of a regular $(2n+2)$-gon $S$, we can assume that the set $T_F$ of orbits corresponding to $F$ contains orbits $\bar{m}_1,\dots,\bar{m}_n$ (which may be duplicated) of diagonals such that $\bar{m}_i$ intersects with $\bar{\ell}'_i$, which is an orbit corresponding to $x'_i$, at the interior of $S$.  We note that $\{\bar{\ell}'_1,\dots, \bar{\ell}'_n\}$ satisfies the condition in Remark \ref{rem:equivalent-condition-Bn} by Lemma \ref{lem:description-positiveconeBn}. Then, we can find an orbit $\bar{m}$ of diagonals in $T_F$ intersecting all $\bar{\ell}'_i$ at the interior of $S$ by the following algorithm (see Figure \ref{fig:find-algorithm}):
\begin{itemize}
    \item [(1)] We assume that labels of orbits $\{\bar{\ell}'_1,\dots, \bar{\ell}'_n\}$ are numbered from the outside in order, and denote by $u$ a vertex sharing only one triangle in a triangulation given by $\{\bar{\ell}'_1,\dots, \bar{\ell}'_n\}$.
    \item[(2)] We choose an orbit $\bar{m}_1$ in $T_F$ intersecting with $\bar{\ell}'_1$. Then, one of the diagonals in $\bar{m}_1$ must have an endpoint at $u$. If $\bar{m}_1$ intersects with $\bar{\ell}'_n$, then we have $\bar{m}=\bar{m_1}$. If not, we proceed with the next step.
    \item[(3)] We assume that $\bar{m}_1$ intersects with $\bar{\ell}_i'$ and does not intersect with $\bar{\ell}_{i+1}'$. We choose an orbit $\bar{m}_{i+1}$ in $T_F$ intersecting with $\bar{\ell}'_{i+1}$. Then, one of the diagonals in $\bar{m}_{i+1}$ must have an endpoint at $u$ since if it did not, $\bar{m}_{i+1}$ would intersect with $\bar{m}_1$ at the interior of $S$. If $\bar{m}_{i+1}$ intersects with $\bar{\ell}'_n$ at the interior of $S$, then we have $\bar{m}=\bar{m}_{i+1}$. If not, we repeat to choose $\bar{m}_{j}$ in the same manner as choosing $\bar{m}_{i+1}$ until $\bar{m}_{j}$ intersects with $\bar{\ell}'_n$. Since one of the diagonals in $\bar{m}_{j}$ must have an endpoint at $u$, $\bar{m}_{j}$ intersects with all $\bar{\ell}_i$ at the interior of $S$. Therefore, we have $\bar{m}=\bar{m}_{j}$.
\end{itemize}
\begin{figure}[ht]
\caption{Algorithm for finding $\bar{m}$ ($n=4$).}\label{fig:find-algorithm}
\[
\begin{tikzpicture}
 \coordinate (1) at(18:1.5); \coordinate (2) at(54:1.5);\coordinate (3) at(90:1.5); \coordinate (4) at(126:1.5);
 \coordinate (5) at(162:1.5); \coordinate (10) at(-18:1.5); \coordinate (9) at(-54:1.5);  \coordinate (8) at(-90:1.5); \coordinate (7) at(-126:1.5); \coordinate (6) at(-162:1.5);
 \draw (1)--(2)--(3)--(4)--(5)--(6)--(7)--(8)--(9)--(10)--(1);  \draw[blue,thick] (1)--(3); \draw[blue,thick] (1)--(4); \draw[blue,thick] (1)--(5); \draw[blue,thick] (1)--(6);
 \draw[blue,dashed] (10)--(6);  \draw[blue,dashed] (9)--(6);
 \draw[blue,dashed] (8)--(6);
 \node at (54:1.8) {$u$};
 \fill [white](0.5,1) circle (2mm);
\fill [white](-0.3,1) circle (3mm);
\fill [white](-0.8,0.5) circle (3mm);
\fill [white](0,0) circle (3mm);
\node at (0.5,1) {$\bar{\ell}'_1$};
\node at (-0.3,1) {$\bar{\ell}'_2$};
\node at (-0.8,0.5) {$\bar{\ell}'_3$};
\node at (0,0) {$\bar{\ell}'_4$};
\node at (270:2) {(1)};
\end{tikzpicture}
\hspace{20pt}
\begin{tikzpicture}
 \coordinate (1) at(18:1.5); \coordinate (2) at(54:1.5);\coordinate (3) at(90:1.5); \coordinate (4) at(126:1.5);
 \coordinate (5) at(162:1.5); \coordinate (10) at(-18:1.5); \coordinate (9) at(-54:1.5);  \coordinate (8) at(-90:1.5); \coordinate (7) at(-126:1.5); \coordinate (6) at(-162:1.5);
 \draw (1)--(2)--(3)--(4)--(5)--(6)--(7)--(8)--(9)--(10)--(1);  \draw[blue,thick] (1)--(3); \draw[blue,dashed] (1)--(4); \draw[blue,dashed] (1)--(5); \draw[blue,dashed] (1)--(6);
 \draw[blue,dashed] (10)--(6);  \draw[blue,dashed] (9)--(6);
 \draw[blue,dashed] (8)--(6);
 \draw[red,thick] (2)--(5);
 \draw[red,dashed] (7)--(10);
 \node at (54:1.8) {$u$};
 \fill [white](0,0.8) circle (3mm);
 \node at (0,0.8) {$\bar{m}_1$};
 \node at (270:2) {(2)};
\end{tikzpicture}
\hspace{20pt}
\begin{tikzpicture}
 \coordinate (1) at(18:1.5); \coordinate (2) at(54:1.5);\coordinate (3) at(90:1.5); \coordinate (4) at(126:1.5);
 \coordinate (5) at(162:1.5); \coordinate (10) at(-18:1.5); \coordinate (9) at(-54:1.5);  \coordinate (8) at(-90:1.5); \coordinate (7) at(-126:1.5); \coordinate (6) at(-162:1.5);
 \draw (1)--(2)--(3)--(4)--(5)--(6)--(7)--(8)--(9)--(10)--(1);  \draw[blue,dashed] (1)--(3); \draw[blue,dashed] (1)--(4); \draw[blue,thick] (1)--(5); \draw[blue,dashed] (1)--(6);
 \draw[blue,dashed] (10)--(6);  \draw[blue,dashed] (9)--(6);
 \draw[blue,dashed] (8)--(6);
 \draw[red,dashed] (2)--(5);
 \draw[red,dashed] (7)--(10);
 \draw[red,thick] (2)--(7);
 \node at (54:1.8) {$u$};
 \fill [white] (-0.8,0.2) rectangle (0.8,-0.2);
 \node at (0,0) {$\bar{m}=\bar{m}_3$};
 \node at (270:2) {(3)};
\end{tikzpicture}
\]
\end{figure}
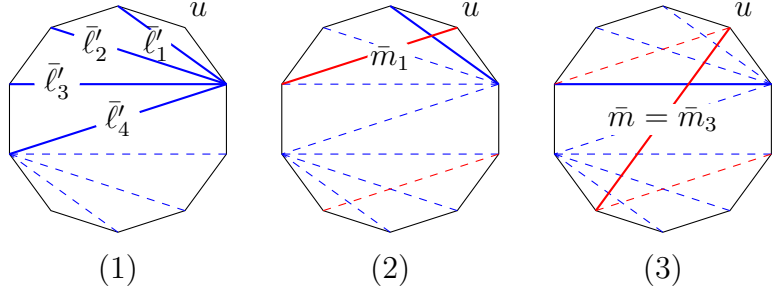
Since an orbit of diagonals is in the initial maximal compatible set if and only if an orbit of diagonals crossing all $\bar{\ell}_i'$, this contradicts the fact that $F$ is a simplex of $\Delta^+(\xx,B)$.
\end{proof}
 Let $K$ be a simplicial complex and $V$ a vertex subset of $K$. The \emph{link of $V$} is a (not necessarily full) subcomplex of $K$ defined by $\lk(V) = \{\sigma\in\st_K(V) \mid \sigma \cap V=\emptyset \}$; that is, simplices of $\lk_K(V)$ are all the simplices $\sigma$ of $K$ with $v \not \in \sigma$ for all $v\in V$ such that there exists $v\in V$ such that $\sigma \cup \{v\}$ is a simplex of $K$.
 \begin{corollary}\label{cor:boundary-Bn}
For a cluster complex $\Delta(\xx,B)$ satisfying the assumption in Theorem \ref{thm:positive-simplex-B_n}, we have $\lk_{\Delta(\xx,B)}(I)=\lk_{\Delta(\xx,B)}(I')$. Furthermore, it is isomorphic to $\Delta(B_{n-1})$.
\end{corollary}
\begin{proof}
Here, we prove $\lk_{\Delta(\xx,B)}(I)\subset\lk_{\Delta(\xx,B)}(I')$. Let $F$ be a simplex of $\lk_{\Delta(\xx,B)}(I)$. For all $x_i\in I$, we have $x_i\notin F$. Therefore, $F$ is a simplex of $\Delta^+(\xx,B)$. By Theorem \ref{thm:positive-simplex-B_n}, $F$ is also a simplex of $\st_{\Delta(\xx,B)}(I')$. Furthermore, for all $x'_i\in I'$, we have $x'_i\notin F$ since each $x'_i$ is non-linked to $I$. Therefore, $F$ is a simplex of $\lk_{\Delta(\xx,B)}(I')$. The converse follows from the symmetry of $\lk_{\Delta(\xx,B)}(I)$ and $\lk_{\Delta(\xx,B)}(I')$. Indeed, $I$ is a unique set of $n$ non-linked vertex to $I'$.

Next, we prove the second statement. By using geometric realization, we identify $\lk_{\Delta(\xx,B)}(I)$ with $\Delta(B_{n-1})$ as follows.
Let $S$ be a $(2n+2)$-gon and let $T$ and $T'$ maximal compatible sets corresponding to $\xx$ and $\xx'$ (see Figure \ref{fig:T,T',D-in-B_n}).
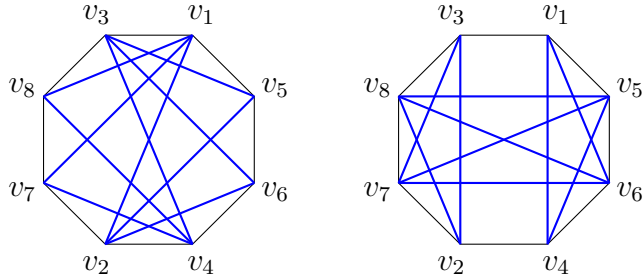
\begin{figure}[ht]
\caption{$T\cup T'$ (left), $D$ (right) in the case that $n=3$ }\label{fig:T,T',D-in-B_n}
\[
\begin{tikzpicture}[baseline=0mm]
 \coordinate (u1) at(67.5:1.5); \coordinate (u2) at(112.5:1.5); \coordinate (lu) at(157.5:1.5);
 \coordinate (ld) at(-157.5:1.5); \coordinate (ru) at(22.5:1.5); \coordinate (rd) at(-22.5:1.5);  \coordinate (d1) at(-67.5:1.5); \coordinate (d2) at(-112.5:1.5);
 \draw (u1)--(u2)--(lu)--(ld)--(d2)--(d1)--(rd)--(ru)--(u1);
 \draw [blue,thick](u1)--(lu);
 \draw [blue,thick](u1)--(ld);
 \draw [blue,thick](u1)--(d2);
 \draw [blue,thick](ru)--(d2);
 \draw [blue,thick](rd)--(d2);
 \draw [blue,thick](u2)--(ru);
 \draw [blue,thick](u2)--(rd);
 \draw [blue,thick](u2)--(d1);
 \draw [blue,thick](lu)--(d1);
 \draw [blue,thick](ld)--(d1);
 \node at (67.5:1.8) {$v_1$};
 \node at (-112.5:1.8) {$v_2$};
 \node at (112.5:1.8) {$v_3$};
 \node at (-67.5:1.8) {$v_4$};
 \node at (22.5:1.8) {$v_5$};
 \node at (157.5:1.8) {$v_8$};
 \node at (-157.5:1.8) {$v_7$};
 \node at (-22.5:1.8) {$v_6$};
\end{tikzpicture}
\hspace{20pt}
\begin{tikzpicture}[baseline=0mm]
 \coordinate (u1) at(67.5:1.5); \coordinate (u2) at(112.5:1.5); \coordinate (lu) at(157.5:1.5);
 \coordinate (ld) at(-157.5:1.5); \coordinate (ru) at(22.5:1.5); \coordinate (rd) at(-22.5:1.5);  \coordinate (d1) at(-67.5:1.5); \coordinate (d2) at(-112.5:1.5);
 \draw (u1)--(u2)--(lu)--(ld)--(d2)--(d1)--(rd)--(ru)--(u1);
 \draw [blue,thick](u1)--(d1);
 \draw [blue,thick](u1)--(rd);
 \draw [blue,thick](u2)--(d2);
 \draw [blue,thick](u2)--(ld);
 \draw [blue,thick](d1)--(ru);
 \draw [blue,thick](d2)--(lu);
 \draw [blue,thick](lu)--(ru);
 \draw [blue,thick](ld)--(ru);
 \draw [blue,thick](lu)--(rd);
 \draw [blue,thick](ld)--(rd);
 \node at (67.5:1.8) {$v_1$};
 \node at (-112.5:1.8) {$v_2$};
 \node at (112.5:1.8) {$v_3$};
 \node at (-67.5:1.8) {$v_4$};
 \node at (22.5:1.8) {$v_5$};
 \node at (157.5:1.8) {$v_8$};
 \node at (-157.5:1.8) {$v_7$};
 \node at (-22.5:1.8) {$v_6$};
\end{tikzpicture}
\]
\end{figure}
Let $\{v_1,v_2\}$ and $\{v_3,v_4\}$ be pairs of vertices of $S$ sharing only a single triangle in a triangulation given by $T$ and $T'$, respectively. We assume the edges $(v_1,v_3)$ and $(v_2,v_4)$ are diagonals of $S$, and $v_1$ is a neighbor of $v_3$. We name the remainder of the vertices $v_5,\dots,v_{2n+2}$ in the clockwise order from the right neighbor of $v_1$. Let $S'$ be a $2n$-gon obtained by reducing $(v_1,v_3)$ and $(v_2,v_4)$ to points respectively. Let $D$ be the set of all diagonals of $S$ except for those contained in $T$ or $T'$ (see Figure \ref{fig:T,T',D-in-B_n}), and $D'$ the set of all diagonals of $S'$. We consider the surjection $\psi\colon D \to D',
(v_i,v_j)\mapsto (v_i,v_j)$ (see Figure \ref{fig:sujection-D-to-D'}).
\begin{figure}[ht]
\caption{Surjection $D$ to $D'$ ($n=3$) }\label{fig:sujection-D-to-D'}
\[
\begin{tikzpicture}[baseline=0mm]
 \coordinate (u1) at(67.5:1.5); \coordinate (u2) at(112.5:1.5); \coordinate (lu) at(157.5:1.5);
 \coordinate (ld) at(-157.5:1.5); \coordinate (ru) at(22.5:1.5); \coordinate (rd) at(-22.5:1.5);  \coordinate (d1) at(-67.5:1.5); \coordinate (d2) at(-112.5:1.5);
 \draw (u1)--(u2)--(lu)--(ld)--(d2)--(d1)--(rd)--(ru)--(u1);
 \draw [blue,thick](u1)--(d1);
 \draw [blue,thick](u1)--(rd);
 \draw [blue,thick](u2)--(d2);
 \draw [blue,thick](u2)--(ld);
 \draw [blue,thick](u2)--(ld);
 \draw [blue,thick](d1)--(ru);
 \draw [blue,thick](d2)--(lu);
 \draw [blue,thick](d1)--(ru);
 \draw [blue,thick](d2)--(lu);
 \draw [blue,thick](lu)--(ru);
 \draw [blue,thick](ld)--(ru);
 \draw [blue,thick](lu)--(rd);
 \draw [blue,thick](ld)--(rd);
 \node at (67.5:1.8) {$v_1$};
 \node at (-112.5:1.8) {$v_2$};
 \node at (112.5:1.8) {$v_3$};
 \node at (-67.5:1.8) {$v_4$};
 \node at (22.5:1.8) {$v_5$};
 \node at (157.5:1.8) {$v_8$};
 \node at (-157.5:1.8) {$v_7$};
 \node at (-22.5:1.8) {$v_6$};
\end{tikzpicture}
\hspace{20pt}\to \hspace{20pt}
\begin{tikzpicture}[baseline=0mm]
 \coordinate (u) at(90:1.5);  \coordinate (lu) at(145:1);
 \coordinate (ld) at(-145:1); \coordinate (ru) at(35:1); \coordinate (rd) at(-35:1);  \coordinate (d) at(-90:1.5);
 \draw (u)--(lu)--(ld)--(d)--(rd)--(ru)--(u);
 \draw [blue,thick](u)--(d);
 \draw [blue,thick](u)--(rd);
 \draw [blue,thick](u)--(ld);
 \draw [blue,thick](u)--(ld);
 \draw [blue,thick](d)--(ru);
 \draw [blue,thick](d)--(lu);
 \draw [blue,thick](d)--(ru);
 \draw [blue,thick](d)--(lu);
 \draw [blue,thick](lu)--(ru);
 \draw [blue,thick](ld)--(ru);
 \draw [blue,thick](lu)--(rd);
 \draw [blue,thick](ld)--(rd);
 \node at (90:1.8) {$v_1=v_3$};
 \node at (-90:1.8) {$v_2=v_4$};
 \node at (35:1.3) {$v_5$};
 \node at (145:1.3) {$v_8$};
 \node at (-35:1.3) {$v_6$};
 \node at (-145.5:1.3) {$v_7$};
\end{tikzpicture}
\]
\end{figure}
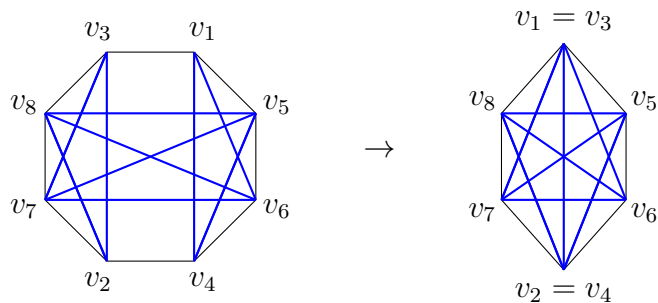
We show that $\psi$ provides a bijection $\bar{\psi}\colon D/\Theta \to D'/\Theta,
\overline{(v_i,v_j)}\mapsto \overline{(v_i,v_j)}$ of orbits by $\Theta$ by induction on $n$. When $n=2$, we can obtain this fact by direct calculation. We assume that $\bar{\psi}$ is bijective when the rank is $n-1$. We prove this for rank $n$. Let $\tilde{S}$ be a $2n$-gon obtained by the convex hull of vertices of $S$ except for $v_5$ and its antipodal vertex $v_{n+4}$. We consider subsets $\tilde{T}$ and $\tilde{T}'$, which are obtained by removing diagonals with $v_5$ or $v_{n+4}$ as endpoints from $T$ and $T'$, respectively. Then, we can consider $\tilde{T}$ and $\tilde{T'}$ as a triangulation of $\tilde{S}$ satisfying the condition in Remark \ref{rem:equivalent-condition-Bn}, and $\tilde{T'}/\Theta$ is an orbit set such that each orbit crosses all orbits of $\tilde{T}/\Theta$ at the interior of $\tilde{S}$. Therefore, by the assumption of induction and a correspondence $\overline{(v_1,v_6)}\mapsto\overline{(v_1,v_6)}$, we have a bijection $\bar{\psi}|_{\tilde{D}}\colon \tilde{D}/\Theta\to\tilde{D'}/\Theta$, where $\tilde{D}$ is the set of all diagonals of $\tilde{S}$, except for those contained in $\tilde{T}$ or $\tilde{T'}$, and $\tilde{D'}$ is the set of all diagonals of $S'$ except for those with $v_5$ or $v_{n+4}$ as endpoints (see Figure \ref{fig:induction-B_n}).
\begin{figure}[ht]
\caption{$\tilde{T}\cup \tilde{T'}$ (left), $\tilde{D}$ (center), $\tilde{D'}$ (right) in the case that $n=3$ }\label{fig:induction-B_n}
\[\begin{tikzpicture}[baseline=0mm]
 \coordinate (u1) at(67.5:1.5); \coordinate (u2) at(112.5:1.5); \coordinate (lu) at(157.5:1.5);
 \coordinate (ld) at(-157.5:1.5); \coordinate (ru) at(22.5:1.5); \coordinate (rd) at(-22.5:1.5);  \coordinate (d1) at(-67.5:1.5); \coordinate (d2) at(-112.5:1.5);
 \draw (u1)--(u2)--(lu);
 \draw [dashed](lu)--(ld)--(d2);
 \draw (d2)--(d1)--(rd);
 \draw [dashed](rd)--(ru)--(u1);
 \draw [blue,thick](u1)--(lu);
 \draw (u1)--(rd);
 \draw [blue,thick](u1)--(d2);
 \draw [blue,thick](u2)--(rd);
 \draw [blue,thick](d2)--(rd);
 \draw [blue,thick](u2)--(d1);
 \draw [blue,thick](d1)--(lu);
 \draw (d2)--(lu);
 \node at (67.5:1.8) {$v_1$};
 \node at (-112.5:1.8) {$v_2$};
 \node at (112.5:1.8) {$v_3$};
 \node at (-67.5:1.8) {$v_4$};
 \node at (22.5:1.8) {$v_5$};
 \node at (157.5:1.8) {$v_8$};
 \node at (-157.5:1.8) {$v_7$};
 \node at (-22.5:1.8) {$v_6$};
\end{tikzpicture}
\hspace{20pt}
\begin{tikzpicture}[baseline=0mm]
 \coordinate (u1) at(67.5:1.5); \coordinate (u2) at(112.5:1.5); \coordinate (lu) at(157.5:1.5);
 \coordinate (ld) at(-157.5:1.5); \coordinate (ru) at(22.5:1.5); \coordinate (rd) at(-22.5:1.5);  \coordinate (d1) at(-67.5:1.5); \coordinate (d2) at(-112.5:1.5);
 \draw (u1)--(u2)--(lu);
 \draw [dashed](lu)--(ld)--(d2);
 \draw (d2)--(d1)--(rd);
 \draw [dashed](rd)--(ru)--(u1);
 \draw [blue,thick](u1)--(d1);
 \draw [blue,thick](u1)--(rd);
 \draw [blue,thick](lu)--(rd);
 \draw [blue,thick](u2)--(d2);
 \draw [blue,thick](d2)--(lu);
 \node at (67.5:1.8) {$v_1$};
 \node at (-112.5:1.8) {$v_2$};
 \node at (112.5:1.8) {$v_3$};
 \node at (-67.5:1.8) {$v_4$};
 \node at (22.5:1.8) {$v_5$};
 \node at (157.5:1.8) {$v_8$};
 \node at (-157.5:1.8) {$v_7$};
 \node at (-22.5:1.8) {$v_6$};
\end{tikzpicture}
\hspace{20pt}
\begin{tikzpicture}[baseline=0mm]
 \coordinate (u) at(90:1.5);  \coordinate (lu) at(145:1);
 \coordinate (ld) at(-145:1); \coordinate (ru) at(35:1); \coordinate (rd) at(-35:1);  \coordinate (d) at(-90:1.5);
 \draw (u)--(lu);
 \draw [dashed](lu)--(ld)--(d);
 \draw (d)--(rd);
 \draw [dashed](rd)--(ru)--(u);
 \draw [blue,thick](u)--(d);
 \draw [blue,thick](u)--(rd);
 \draw [blue,thick](d)--(lu);
 \draw [blue,thick](d)--(lu);
 \draw [blue,thick](lu)--(rd);
 \node at (90:1.8) {$v_1=v_3$};
 \node at (-90:1.8) {$v_2=v_4$};
 \node at (35:1.3) {$v_5$};
 \node at (145:1.3) {$v_8$};
 \node at (-35:1.3) {$v_6$};
 \node at (-145.5:1.3) {$v_7$};
\end{tikzpicture}
\]
\end{figure}
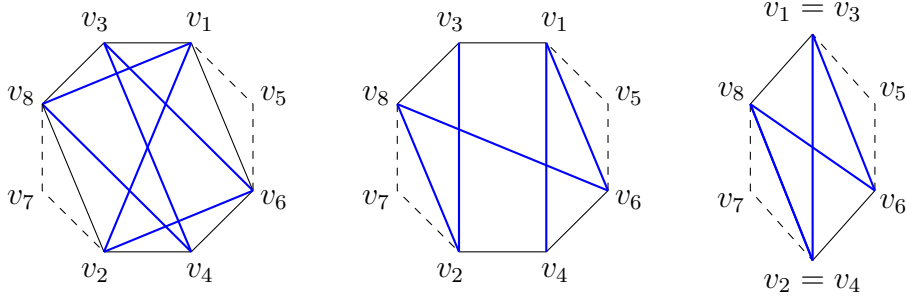
Moreover, we have
\begin{align*}
    D\setminus\tilde{D}=\{(v_5,v_i)\mid i \neq 1,2,3,6\}\cup \{(v_{n+4},v_j)\mid j\neq 1,2,4,n+5\}.
\end{align*}
Therefore, $\psi|_{D\setminus\tilde{D}}\colon D\setminus\tilde{D} \to D'\setminus\tilde{D'}$ is bijective. Thus, $\bar{\psi}|_{D\setminus\tilde{D}}$ is bijective as well. This implies that $\bar{\psi}$ is a bijection.
Furthermore, $\bar{\psi}$ preserves the compatibility of orbits; that is, if $\overline{(v_i,v_j)}$ intersects with $\overline{(v_k,v_\ell)}$ at the interior of $S$, then they also do at the interior $S'$, and vice versa.
We note that $\lk_{\Delta(\xx,B)}(I)$ is a full subcomplex of $\Delta(\xx,B)$ since $\lk_{\Delta(\xx,B)}(I)$ is the intersection of full subcomplexes $\Delta^+(\xx,B)$ and $\Delta^+(\xx',B')$ Therefore, $\bar{\psi}$ can be lifted to an isomorphism between $\lk_{\Delta(\xx,B)}(I)$ and $\Delta(B_{n-1})$.
\end{proof}

\begin{corollary}\label{cor:intersetion-B}
 For a cluster complex $\Delta(\xx,B)$ satisfying the assumption in Theorem \ref{thm:positive-simplex-B_n}, we have
 \[\Delta(\xx,B)=\st_{\Delta(\xx,B)}(I)\cup\st_{\Delta(\xx,B)}(I')\] and \[\lk_{\Delta(\xx,B)}(I)=\st_{\Delta(\xx,B)}(I)\cap\st_{\Delta(\xx,B)}(I').\]
\end{corollary}

\begin{proof}
The first equation follows from Lemma \ref{lem:interior-B_n}, and we provide a proof for the second here. For any $F \in \lk_{\Delta(\xx,B)}(I)$, $F$ is in $\st_{\Delta(\xx,B)}(I)$, and by Corollary \ref{cor:boundary-Bn}, $F$ is also in $\st_{\Delta(\xx,B)}(I')$. Therefore, we have $\lk_{\Delta(\xx,B)}(I)\subset\st_{\Delta(\xx,B)}(I)\cap\st_{\Delta(\xx,B)}(I')$. Let $F\in \st_{\Delta(\xx,B)}(I)\cap\st_{\Delta(\xx,B)}(I')$. Since $F\in \st_{\Delta(\xx,B)}(I)$, we have $F\cap I'=\emptyset$. By this fact and the assumption $F\in \st_{\Delta(\xx,B)}(I')$, we have $F\in \lk_{\Delta(\xx,B)}(I')$. Therefore, by Corollary \ref{cor:boundary-Bn}, we have  $F\in \lk_{\Delta(\xx,B)}(I)$.
\end{proof}

By using Corollary \ref{cor:intersetion-B}, we can calculate the number of simplices in $\Delta^+(\xx,B)$ where $B$ is an exchange matrix of tree $B_n$ or $C_n$ type from cluster complexes of $B_{n}$ type and $B_{n-1}$ type as in Corollary \ref{cor:positive-face-from-all-face-n-1}.
\begin{corollary}[\cite{cha04}*{(46)}]\label{cor:positive-face-from-all-face-B_nB_n-1}
Let $B$ be an exchange matrix the quiver $Q_B$ of which has a diagram of $B_n$ or $C_n$ type in Figure \ref{fig:valued-quiver} as its underlying graph. Let $f_{n,k}$ be the number of $k$-dimensional simplices of $\Delta(B_{n})$, and $f^+_{n,k}$ the number of $k$-dimensional simplices of $\Delta^+(\xx,B)$. Then, we have
\begin{align}\label{eq:B_n-positice-face-vector}
    f^+_{n,k}=
    \dfrac{1}{2}(f_{n,k}+f_{n-1,k})=\begin{pmatrix}n\\ k+1\end{pmatrix}\begin{pmatrix}n+k\\ k+1\end{pmatrix}.
\end{align}
\end{corollary}
\begin{proof}
By Corollary \ref{cor:intersetion-B} and Corollary \ref{cor:independent-orientation}, we have the difference of face vectors of $\Delta^+(\xx,B)$ and $\Delta(\xx,B)\setminus\Delta^+(\xx,B)$ equals the face vector of $\Delta(B_{n-1})$. Therefore, we have
\[f(\Delta(B_n))=f(\Delta^+(\xx,B))+(f(\Delta^+(\xx,B)-f(\Delta(B_{n-1}))),\]
and by simplifying this equation, we obtain 
\begin{align}\label{facevectorBn}
f(\Delta^+(\xx,B))=\dfrac{1}{2}(f(\Delta(B_n))+f(\Delta(B_{n-1}))).
\end{align}
By \cite{cha04}*{(44)}, we have
\begin{align}\label{eq:formula-faces-B_n}
    f_{n,k}=\begin{pmatrix}n\\ k+1\end{pmatrix}\begin{pmatrix}n+k+1\\ k+1\end{pmatrix}
\end{align}
By \eqref{facevectorBn} and \eqref{eq:formula-faces-B_n}, we obtain the desired equation \eqref{eq:B_n-positice-face-vector}.
\end{proof}
\begin{example}
Let $(\xx,B)$ be a seed which is the same as in Example \ref{ex:correspondence-arc-variable-B3}. Then, $\Delta^+(\xx,B)$ is a domain shown filled in with gray in Figure \ref{positivecomplexB_n}.
 \begin{figure}[ht]
\caption{positive cluster complex of $B_3$ type\label{positivecomplexB_n}}
\hspace{5pt}
\begin{center}
\scalebox{0.8}{
\begin{tikzpicture}
\coordinate (0) at (0,0);
\coordinate (u*) at (90:8);
\coordinate (u) at (90:6);
\coordinate (ul) at (150:4);
\coordinate (ur) at (30:4);
\coordinate (uml) at (150:2);
\coordinate (umr) at (30:2);
\coordinate (dmc) at (-90:2);
\coordinate (dl) at (-150:6);
 \coordinate (dl*) at (-150:8);
\coordinate (dr) at (-30:6);
\coordinate (dr*) at (-30:8);
\coordinate (d) at (-90:4);
\coordinate (ll) at (150:6);
\coordinate (rr) at (30:6);
\coordinate (dd) at (-90:6);
\draw (u) to (ul);
\draw (u) to (ur);
\draw (ul) to (uml);
\draw (ur) to (umr);
\draw (umr) to (uml);
\draw (u) to (uml);
\draw (u) to (umr);
\draw (uml) to (dl);
\draw (umr) to (dr);
\draw (ul) to (dl);
\draw (ur) to (dr);
\draw (uml) to (dmc);
\draw (umr) to (dmc);
\draw (dl) to (dmc);
\draw (umr) [out=290,in=45] .. controls (-35:4) and (-45:4) ..(d);
\draw (ur) [out=290,in=45] .. controls (-20:9) and (-30:9) .. (dd);
\draw (dl) to (d);
\draw (dmc) to (d);
\draw (dr) to (d);
\draw (dd) to (d);
\draw (dd) to (dl);
\draw (dd) to (dr);
\draw (ll) to (ul);
\draw (ll) to (dl);
\draw (ll) to (u);
\draw (rr) to (ur);
\draw (rr) to (u);
\draw(ll) [out=90,in=180]to (u*);
\draw(rr) [out=90,in=0]to (u*);
\draw(ll) [out=210,in=120]to (dl*);
\draw(dd) [out=210,in=300]to (dl*);
\draw(dd) [out=330,in=240]to (dr*);
\draw(rr) [out=330,in=60]to (dr*);
\fill [white](u) circle (0.9cm);
\fill[white] (ul) circle (0.9cm);
\fill[white] (ur) circle (0.9cm);
\fill[white] (uml) circle (0.9cm);
\fill[white] (umr) circle (0.9cm);
\fill[white] (dmc) circle (0.9cm);
\fill[white] (dl) circle (0.9cm);
\fill[white] (dr) circle (0.9cm);
\fill[white] (d) circle (0.9cm);
\fill[white] (ll) circle (0.9cm);
\fill[white] (rr) circle (0.9cm);
\fill[white] (dd) circle (0.9cm);
\draw (u)++(22.5:0.7cm) to ++(135:0.54cm);
\draw (u)++(67.5:0.7cm) to ++(180:0.54cm);
\draw (u)++(112.5:0.7cm) to ++(225:0.54cm);
\draw (u)++(157.5:0.7cm) to ++(270:0.54cm);
\draw (u)++(202.5:0.7cm) to ++(315:0.54cm);
\draw (u)++(247.5:0.7cm) to ++(0:0.54cm);
\draw (u)++(292.5:0.7cm) to ++(45:0.54cm);
\draw (u)++(337.5:0.7cm) to ++(90:0.54cm);
\draw (ur)++(22.5:0.7cm) to ++(135:0.54cm);
\draw (ur)++(67.5:0.7cm) to ++(180:0.54cm);
\draw (ur)++(112.5:0.7cm) to ++(225:0.54cm);
\draw (ur)++(157.5:0.7cm) to ++(270:0.54cm);
\draw (ur)++(202.5:0.7cm) to ++(315:0.54cm);
\draw (ur)++(247.5:0.7cm) to ++(0:0.54cm);
\draw (ur)++(292.5:0.7cm) to ++(45:0.54cm);
\draw (ur)++(337.5:0.7cm) to ++(90:0.54cm);
\draw (ul)++(22.5:0.7cm) to ++(135:0.54cm);
\draw (ul)++(67.5:0.7cm) to ++(180:0.54cm);
\draw (ul)++(112.5:0.7cm) to ++(225:0.54cm);
\draw (ul)++(157.5:0.7cm) to ++(270:0.54cm);
\draw (ul)++(202.5:0.7cm) to ++(315:0.54cm);
\draw (ul)++(247.5:0.7cm) to ++(0:0.54cm);
\draw (ul)++(292.5:0.7cm) to ++(45:0.54cm);
\draw (ul)++(337.5:0.7cm) to ++(90:0.54cm);
\draw (umr)++(22.5:0.7cm) to ++(135:0.54cm);
\draw (umr)++(67.5:0.7cm) to ++(180:0.54cm);
\draw (umr)++(112.5:0.7cm) to ++(225:0.54cm);
\draw (umr)++(157.5:0.7cm) to ++(270:0.54cm);
\draw (umr)++(202.5:0.7cm) to ++(315:0.54cm);
\draw (umr)++(247.5:0.7cm) to ++(0:0.54cm);
\draw (umr)++(292.5:0.7cm) to ++(45:0.54cm);
\draw (umr)++(337.5:0.7cm) to ++(90:0.54cm);
\draw (uml)++(22.5:0.7cm) to ++(135:0.54cm);
\draw (uml)++(67.5:0.7cm) to ++(180:0.54cm);
\draw (uml)++(112.5:0.7cm) to ++(225:0.54cm);
\draw (uml)++(157.5:0.7cm) to ++(270:0.54cm);
\draw (uml)++(202.5:0.7cm) to ++(315:0.54cm);
\draw (uml)++(247.5:0.7cm) to ++(0:0.54cm);
\draw (uml)++(292.5:0.7cm) to ++(45:0.54cm);
\draw (uml)++(337.5:0.7cm) to ++(90:0.54cm);
\draw (dl)++(22.5:0.7cm) to ++(135:0.54cm);
\draw (dl)++(67.5:0.7cm) to ++(180:0.54cm);
\draw (dl)++(112.5:0.7cm) to ++(225:0.54cm);
\draw (dl)++(157.5:0.7cm) to ++(270:0.54cm);
\draw (dl)++(202.5:0.7cm) to ++(315:0.54cm);
\draw (dl)++(247.5:0.7cm) to ++(0:0.54cm);
\draw (dl)++(292.5:0.7cm) to ++(45:0.54cm);
\draw (dl)++(337.5:0.7cm) to ++(90:0.54cm);
\draw (dr)++(22.5:0.7cm) to ++(135:0.54cm);
\draw (dr)++(67.5:0.7cm) to ++(180:0.54cm);
\draw (dr)++(112.5:0.7cm) to ++(225:0.54cm);
\draw (dr)++(157.5:0.7cm) to ++(270:0.54cm);
\draw (dr)++(202.5:0.7cm) to ++(315:0.54cm);
\draw (dr)++(247.5:0.7cm) to ++(0:0.54cm);
\draw (dr)++(292.5:0.7cm) to ++(45:0.54cm);
\draw (dr)++(337.5:0.7cm) to ++(90:0.54cm);
\draw (dmc)++(22.5:0.7cm) to ++(135:0.54cm);
\draw (dmc)++(67.5:0.7cm) to ++(180:0.54cm);
\draw (dmc)++(112.5:0.7cm) to ++(225:0.54cm);
\draw (dmc)++(157.5:0.7cm) to ++(270:0.54cm);
\draw (dmc)++(202.5:0.7cm) to ++(315:0.54cm);
\draw (dmc)++(247.5:0.7cm) to ++(0:0.54cm);
\draw (dmc)++(292.5:0.7cm) to ++(45:0.54cm);
\draw (dmc)++(337.5:0.7cm) to ++(90:0.54cm);
\draw (d)++(22.5:0.7cm) to ++(135:0.54cm);
\draw (d)++(67.5:0.7cm) to ++(180:0.54cm);
\draw (d)++(112.5:0.7cm) to ++(225:0.54cm);
\draw (d)++(157.5:0.7cm) to ++(270:0.54cm);
\draw (d)++(202.5:0.7cm) to ++(315:0.54cm);
\draw (d)++(247.5:0.7cm) to ++(0:0.54cm);
\draw (d)++(292.5:0.7cm) to ++(45:0.54cm);
\draw (d)++(337.5:0.7cm) to ++(90:0.54cm);
\draw (ll)++(22.5:0.7cm) to ++(135:0.54cm);
\draw (ll)++(67.5:0.7cm) to ++(180:0.54cm);
\draw (ll)++(112.5:0.7cm) to ++(225:0.54cm);
\draw (ll)++(157.5:0.7cm) to ++(270:0.54cm);
\draw (ll)++(202.5:0.7cm) to ++(315:0.54cm);
\draw (ll)++(247.5:0.7cm) to ++(0:0.54cm);
\draw (ll)++(292.5:0.7cm) to ++(45:0.54cm);
\draw (ll)++(337.5:0.7cm) to ++(90:0.54cm);
\draw (rr)++(22.5:0.7cm) to ++(135:0.54cm);
\draw (rr)++(67.5:0.7cm) to ++(180:0.54cm);
\draw (rr)++(112.5:0.7cm) to ++(225:0.54cm);
\draw (rr)++(157.5:0.7cm) to ++(270:0.54cm);
\draw (rr)++(202.5:0.7cm) to ++(315:0.54cm);
\draw (rr)++(247.5:0.7cm) to ++(0:0.54cm);
\draw (rr)++(292.5:0.7cm) to ++(45:0.54cm);
\draw (rr)++(337.5:0.7cm) to ++(90:0.54cm);
\draw (dd)++(22.5:0.7cm) to ++(135:0.54cm);
\draw (dd)++(67.5:0.7cm) to ++(180:0.54cm);
\draw (dd)++(112.5:0.7cm) to ++(225:0.54cm);
\draw (dd)++(157.5:0.7cm) to ++(270:0.54cm);
\draw (dd)++(202.5:0.7cm) to ++(315:0.54cm);
\draw (dd)++(247.5:0.7cm) to ++(0:0.54cm);
\draw (dd)++(292.5:0.7cm) to ++(45:0.54cm);
\draw (dd)++(337.5:0.7cm) to ++(90:0.54cm);
\draw [thick, blue](ll)++(67.5:0.7cm) to ++(247.5:1.4cm);
\draw [thick, blue](dd)++(67.5:0.7cm) to ++(202.5:0.99cm);
\draw [thick, blue](dd)++(-22.5:0.7cm) to ++(202.5:0.99cm);
\draw [thick, blue](rr)++(67.5:0.7cm) to ++(225:1.29cm);
\draw [thick, blue](rr)++(22.5:0.7cm) to ++(225:1.29cm);
\draw [thick, blue](u)++(112.5:0.7cm) to ++(-112.5:0.99cm);
\draw [thick, blue](u)++(22.5:0.7cm) to ++(-112.5:0.99cm);
\draw [thick, blue](dr)++(22.5:0.7cm) to ++(180:1.29cm);
\draw [thick, blue](dr)++(-22.5:0.7cm) to ++(180:1.29cm);
\draw [thick, blue](dl)++(67.5:0.7cm) to ++(-67.5:0.99cm);
\draw [thick, blue](dl)++(157.5:0.7cm) to ++(-67.5:0.99cm);
\draw [thick, blue](ul)++(112.5:0.7cm) to ++(270:1.29cm);
\draw [thick, blue](ul)++(67.5:0.7cm) to ++(270:1.29cm);
\draw [thick, blue](ur)++(22.5:0.7cm) to ++(-157.5:1.4cm);
\draw [thick, blue](d)++(-22.5:0.7cm) to ++(157.5:1.4cm);
\draw [thick, blue](uml)++(112.5:0.7cm) to ++(-67.5:1.4cm);
\draw [thick, blue](dmc)++(112.5:0.7cm) to ++(-45:1.29cm);
\draw [thick, blue](dmc)++(157.5:0.7cm) to ++(-45:1.29cm);
\draw [thick, blue](umr)++(112.5:0.7cm) to ++(-22.5:0.99cm);
\draw [thick, blue](umr)++(202.5:0.7cm) to ++(-22.5:0.99cm);
\fill [gray, opacity=.4] (u)--(ul)--(dl)--(d)--(dr)--(ur)--(u);
\end{tikzpicture}}
\end{center}
\end{figure}
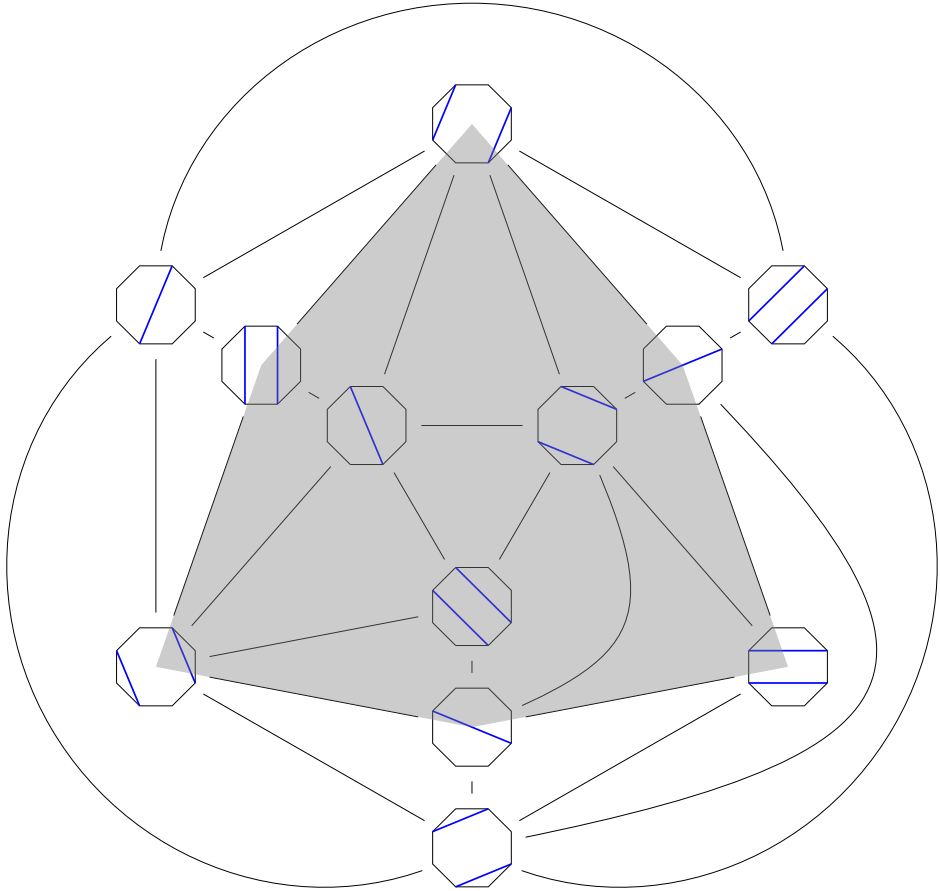
\end{example}
\subsection{Type $D_n$}\label{subsec:focus-Dn}
In this subsection, we focus on cluster complexes of $D_n$ type, a cluster complex $\Delta(\xx,B)$ such that $B$ is mutation equivalent to an exchange matrix of the tree $D_n$ type. Notably, we consider the case that $B$ is an exchange matrix of the tree $D_n$ type (in other words, $Q_B$ has a diagram of $D_n$ type in Figure \ref{fig:valued-quiver} as its underlying diagram).

Similar to $A_n$ type or $B_n,C_n$ type, we begin by explaining the realization of cluster structure by marked surfaces, which is given by \cite{fzii}*{Subsection 12.4}.
Let $S$ be a regular $2n$-gon. We consider orbits of their diagonals by a $\pi$ rotation $\Theta$ as $B_n$ and $C_n$ type. However, we provide two distinguished orbits for (an orbit of) a diameter. In this work, we distinguish these two by the diameter labels 1 and 2. We say that, for orbits, $\bar{\ell}$ is \emph{compatible} with $\bar{\ell'}$ if they are diagonals falling under any of the following:
\begin{itemize}
\item $\bar{\ell}$ and $\bar{\ell}'$ are diameters with the same diameter label,
\item $\bar{\ell}$ and $\bar{\ell}'$ are the same diameters with different diameter labels,
\item $\bar{\ell}$ and $\bar{\ell}'$ do not cross each other in the interior of $S$.
\end{itemize}
We refer to a set of pairwise compatible orbits of diagonals a \emph{compatible set}. We label these orbits by $\{\bar{\ell}_1,\dots,\bar{\ell}_n\}$, and define the \emph{labeled maximal compatible set} in the same way as $A_n$ or $B_n,C_n$ type. We note that a maximal compatible set $T$ consists of $n-2$ orbits composed of two diagonals and 2 (distinguished) diameters. We assume that these 2 diameters are $\bar{\ell}_{n-1}=\{\ell_{n-1}\},\bar{\ell}_n=\{\ell_n\}$. We define a matrix $B_T$ as the matrix defined from $T$ as follows. Let $S_{\ell_n}$ be one of $(n+2)$-gon cut out by a diameter $\ell_n$. For each triangle $\Delta$ in a triangulation of $S_{\ell_n}$ given by $T$, we define the $n\times n$ integer matrix $B^\Delta=(b^\Delta_{ij})$ by setting
\begin{align*}
b^\Delta_{ij} =\begin{cases}
 1 \quad&\text{if $\Delta$ has diagonals labeled $\bar{\ell}_{i}$ and  $\bar{\ell}_{j}$, with $\bar{\ell}_j$ following $\bar{\ell}_i$ in clockwise order},\\
 -1 \quad &\text{if the same holds as the above, in counterclockwise order,}\\
 0 \quad &\text{otherwise}.
\end{cases}
\end{align*}
The matrix $B=B_T=(b_{ij})$ is then defined by
\begin{align*}
B=\sum_\Delta B^\Delta,
\end{align*}
where the summation runs over all triangles $\Delta$ in a triangulation of $S_{\ell_n}$ given by $T$. This matrix is independent of the choice of diameter $\ell_n$ in $T$. Therefore, we can use $\ell_{n-1}$ instead of $\ell_n$.
\begin{example}
Let $T=(\bar{\ell_1},\bar{\ell_2},\bar{\ell_3},\bar{\ell_4})$ be a labeled maximal compatible set given as in Figure \ref{fig:triangulation-matrix-correspondence-Dn}. Then, we have $B_T=\begin{bmatrix}
0&0&1&-1\\
0&0&-1&1\\
-1&1&0&0\\
1&-1&0&0
\end{bmatrix}$. The corresponding quiver is $
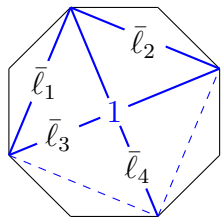
\begin{xy}(-10,0)*+{1}="H",(0,0)*+{2}="I",(10,10)*+{3}="J", (10,-10)*+{{4}}="K" \ar@{<-}"I";"J"  \ar@ {<-}"K";"I" \ar@{->}"H";"J" \ar@{<-}"H";"K"\end{xy}$.
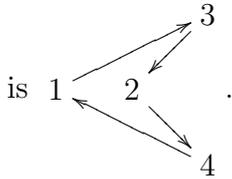
\begin{figure}[ht]
\caption{Labeled maximal compatible set of an octagon}\label{fig:triangulation-matrix-correspondence-Dn}
\[
\begin{tikzpicture}[baseline=0cm]
\coordinate (0) at (0,0);
\coordinate (1) at (22.5:1.5);
\coordinate (2) at (67.5:1.5);
\coordinate (3) at (112.5:1.5);
\coordinate (4) at (157.5:1.5);
\coordinate (5) at (202.5:1.5);
\coordinate (6) at (247.5:1.5);
\coordinate (7) at (292.5:1.5);
\coordinate (8) at (337.5:1.5);
\draw (1) to (2);
\draw (2) to (3);
\draw (3) to (4);
\draw (4) to (5);
\draw (5) to (6);
\draw (6) to (7);
\draw (7) to (8);
\draw (8) to (1);
\draw [blue,thick] (1) to (3);
\draw [blue,thick] (3) to (5);
\draw [blue,thick] (1) to (5);
\draw [blue,thick] (3) to (7);
\draw [blue,dashed] (5) to (7);
\draw [blue,dashed] (7) to (1);
\fill [white](67.5:1) circle (3mm);
\fill [white](157.5:1) circle (3mm);
\fill [white](202.5:0.8) circle (3mm);
\fill [white](292.5:0.8) circle (3mm);
\fill [white](0,0) circle (2mm);
\node at (0,0) {\textcolor{blue}{1}};
\node at (157.5:1) {$\bar{\ell}_1$};
\node at (67.5:1) {$\bar{\ell}_2$};
\node at (202.5:0.8) {$\bar{\ell}_3$};
\node at (292.5:0.8){$\bar{\ell}_4$};
\end{tikzpicture}
\]
\end{figure}
\end{example}
For $i\in \{1,\dots,n\}$ and a maximal compatible set $T$, we define a \emph{flip $\varphi_i(T)$ in direction $i$} as a transformation that obtains the unique maximal compatible set $T'$ from $T$ by removing $\bar{\ell_i}$ and adding the other diagonal $\bar{\ell'_i}$ to $T\setminus\{\bar{\ell_i}\}$.
We explain that flips establish the same structure as seed mutations of $D_n$ type. We define a \emph{labeled triangulation pair} $(T,B_T)$, a \emph{flip of a labeled triangulation pair} $\varphi_i(T,B_T):=(\varphi_i(T),B_{\varphi_i(T)})$, a \emph{triangulation pattern} $P\colon t\mapsto \Xi_t$, and \emph{arc complex} $\Delta(T,B_T)$ in the same way as Subsection \ref{subsec:focus-An}. In parallel with $A_n$ and $B_n, C_n$ type, the following property can be easily obtained by \cite{fzii}*{Subsection 12.4} and checking exchange relations \cite{fzii}*{(12.19)--(12.23)}.
\begin{theorem}\label{thm:arc-variable-Dn}
Let $S$ be a regular $2n$-gon and $T$ be a maximal compatible set.
\begin{itemize}
\item[(1)]We have $B_{\varphi_i(T)}=\mu_i(B_T)$, where $\mu_i$ is a mutation in direction $i$ of an exchange matrix.
\item[(2)]We have the canonical bijection between the sets of all orbits of diagonals of $S$ and the set of cluster variables in $\{\xx_t\}_{t\in\TT_n}$ by the correspondence $\bar{\ell}_{i;t}\mapsto x_{i;t}$. Furthermore, this bijection induces an isomorphism between an arc complex $\Delta(T,B_T)$ and a cluster complex $\Delta(\xx,B_T)$.
\end{itemize}
\end{theorem}
We note that $B_T$ is a skew-symmetric matrix which is mutation equivalent to an exchange matrix of tree $D_n$ type.
\begin{remark}
For convenience, we regard $D_2$ type as $A_1 \times A_1$ type. Further, the Dynkin diagram of the $D_3$ type has the same shape as that of $A_3$. However, they differ in terms of the order of their vertices. In the $A_3$ type, a vertex indexed by $1$ is an end of one of the diagrams. Conversely, in the $D_3$ type, it is the central vertex of the diagrams.
\end{remark}
\begin{example}\label{ex:correspondence-arc-variable-D3}
 Let $B=\begin{bmatrix}
 0&-1&-1\\
 1&0&0\\
 1&0&0
 \end{bmatrix}$ of type $D_3$. The corresponding quiver is $
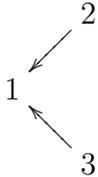
\begin{xy}(0,0)*+{1}="I",(10,10)*+{2}="J", (10,-10)*+{{3}}="K" \ar@{<-}"I";"J"  \ar@ {->}"K";"I" \end{xy}$. Let $T$ be a set $\left\{\bar{\ell}_1=\begin{tikzpicture}[baseline=-1mm,scale=0.3]
 \coordinate (u) at(90:1); \coordinate (lu) at(150:1); \coordinate (ld) at(-150:1);
 \coordinate (ru) at(30:1); \coordinate (rd) at(-30:1); \coordinate (d) at(-90:1);
 \draw (u)--(lu)--(ld)--(d)--(rd)--(ru)--(u); \node at(0,1.3) {}; \node at(0,-1.3) {}; \draw[blue,thick] (ru)--(d); \draw[blue,thick] (u)--(ld);
\end{tikzpicture}, \ \bar{\ell}_2=
\begin{tikzpicture}[baseline=-1mm,scale=0.3]
 \coordinate (u) at(90:1); \coordinate (lu) at(150:1); \coordinate (ld) at(-150:1);
 \coordinate (ru) at(30:1); \coordinate (rd) at(-30:1); \coordinate (d) at(-90:1);
 \draw (u)--(lu)--(ld)--(d)--(rd)--(ru)--(u); \node at(0,1.3) {}; \node at(0,-1.3) {}; \draw[blue,thick] (u)--(d);
\fill [white](0,0) circle (4mm);
\node at (0,0){\textcolor{blue}{\tiny 1}};
\end{tikzpicture},\ \bar{\ell}_3=
\begin{tikzpicture}[baseline=-1mm,scale=0.3]
 \coordinate (u) at(90:1); \coordinate (lu) at(150:1); \coordinate (ld) at(-150:1);
 \coordinate (ru) at(30:1); \coordinate (rd) at(-30:1); \coordinate (d) at(-90:1);
 \draw (u)--(lu)--(ld)--(d)--(rd)--(ru)--(u); \node at(0,1.3) {}; \node at(0,-1.3) {}; \draw[blue,thick] (u)--(d);
 \fill [white](0,0) circle (4mm);
\node at (0,0){\textcolor{blue}{\tiny 2}};
\end{tikzpicture}
\right\}$ of orbits of diagonals. Then, we have $B=B_T$. For example, we consider a cluster pattern $P^{\Sigma}(\xx,B)$ and a triangulation pattern $P^\Xi(T,B_T)$. Then, we have $\mu_{x_3}\left[\left(x_1,x_2,x_3\right)\right]=\left[\left(x_1,x_2,\dfrac{x_2+1}{x_3}\right)\right]$ by a mutation in direction $x_3$. Meanwhile, we have $\varphi_{\ell_3}\left(
\begin{tikzpicture}[baseline=-1mm,scale=0.3]
 \coordinate (u) at(90:1);
 \coordinate (lu) at(150:1);
 \coordinate (ld) at(-150:1);
 \coordinate (ru) at(30:1);
 \coordinate (rd) at(-30:1);
 \coordinate (d) at(-90:1);
 \draw (u)--(lu)--(ld)--(d)--(rd)--(ru)--(u);
 \node at(0,1.3) {};
 \node at(0,-1.3) {};
 \draw[blue,thick] (u)--(ld);
 \draw[blue,thick] (u)--(d);
 \draw[blue,thick] (d)--(ru);
 \fill [white](0,0) circle (4mm);
\node at (0,0){\textcolor{blue}{\tiny 12}};
\end{tikzpicture}\right)=
\begin{tikzpicture}[baseline=-1mm,scale=0.3]
 \coordinate (u) at(90:1);
 \coordinate (lu) at(150:1);
 \coordinate (ld) at(-150:1);
 \coordinate (ru) at(30:1);
 \coordinate (rd) at(-30:1);
 \coordinate (d) at(-90:1);
 \draw (u)--(lu)--(ld)--(d)--(rd)--(ru)--(u);
 \node at(0,1.3) {};
 \node at(0,-1.3) {};
 \draw[blue,thick] (u)--(ld);
 \draw[blue,thick] (u)--(d);
 \draw[blue,thick] (ld)--(ru);
 \draw[blue,thick] (d)--(ru);
 \fill [white](0,0) circle (4mm);
\node at (0,0){\textcolor{blue}{\tiny 1}};
\end{tikzpicture}
$ by a flip in direction $\ell_3$. Therefore, for the canonical bijection, $\dfrac{x_1+1}{x_3}$ corresponds to \begin{tikzpicture}[baseline=-1mm,scale=0.3]
 \coordinate (u) at(90:1);
 \coordinate (lu) at(150:1);
 \coordinate (ld) at(-150:1);
 \coordinate (ru) at(30:1);
 \coordinate (rd) at(-30:1);
 \coordinate (d) at(-90:1);
 \draw (u)--(lu)--(ld)--(d)--(rd)--(ru)--(u);
 \node at(0,1.3) {};
 \node at(0,-1.3) {};
 \draw[blue,thick] (ld)--(ru);
 \fill [white](0,0) circle (4mm);
\node at (0,0){\textcolor{blue}{\tiny 1}};
\end{tikzpicture}.
The canonical bijection between the set of cluster variables and the set of all orbits of diagonals is given in Table \ref{exampDn}, and the cluster complex of $D_3$ type is given in Figure \ref{D3complex}.
\begin{table}[ht]
\begin{tabular}{c|c}
 Cluster variable $x$ & diagonal corresponding to $x$
\end{tabular}
\vspace{2mm}\\
\begin{minipage}{0.3\hsize}
\begin{center}\begin{tabular}{c|c}
 $x_1$ &
\begin{tikzpicture}[baseline=-1mm,scale=0.5]
 \coordinate (u) at(90:1); \coordinate (lu) at(150:1); \coordinate (ld) at(-150:1);
 \coordinate (ru) at(30:1); \coordinate (rd) at(-30:1); \coordinate (d) at(-90:1);
 \draw (u)--(lu)--(ld)--(d)--(rd)--(ru)--(u); \node at(0,1.3) {}; \node at(0,-1.3) {}; \draw[blue,thick] (u)--(ld);
 \draw[blue,thick] (d)--(ru);
\end{tikzpicture}
\\\hline
 $x_2$ &
\begin{tikzpicture}[baseline=-1mm,scale=0.5]
 \coordinate (u) at(90:1); \coordinate (lu) at(150:1); \coordinate (ld) at(-150:1);
 \coordinate (ru) at(30:1); \coordinate (rd) at(-30:1); \coordinate (d) at(-90:1);
 \draw (u)--(lu)--(ld)--(d)--(rd)--(ru)--(u); \node at(0,1.3) {}; \node at(0,-1.3) {}; \draw[blue,thick] (u)--(d);
 \fill [white](0,0) circle (4mm);
\node at (0,0){\textcolor{blue}{\footnotesize 1}};
\end{tikzpicture}
\\\hline
 $x_3$ &
\begin{tikzpicture}[baseline=-1mm,scale=0.5]
 \coordinate (u) at(90:1); \coordinate (lu) at(150:1); \coordinate (ld) at(-150:1);
 \coordinate (ru) at(30:1); \coordinate (rd) at(-30:1); \coordinate (d) at(-90:1);
 \draw (u)--(lu)--(ld)--(d)--(rd)--(ru)--(u); \node at(0,1.3) {}; \node at(0,-1.3) {}; \draw[blue,thick] (u)--(d);
  \fill [white](0,0) circle (4mm);
\node at (0,0){\textcolor{blue}{\footnotesize 2}};
\end{tikzpicture}
\end{tabular}\end{center}
\end{minipage}
\begin{minipage}{0.3\hsize}
\begin{center}\begin{tabular}{c|c}
 $\cfrac{x_2x_3+1}{x_1}$ &
\begin{tikzpicture}[baseline=-1mm,scale=0.5]
 \coordinate (u) at(90:1); \coordinate (lu) at(150:1); \coordinate (ld) at(-150:1);
 \coordinate (ru) at(30:1); \coordinate (rd) at(-30:1); \coordinate (d) at(-90:1);
 \draw (u)--(lu)--(ld)--(d)--(rd)--(ru)--(u); \node at(0,1.3) {}; \node at(0,-1.3) {}; \draw[blue,thick] (lu)--(d); \draw[blue,thick] (u)--(rd);
\end{tikzpicture}
\\\hline
 $\cfrac{x_1+1}{x_2}$ &
\begin{tikzpicture}[baseline=-1mm,scale=0.5]
 \coordinate (u) at(90:1); \coordinate (lu) at(150:1); \coordinate (ld) at(-150:1);
 \coordinate (ru) at(30:1); \coordinate (rd) at(-30:1); \coordinate (d) at(-90:1);
 \draw (u)--(lu)--(ld)--(d)--(rd)--(ru)--(u); \node at(0,1.3) {}; \node at(0,-1.3) {}; \draw[blue,thick] (ru)--(ld);
  \fill [white](0,0) circle (4mm);
\node at (0,0){\textcolor{blue}{\footnotesize 2}};
\end{tikzpicture}
\\\hline
 $\cfrac{x_1+1}{x_3}$ &
\begin{tikzpicture}[baseline=-1mm,scale=0.5]
 \coordinate (u) at(90:1); \coordinate (lu) at(150:1); \coordinate (ld) at(-150:1);
 \coordinate (ru) at(30:1); \coordinate (rd) at(-30:1); \coordinate (d) at(-90:1);
 \draw (u)--(lu)--(ld)--(d)--(rd)--(ru)--(u); \node at(0,1.3) {}; \node at(0,-1.3) {}; \draw[blue,thick] (ru)--(ld);
  \fill [white](0,0) circle (4mm);
\node at (0,0){\textcolor{blue}{\footnotesize 1}};
\end{tikzpicture}
\end{tabular}\end{center}
\end{minipage}
\begin{tabular}{c|c}
 $\cfrac{x_1+x_2x_3+1}{x_1x_2}$ &
\begin{tikzpicture}[baseline=-1mm,scale=0.5]
 \coordinate (u) at(90:1); \coordinate (lu) at(150:1); \coordinate (ld) at(-150:1);
 \coordinate (ru) at(30:1); \coordinate (rd) at(-30:1); \coordinate (d) at(-90:1);
 \draw (u)--(lu)--(ld)--(d)--(rd)--(ru)--(u); \node at(0,1.3) {}; \node at(0,-1.3) {}; \draw[blue,thick] (lu)--(rd);
  \fill [white](0,0) circle (4mm);
\node at (0,0){\textcolor{blue}{\footnotesize 2}};
\end{tikzpicture}
\\\hline
 $\cfrac{x_1+x_2x_3+1}{x_1x_3}$ &
\begin{tikzpicture}[baseline=-1mm,scale=0.5]
 \coordinate (u) at(90:1); \coordinate (lu) at(150:1); \coordinate (ld) at(-150:1);
 \coordinate (ru) at(30:1); \coordinate (rd) at(-30:1); \coordinate (d) at(-90:1);
 \draw (u)--(lu)--(ld)--(d)--(rd)--(ru)--(u); \node at(0,1.3) {}; \node at(0,-1.3) {}; \draw[blue,thick] (rd)--(lu);
  \fill [white](0,0) circle (4mm);
\node at (0,0){\textcolor{blue}{\footnotesize 1}};
\end{tikzpicture}
\\\hline
 $\cfrac{x_1^2+x_2x_3+2x_1+1}{x_1x_2x_3}$ &
\begin{tikzpicture}[baseline=-1mm,scale=0.5]
 \coordinate (u) at(90:1); \coordinate (lu) at(150:1); \coordinate (ld) at(-150:1);
 \coordinate (ru) at(30:1); \coordinate (rd) at(-30:1); \coordinate (d) at(-90:1);
 \draw (u)--(lu)--(ld)--(d)--(rd)--(ru)--(u); \node at(0,1.3) {}; \node at(0,-1.3) {}; \draw[blue,thick] (lu)--(ru); \draw[blue,thick] (ld)--(rd);
\end{tikzpicture}
  \end{tabular}
 \vspace{5mm}
 \caption{Canonical bijection between the set of cluster variables and the set of orbits of diagonal of $D_3$ type}\label{exampDn}
\end{table}
\begin{figure}[ht]
\caption{Cluster complex of $D_3$ type \label{D3complex}}
\begin{center}
\scalebox{0.8}{
\begin{tikzpicture}
\coordinate (0) at (0,0);
\coordinate (u*) at (90:5.33);
\coordinate (u) at (90:4);
\coordinate (ul) at (150:4);
\coordinate (ur) at (30:4);
\coordinate (uml) at (150:2);
\coordinate (umr) at (30:2);
\coordinate (dmc) at (-90:2);
\coordinate (dl) at (-150:4);
 \coordinate (dl*) at (-150:5.33);
\coordinate (dr) at (-30:4);
\coordinate (dr*) at (-30:5.33);
\coordinate (d) at (-90:4);
\draw (u) to (ul);
\draw (u) to (ur);
\draw (ul) to (uml);
\draw (ur) to (umr);
\draw (umr) to (uml);
\draw (u) to (uml);
\draw (u) to (umr);
\draw (uml) to (dl);
\draw (umr) to (dr);
\draw (ul) to (dl);
\draw (ur) to (dr);
\draw (uml) to (dmc);
\draw (umr) to (dmc);
\draw (dl) to (dmc);
\draw (dr) to (dmc);
\draw (dl) to (d);
\draw (dmc) to (d);
\draw (dr) to (d);
\draw(ul) [out=90,in=180]to (u*);
\draw(ur) [out=90,in=0]to (u*);
\draw(ul) [out=210,in=120]to (dl*);
\draw(d) [out=210,in=300]to (dl*);
\draw(d) [out=330,in=240]to (dr*);
\draw(ur) [out=330,in=60]to (dr*);
\fill [white](u) circle (0.9cm);
\fill[white] (ul) circle (0.9cm);
\fill[white] (ur) circle (0.9cm);
\fill[white] (uml) circle (0.9cm);
\fill[white] (umr) circle (0.9cm);
\fill[white] (dmc) circle (0.9cm);
\fill[white] (dl) circle (0.9cm);
\fill[white] (dr) circle (0.9cm);
\fill[white] (d) circle (0.9cm);
\draw (u)++(210:0.7cm) to ++(90:0.7cm);
\draw (u)++(150:0.7cm) to ++(30:0.7cm);
\draw (u)++(90:0.7cm) to ++(330:0.7cm);
\draw (u)++(30:0.7cm) to ++(270:0.7cm);
\draw (u)++(330:0.7cm) to ++(210:0.7cm);
\draw (u)++(270:0.7cm) to ++(150:0.7cm);
\draw (ur)++(210:0.7cm) to ++(90:0.7cm);
\draw (ur)++(150:0.7cm) to ++(30:0.7cm);
\draw (ur)++(90:0.7cm) to ++(330:0.7cm);
\draw (ur)++(30:0.7cm) to ++(270:0.7cm);
\draw (ur)++(330:0.7cm) to ++(210:0.7cm);
\draw (ur)++(270:0.7cm) to ++(150:0.7cm);
\draw (ul)++(210:0.7cm) to ++(90:0.7cm);
\draw (ul)++(150:0.7cm) to ++(30:0.7cm);
\draw (ul)++(90:0.7cm) to ++(330:0.7cm);
\draw (ul)++(30:0.7cm) to ++(270:0.7cm);
\draw (ul)++(330:0.7cm) to ++(210:0.7cm);
\draw (ul)++(270:0.7cm) to ++(150:0.7cm);
\draw (umr)++(210:0.7cm) to ++(90:0.7cm);
\draw (umr)++(150:0.7cm) to ++(30:0.7cm);
\draw (umr)++(90:0.7cm) to ++(330:0.7cm);
\draw (umr)++(30:0.7cm) to ++(270:0.7cm);
\draw (umr)++(330:0.7cm) to ++(210:0.7cm);
\draw (umr)++(270:0.7cm) to ++(150:0.7cm);
\draw (uml)++(210:0.7cm) to ++(90:0.7cm);
\draw (uml)++(150:0.7cm) to ++(30:0.7cm);
\draw (uml)++(90:0.7cm) to ++(330:0.7cm);
\draw (uml)++(30:0.7cm) to ++(270:0.7cm);
\draw (uml)++(330:0.7cm) to ++(210:0.7cm);
\draw (uml)++(270:0.7cm) to ++(150:0.7cm);
\draw (dl)++(210:0.7cm) to ++(90:0.7cm);
\draw (dl)++(150:0.7cm) to ++(30:0.7cm);
\draw (dl)++(90:0.7cm) to ++(330:0.7cm);
\draw (dl)++(30:0.7cm) to ++(270:0.7cm);
\draw (dl)++(330:0.7cm) to ++(210:0.7cm);
\draw (dl)++(270:0.7cm) to ++(150:0.7cm);
\draw (dr)++(210:0.7cm) to ++(90:0.7cm);
\draw (dr)++(150:0.7cm) to ++(30:0.7cm);
\draw (dr)++(90:0.7cm) to ++(330:0.7cm);
\draw (dr)++(30:0.7cm) to ++(270:0.7cm);
\draw (dr)++(330:0.7cm) to ++(210:0.7cm);
\draw (dr)++(270:0.7cm) to ++(150:0.7cm);
\draw (dmc)++(210:0.7cm) to ++(90:0.7cm);
\draw (dmc)++(150:0.7cm) to ++(30:0.7cm);
\draw (dmc)++(90:0.7cm) to ++(330:0.7cm);
\draw (dmc)++(30:0.7cm) to ++(270:0.7cm);
\draw (dmc)++(330:0.7cm) to ++(210:0.7cm);
\draw (dmc)++(270:0.7cm) to ++(150:0.7cm);
\draw (d)++(210:0.7cm) to ++(90:0.7cm);
\draw (d)++(150:0.7cm) to ++(30:0.7cm);
\draw (d)++(90:0.7cm) to ++(330:0.7cm);
\draw (d)++(30:0.7cm) to ++(270:0.7cm);
\draw (d)++(330:0.7cm) to ++(210:0.7cm);
\draw (d)++(270:0.7cm) to ++(150:0.7cm);
\draw [thick, blue](uml)++(150:0.7cm) to ++(-30:1.39cm);
\draw [thick, blue](umr)++(30:0.7cm) to ++(210:1.39cm);
\draw [thick, blue](dmc)++(90:0.7cm) to ++(270:1.39cm);
\draw [thick, blue](dl)++(90:0.7cm) to ++(-60:1.21cm);
\draw [thick, blue](dl)++(150:0.7cm) to ++(-60:1.21cm);
\draw [thick, blue](dr)++(90:0.7cm) to ++(240:1.21cm);
\draw [thick, blue](dr)++(30:0.7cm) to ++(240:1.21cm);
\draw [thick, blue](u)++(150:0.7cm) to ++(0:1.21cm);
\draw [thick, blue](u)++(210:0.7cm) to ++(0:1.21cm);
\draw [thick, blue](ul)++(150:0.7cm) to ++(-30:1.39cm);
\draw [thick, blue](d)++(90:0.7cm) to ++(270:1.39cm);
\draw [thick, blue](ur)++(30:0.7cm) to ++(210:1.39cm);
\fill [white](umr) circle (2mm);
\node at (umr){\textcolor{blue}{1}};
\fill [white](uml) circle (2mm);
\node at (uml){\textcolor{blue}{1}};
\fill [white](dmc) circle (2mm);
\node at (dmc){\textcolor{blue}{1}};
\fill [white](ur) circle (2mm);
\node at (ur){\textcolor{blue}{2}};
\fill [white](ul) circle (2mm);
\node at (ul){\textcolor{blue}{2}};
\fill [white](d) circle (2mm);
\node at (d){\textcolor{blue}{2}};
\end{tikzpicture}}
\end{center}
\end{figure}
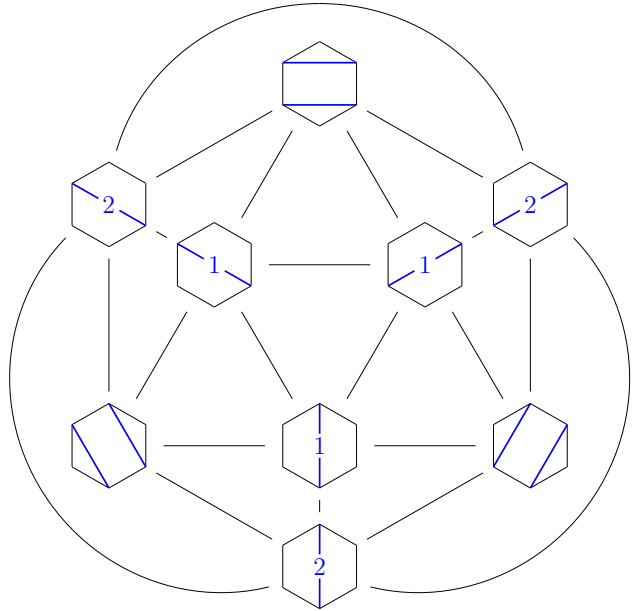
\end{example}
By using the canonical bijection given by Theorem \ref{thm:arc-variable-Dn}, we identify orbits of diagonal of a $2n$-gon with cluster variables of a cluster pattern of $D_n$ type.
\begin{definition}
Let $K$ and $L$ be simplicial complexes the vertex sets of which are disjoint, and $M$ and $N$ be subcomplexes of $K$ and $L$, respectively, such that $M\cong N$. We define a \emph{complex $G$ obtained by gluing $K$ and $L$ together along $M$ and $N$} as $G=K'\cup L'$, where $K'\cong K$, $L\cong L'$ and $K'\cap L'\cong M\cong N$.
\end{definition}
We denote by $I$ the vertex set consisting of all cluster variables in the initial cluster $\xx$. We define $\Delta_{\mathrm{glue}}$ as a full subcomplex of $\Delta(\xx,B)$ the vertex set of which consists of all vertices of $\st_{\Delta(\xx,B)}(I\setminus\{x_{n-1},x_n\})$ which are not in $\{x_1,\dots,x_{n-2}\}$ and compatible with both $x_{n-1}$ and ${x_n}$.
The main theorem in the subsection is given below.
\begin{theorem}\label{thm:positive-simplex-D_n}
Let $\Delta(\xx,B)$ be a cluster complex of $D_n$ type. Suppose $Q_B$ has one of the following forms.
\begin{align}
\begin{xy}
   (50,0)*{\circ}="A",(60,0)*{\circ}="B",(70,0)*{\circ}="C",(80,0)*{\circ}="D",(85,0)*{}="E",(90,0)*{}="F",(95,0)*{\circ}="G",(105,5)*{\circ}="H",(105,-5)*{\circ}="I",\ar@{->} "A";"B", \ar@{->} "B";"C", \ar@{->} "C";"D", \ar@{-} "D";"E", \ar@{.} "E";"F",\ar@{->} "F";"G",\ar@{->} "G";"H",\ar@{->} "G";"I"
\end{xy}\label{assumption:Dn1}\\
\begin{xy}
   (50,0)*{\circ}="A",(60,0)*{\circ}="B",(70,0)*{\circ}="C",(80,0)*{\circ}="D",(85,0)*{}="E",(90,0)*{}="F",(95,0)*{\circ}="G",(105,5)*{\circ}="H",(105,-5)*{\circ}="I",\ar@{<-} "A";"B", \ar@{<-} "B";"C", \ar@{<-} "C";"D", \ar@{<-} "D";"E", \ar@{.} "E";"F",\ar@{-} "F";"G",\ar@{<-} "G";"H",\ar@{<-} "G";"I"
\end{xy}\label{assumption:Dn2}
\end{align}
Then, $\Delta^+(\xx,B)$ is isomorphic to a complex obtained by gluing $\join(\Delta_0,\Delta_1,\Delta(A_{n-3}))$ and $\st_{\Delta(\xx,B)}(I\setminus\{x_{n-1},x_n\})$ together along $\join (\Delta_1,\Delta(A_{n-3}))$ and $\Delta_{\mathrm{glue}}$, where $\Delta_i$ is the $i$-dimensional simplex.
\end{theorem}
\begin{remark}\label{rem:equivalent-condition-Dn}
A matrix $B$ satisfies the assumption in Theorem \ref{thm:positive-simplex-D_n} if and only if the matrix has the following property. For an orbit set of an $2n$-gon $S$ obtained by a maximal compatible set $T$ which gives $B_T=B$ and diameters $\bar{\ell}_{n-1}$ and $\bar{\ell}_n$ of $T$, $\bar{\ell}_{n-1}$ and $\bar{\ell}_{n}$ are the same diameter with different diameter labels, and there is a vertex $v$ of $S_{\ell_n}$ such that all diagonals in $T$ in $S_{\ell_n}$ share $v$ (see Figure \ref{ex:assumption-of-Dn}). The shape shown at left in Figure \ref{ex:assumption-of-Dn} corresponds to the quiver \eqref{assumption:Dn1}, and that at right corresponds to \eqref{assumption:Dn2}. Because of symmetry, we may assume the version at left in Figure \ref{ex:assumption-of-Dn} without loss of generality.
\begin{figure}[ht]
\caption{Triangulation corresponding to $B$ satisfying the assumption in Theorem  \ref{thm:positive-simplex-D_n}}\label{ex:assumption-of-Dn}
\[
\begin{tikzpicture}
 \coordinate (u1) at(67.5:1.5); \coordinate (u2) at(112.5:1.5); \coordinate (lu) at(157.5:1.5);
 \coordinate (ld) at(-157.5:1.5); \coordinate (ru) at(22.5:1.5); \coordinate (rd) at(-22.5:1.5);  \coordinate (d1) at(-67.5:1.5); \coordinate (d2) at(-112.5:1.5);
 \draw (u1)--(u2)--(lu)--(ld)--(d2)--(d1)--(rd)--(ru)--(u1);  \draw[blue,thick] (u1)--(lu); \draw[blue,thick] (u1)--(ld); \draw[blue,thick] (u1)--(d2);
 \draw[blue,thick] (d2)--(rd); \draw[blue,thick] (d2)--(ru);
 \node at (67.5:1.8) {$v$};
 \fill [white](0,0) circle (2mm);
\node at (0,0){\textcolor{blue}{12}};
\end{tikzpicture}
\hspace{10mm}
\begin{tikzpicture}
 \coordinate (u1) at(67.5:1.5); \coordinate (u2) at(112.5:1.5); \coordinate (lu) at(157.5:1.5);
 \coordinate (ld) at(-157.5:1.5); \coordinate (ru) at(22.5:1.5); \coordinate (rd) at(-22.5:1.5);  \coordinate (d1) at(-67.5:1.5); \coordinate (d2) at(-112.5:1.5);
 \draw (u1)--(u2)--(lu)--(ld)--(d2)--(d1)--(rd)--(ru)--(u1);  \draw[blue,thick] (u2)--(ru); \draw[blue,thick] (u2)--(rd); \draw[blue,thick] (u2)--(d1);
 \draw[blue,thick] (d1)--(ld); \draw[blue,thick] (d1)--(lu);
 \node at (112.5:1.8) {$v$};
 \fill [white](0,0) circle (2mm);
\node at (0,0){\textcolor{blue}{12}};
\end{tikzpicture}
\]
\end{figure}
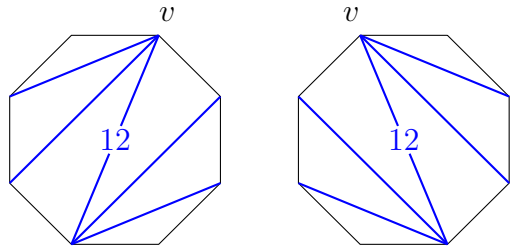
\end{remark}
\begin{lemma}\label{lem:interior-D_n}
Let $B=(b_{ij})$ be an exchange matrix of tree $D_n$ type. For a cluster complex $\Delta(\xx,B)$, there exist just $n-2$ vertices non-linked to $I$.
\end{lemma}
\begin{proof}
By the canonical bijection between cluster variables and orbits of diagonals of a regular $2n$-gon $S$, it suffices to show that there are just $n-2$ orbits of diagonals such that each of them is not compatible with any orbits of diagonals corresponding to the initial cluster variables at the interior of $S$. Since the initial exchange matrix $B$ is of tree $D_n$ type, the initial maximal compatible set $T$ in which there are only two vertices of $S$ shared by only one triangle. Let $v, v'$ be vertices of these two, and we consider an $(n+1)$-gon $S_{\ell_n}$, which is cut out by a diameter $\ell_n$ and not including $v$. Let $v_1,\dots,v_{n-2}$ be the vertices of an $(n+1)$-gon $S_{\ell_n}$, which are not endpoints of $\ell_n$ or $v'$, and $\ell_{v_1},\dots,\ell_{v_{n-2}}$ diagonals combining $v$ with $v_1,\dots,v_{n-2}$, respectively. Then, orbits of $\ell_{v_1},\dots,\ell_{v_{n-2}}$ by $\Theta$ are the desired orbits.
\end{proof}
We denote by $I'=\{x'_1,\dots,x'_{n-2}\}$ the unique set of $n$ vertices non-linked to $I$. 
We introduce some additional notation for convenience. Let $S$ be a regular $2n$-gon and $T$ a maximal compatible set corresponding to $I$ and $T'$ a compatible set corresponding to $I'$. We name four vertices $v_1,v_2,v_3,v_4$ such that $\{v_1,v_2\}$ and $\{v_3,v_4\}$ are pairs of vertices of $S$ sharing only one triangle in a triangulation given by $T$ and $T'$, respectively. Hereafter, let $(v,v')$ denote the edge connecting $v$ and $v'$. We assume the edges $(v_1,v_3)$ and $(v_2,v_4)$ are diagonals of $S$. We name the rest of vertices $v_5,\dots,v_{2n}$ from the right neighbor of $v_1$ in clockwise order. (See Figure \ref{fig:notations}.)
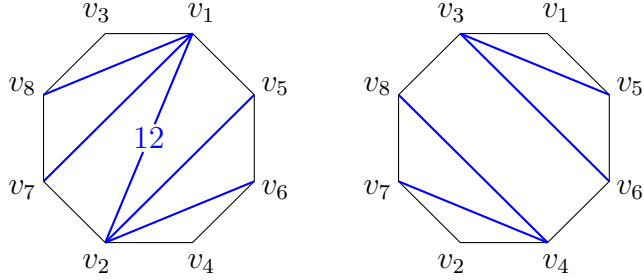
\begin{figure}[ht]
\caption{$T$ (left) and $T'$ (right) in the case that $n=4$ }\label{fig:notations}
\[
\begin{tikzpicture}[baseline=0mm]
 \coordinate (u1) at(67.5:1.5); \coordinate (u2) at(112.5:1.5); \coordinate (lu) at(157.5:1.5);
 \coordinate (ld) at(-157.5:1.5); \coordinate (ru) at(22.5:1.5); \coordinate (rd) at(-22.5:1.5);  \coordinate (d1) at(-67.5:1.5); \coordinate (d2) at(-112.5:1.5);
 \draw (u1)--(u2)--(lu)--(ld)--(d2)--(d1)--(rd)--(ru)--(u1);
 \draw[blue,thick] (u1)--(lu); \draw[blue,thick] (u1)--(ld); \draw[blue,thick] (u1)--(d2);
 \draw[blue,thick] (d2)--(rd); \draw[blue,thick] (d2)--(ru);
 \node at (67.5:1.8) {$v_1$};
 \node at (-112.5:1.8) {$v_2$};
 \node at (112.5:1.8) {$v_3$};
 \node at (-67.5:1.8) {$v_4$};
 \node at (22.5:1.8) {$v_5$};
 \node at (157.5:1.8) {$v_8$};
 \node at (-157.5:1.8) {$v_7$};
 \node at (-22.5:1.8) {$v_6$};
 \fill [white](0) circle (2mm);
\node at (0){\textcolor{blue}{12}};
\end{tikzpicture}
\hspace{20pt}
\begin{tikzpicture}[baseline=0mm]
 \coordinate (u1) at(67.5:1.5); \coordinate (u2) at(112.5:1.5); \coordinate (lu) at(157.5:1.5);
 \coordinate (ld) at(-157.5:1.5); \coordinate (ru) at(22.5:1.5); \coordinate (rd) at(-22.5:1.5);  \coordinate (d1) at(-67.5:1.5); \coordinate (d2) at(-112.5:1.5);
 \draw (u1)--(u2)--(lu)--(ld)--(d2)--(d1)--(rd)--(ru)--(u1);
 \draw [blue,thick](u2)--(ru);
 \draw [blue,thick](u2)--(rd);
 \draw [blue,thick](lu)--(d1);
 \draw [blue,thick](ld)--(d1);
 \node at (67.5:1.8) {$v_1$};
 \node at (-112.5:1.8) {$v_2$};
 \node at (112.5:1.8) {$v_3$};
 \node at (-67.5:1.8) {$v_4$};
 \node at (22.5:1.8) {$v_5$};
 \node at (157.5:1.8) {$v_8$};
 \node at (-157.5:1.8) {$v_7$};
 \node at (-22.5:1.8) {$v_6$};
\end{tikzpicture}
\]
\end{figure}
\begin{lemma}\label{lem:description-positiveconeDn}
For a cluster complex $\Delta(\xx,B)$ satisfying the assumption in Theorem \ref{thm:positive-simplex-D_n}, $\st_{\Delta(\xx,B)}(I')$ is isomorphic to $\st_{\Delta(\xx,B)}(I\setminus\{x_{n-1},x_n\})$.
\end{lemma}
\begin{proof}
The situation is the same as in Lemma \ref{lem:description-positiveconeBn}.
\end{proof}
Unlike $A_n$ or $B_n,C_n$ type, the closed star of non-linked vertices to $I$ does not coincide with $\Delta^+(\xx,B)$. However, the link of non-linked vertices is isomorphic to $\Delta(D_{n-1})$; we provide a proof below.
\begin{lemma}\label{cor:boundary-closedstar}
For a cluster complex $\Delta(\xx,B)$ satisfying the assumption in Theorem \ref{thm:positive-simplex-D_n}, we have $\lk_{\Delta(\xx,B)}(I') \cong\Delta(D_{n-1})$.
\end{lemma}
\begin{proof}
 Let $S'$ be a $2n$-gon obtained by reducing $(v_1,v_3)$ and $(v_2,v_4)$ to points, respectively. Let $D$ be the set of diagonals of $S$ which are compatible with at least one orbit of $T'$ and not in $T'$ (we distinguish two labeled diameters with the same edge). These are diagonals corresponding to a vertex set of $\lk_{\Delta(\xx,B)}(I')$. Let $D'$  be the set of all diagonals of $S'$.
We consider the surjection $\psi\colon D \to D',
(v_i,v_j)\mapsto (v_i,v_j)$ (see Figure \ref{fig:sujection-D-to-D'Dn}). For diameters $(v_i,v_j)$ labeled 1 and 2, we note that $\psi$ gives a correspondence between diameters with the same label.
\begin{figure}[ht]
\caption{Surjection $D$ to $D'$ ($n=4$) }\label{fig:sujection-D-to-D'Dn}
\[
\begin{tikzpicture}[baseline=0mm]
 \coordinate (u1) at(67.5:1.5); \coordinate (u2) at(112.5:1.5); \coordinate (lu) at(157.5:1.5);
 \coordinate (ld) at(-157.5:1.5); \coordinate (ru) at(22.5:1.5); \coordinate (rd) at(-22.5:1.5);  \coordinate (d1) at(-67.5:1.5); \coordinate (d2) at(-112.5:1.5);
 \draw (u1)--(u2)--(lu)--(ld)--(d2)--(d1)--(rd)--(ru)--(u1);
 \draw [blue,thick](u1)--(rd);
 \draw [blue,thick](u2)--(ld);
 \draw [blue,thick](u2)--(ld);
 \draw [blue,thick](d1)--(ru);
 \draw [blue,thick](d2)--(lu);
 \draw [blue,thick](d1)--(ru);
 \draw [blue,thick](d2)--(lu);
 \draw [blue,thick](lu)--(ru);
 \draw [blue,thick](ld)--(rd);
 \draw [blue,thick](ld)--(ru);
 \draw [blue,thick](lu)--(rd);
 \draw [blue,thick](u2)--(d1);
 \node at (67.5:1.8) {$v_1$};
 \node at (-112.5:1.8) {$v_2$};
 \node at (112.5:1.8) {$v_3$};
 \node at (-67.5:1.8) {$v_4$};
 \node at (22.5:1.8) {$v_5$};
 \node at (157.5:1.8) {$v_8$};
 \node at (-157.5:1.8) {$v_7$};
 \node at (-22.5:1.8) {$v_6$};
 \fill [white](0,0) circle (3mm);
\node at (0,0){\textcolor{blue}{12}};
\end{tikzpicture}
\hspace{20pt}\to \hspace{20pt}
\begin{tikzpicture}[baseline=0mm]
 \coordinate (u) at(90:1.5);  \coordinate (lu) at(145:1);
 \coordinate (ld) at(-145:1); \coordinate (ru) at(35:1); \coordinate (rd) at(-35:1);  \coordinate (d) at(-90:1.5);
 \draw (u)--(lu)--(ld)--(d)--(rd)--(ru)--(u);
 \draw [blue,thick](u)--(d);
 \draw [blue,thick](u)--(rd);
 \draw [blue,thick](u)--(ld);
 \draw [blue,thick](u)--(ld);
 \draw [blue,thick](d)--(ru);
 \draw [blue,thick](d)--(lu);
 \draw [blue,thick](d)--(ru);
 \draw [blue,thick](d)--(lu);
 \draw [blue,thick](lu)--(ru);
 \draw [blue,thick](ld)--(ru);
 \draw [blue,thick](lu)--(rd);
 \draw [blue,thick](ld)--(rd);
 \node at (90:1.8) {$v_1=v_3$};
 \node at (-90:1.8) {$v_2=v_4$};
 \node at (35:1.3) {$v_5$};
 \node at (145:1.3) {$v_8$};
 \node at (-35:1.3) {$v_6$};
 \node at (-145.5:1.3) {$v_7$};
 \fill [white](0,0) circle (3mm);
\node at (0,0){\textcolor{blue}{12}};
\end{tikzpicture}
\]
\end{figure}
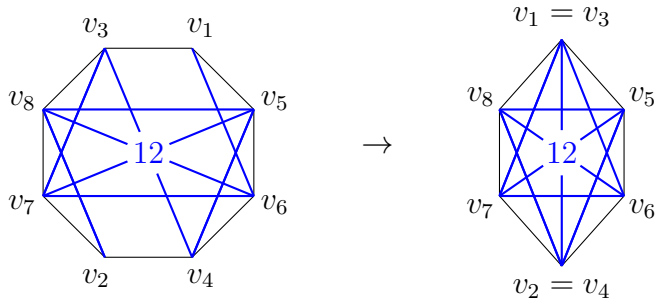
As in the proof of Lemma \ref{cor:boundary-Bn}, we can show that $\psi$ gives a bijection $\bar{\psi}\colon D/\Theta \to D'/\Theta,
\overline{(v_i,v_j)}\mapsto \overline{(v_i,v_j)}$ of orbits by $\Theta$ by using the induction of $n$. Furthermore, $\bar{\psi}$ preserves the compatibility of orbits; that is, if $\overline{(v_i,v_j)}$ intersects with $\overline{(v_k,v_\ell)}$ at the interior of $S$, then they do so at the interior of  $S'$, and vice versa.
Next, we show that $\lk_{\Delta(\xx,B)}(I')$ is a full subcomplex of $\Delta(\xx,B)$.
It suffices to show that, for any compatible set $F$ of $D/\Theta$, an orbit of $T'$ is compatible with all orbits of $F$.
We assume that there exists a compatible set $F$ of $D/\Theta$ such that each orbit of $T'$ is not compatible with at least one orbit of $F$. \begin{figure}
\caption{Orbits of $T'$, $D$, and $\bar{n}_1, \bar{n}_2$ in the case that $n=5$ \label{fig:contradiction-algorithm}}
\[
\begin{tikzpicture}
 \coordinate (1) at(18:1.5); \coordinate (2) at(54:1.5);\coordinate (3) at(90:1.5); \coordinate (4) at(126:1.5);
 \coordinate (5) at(162:1.5); \coordinate (10) at(-18:1.5); \coordinate (9) at(-54:1.5);  \coordinate (8) at(-90:1.5); \coordinate (7) at(-126:1.5); \coordinate (6) at(-162:1.5);
 \draw (1)--(2)--(3)--(4)--(5)--(6)--(7)--(8)--(9)--(10)--(1);  \draw[blue,thick] (1)--(3); \draw[blue,thick] (1)--(4); \draw[blue,thick] (1)--(5);
 \draw[blue,dashed] (10)--(6);  \draw[blue,dashed] (9)--(6);
 \draw[blue,dashed] (8)--(6);
 \node at (54:1.8) {$u$};
 \fill [white](0.5,1) circle (2mm);
\fill [white](-0.3,1) circle (3mm);
\fill [white](-0.8,0.5) circle (3mm);
\fill [white](0,0) circle (3mm);
\node at (0.5,1) {$\bar{\ell}'_1$};
\node at (-0.3,1) {$\bar{\ell}'_2$};
\node at (-0.8,0.5) {$\bar{\ell}'_3$};
\end{tikzpicture}
\hspace{20pt}
\begin{tikzpicture}
 \coordinate (1) at(18:1.5); \coordinate (2) at(54:1.5);\coordinate (3) at(90:1.5); \coordinate (4) at(126:1.5);
 \coordinate (5) at(162:1.5); \coordinate (10) at(-18:1.5); \coordinate (9) at(-54:1.5);  \coordinate (8) at(-90:1.5); \coordinate (7) at(-126:1.5); \coordinate (6) at(-162:1.5);
 \draw (1)--(2)--(3)--(4)--(5)--(6)--(7)--(8)--(9)--(10)--(1);
 \draw[blue,thick] (1)--(6);
 \draw[blue,thick] (1)--(8);
 \draw[blue,thick] (1)--(9);
 \draw[blue,thick] (2)--(4);
 \draw[blue,thick] (2)--(5);
 \draw[blue,thick] (3)--(5);
 \draw[blue,thick] (3)--(6);
 \draw[blue,thick] (3)--(8);
 \draw[blue,thick] (3)--(9);
 \draw[blue,thick] (3)--(10);
 \draw[blue,thick] (4)--(6);
 \draw[blue,thick] (4)--(8);
 \draw[blue,thick] (4)--(9);
 \draw[blue,thick] (4)--(10);
 \draw[blue,thick] (5)--(8);
 \draw[blue,thick] (5)--(9);
 \draw[blue,thick] (5)--(10);
 \draw[blue,thick] (7)--(9);
 \draw[blue,thick] (7)--(10);
 \draw[blue,thick] (8)--(10);
 \fill [white](0,0) circle (3mm);
\node at (0,0){\textcolor{blue}{12}};
\end{tikzpicture}
\hspace{20pt}
\begin{tikzpicture}
 \coordinate (1) at(18:1.5); \coordinate (2) at(54:1.5);\coordinate (3) at(90:1.5); \coordinate (4) at(126:1.5);
 \coordinate (5) at(162:1.5); \coordinate (10) at(-18:1.5); \coordinate (9) at(-54:1.5);  \coordinate (8) at(-90:1.5); \coordinate (7) at(-126:1.5); \coordinate (6) at(-162:1.5);
 \draw (1)--(2)--(3)--(4)--(5)--(6)--(7)--(8)--(9)--(10)--(1);
 \node at (54:1.8) {$u$};
 \draw[red,thick] (2)--(4);
 \draw[red,dashed] (7)--(9);
 \draw[blue,dashed] (1)--(3); \draw[blue,dashed] (1)--(4); \draw[blue,dashed] (1)--(5);
 \draw[blue,dashed] (10)--(6);  \draw[blue,dashed] (9)--(6);
 \draw[blue,dashed] (8)--(6);
 \fill [white](90:1.2) circle (2.2mm);
 \node at (90:1.2) {$\bar{n}_1$};
\end{tikzpicture}
\hspace{20pt}
\begin{tikzpicture}
 \coordinate (1) at(18:1.5); \coordinate (2) at(54:1.5);\coordinate (3) at(90:1.5); \coordinate (4) at(126:1.5);
 \coordinate (5) at(162:1.5); \coordinate (10) at(-18:1.5); \coordinate (9) at(-54:1.5);  \coordinate (8) at(-90:1.5); \coordinate (7) at(-126:1.5); \coordinate (6) at(-162:1.5);
 \draw (1)--(2)--(3)--(4)--(5)--(6)--(7)--(8)--(9)--(10)--(1);
 \node at (54:1.8) {$u$};
 \draw[red,thick] (2)--(5);
 \draw[red,dashed] (7)--(10);
 \draw[red,dashed] (2)--(4);
 \draw[red,dashed] (7)--(9);
  \draw[blue,dashed] (1)--(3); \draw[blue,dashed] (1)--(4); \draw[blue,dashed] (1)--(5);
 \draw[blue,dashed] (10)--(6);  \draw[blue,dashed] (9)--(6);
 \draw[blue,dashed] (8)--(6);
 \fill [white](-0.4,0.7) circle (2.2mm);
 \node at (-0.4,0.7) {$\bar{n}_2$};
\end{tikzpicture}
\]
\end{figure}
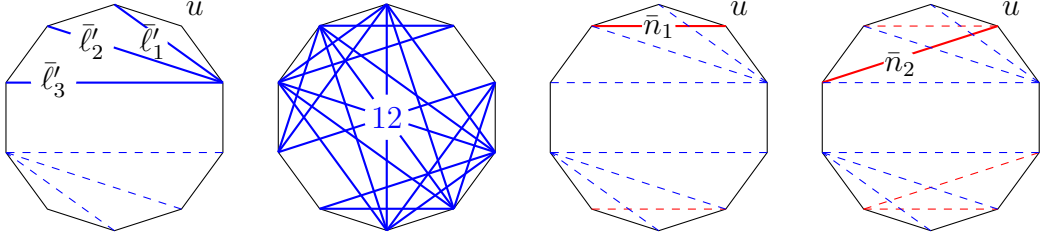
We can assume that labels of orbits $\{\bar{\ell}'_1,\dots, \bar{\ell}'_{n-2}\}$ are numbered in order from the outside without loss of generality, and denote by $u$ a vertex sharing only one triangle in a triangulation given by $\{\bar{\ell}'_1,\dots, \bar{\ell}'_n\}$ (see the first figure in Figure \ref{fig:contradiction-algorithm}). Since $F$ has an orbit which is not compatible with $\bar{\ell}_1'$, we can choose one and refer to this as $\bar{n}_1$ (see the third figure in Figure \ref{fig:contradiction-algorithm}). Then, one of endpoints of $\bar{n}_1$ is $u$. We assume $\bar{n}_1$ is not compatible with $\bar{\ell}'_1,\dots,\bar{\ell}'_{k_1}$. Then, we note that $k_1\neq n-2$ since there is no orbit in $D/\Theta$ which is not compatible with $\bar{\ell}'_{n-2}$ and one of the endpoints of which is $u$. If $k_1<n-3$. Then, we can choose $\bar{n}_2$ an orbit of $D/\Theta$, which is not compatible with $\bar{\ell}'_{k_1+1}$ and compatible with $\bar{n}_1$ (See the fourth figure in Figure \ref{fig:contradiction-algorithm}). We note that one of endpoints of $\bar{n}_2$ is $u$. We assume $\bar{n}_2$ is not compatible with $\bar{\ell}'_1,\dots,\bar{\ell}'_{k_2}$. If $k_2< n-3$, we choose $n_i$ repeatedly in the same as above. We can repeat this operation until $k_i=n-3$. Then, all orbits which are not compatible with $\bar{\ell}'_{n-2}$ are also not compatible with $\bar{n}_{i-1}$, which conflicts with the assumption. Therefore, $D/\Theta$ is a full subcomplex of $\Delta(\xx,B)$ and $\bar{\psi}$ can be lifted to an isomorphism between $\lk_{\Delta(\xx,B)}(I')$ and $\Delta(D_{n-1})$.
\end{proof}
It follows from the next lemma that $\st_{\Delta(\xx,B)}(I')$ does not coincide with $\Delta^+(\xx,B)$.
\begin{lemma}\label{lem:unique-variable-in-Dn}
For a cluster complex $\Delta(\xx,B)$ satisfying the assumption in Theorem \ref{thm:positive-simplex-D_n}, there exists a unique cluster variable $x^*$ such that it is neither in $\st_{\Delta(\xx,B)}(I')$ as a vertex nor in $I$.
\end{lemma}
\begin{proof}
The number of cluster variables of cluster algebras of $D_n$ type is $n^2$ since this number corresponds with that of almost positive roots by \cite{fzii}*{Theorem 1.9}. By lemmas \ref{lem:interior-D_n} and \ref{cor:boundary-closedstar}, the number of $0$-dimensional simplices of $\st_{\Delta(\xx,B)}(I')$ is $(n-2)+(n-1)^2=n^2-n-1$. Since the initial cluster variables are not in $\st_{\Delta(\xx,B)}(I')$, the number of cluster variables which is neither in $\st_{\Delta(\xx,B)}(I)$ as a $0$-dimensional simplex nor in $I$ is $n^2-(n^2-n-1)-n=1$. Indeed, the orbit of $\{(v_1,v_4),(v_2,v_3)\}$ is the unique orbit satisfying this condition.
\end{proof}
Clearly, we have $\st_{\Delta(\xx,B)}(I')\subset\Delta^+(\xx,B)$. Therefore, $\Delta^+(\xx,B)$ is a complex obtained by gluing a cone into $\st_{\Delta(\xx,B)}(I')$ by Lemma \ref{lem:unique-variable-in-Dn}, and this cone is described by $\join(\{x^*\},\lk_{\Delta(\xx,B)}(x^*)\cap\st_{\Delta(\xx,B)}(I'))$, where $x^*$ is the same notation in Lemma \ref{lem:unique-variable-in-Dn}.  We consider determining the shape of this cone and begin by characterizing the cone. We denote by $x'_{n-1}$ and $x'_{n}$ cluster variables corresponding to diameters $(v_3,v_4)$ with diameter labels 1 and 2  of $2n$-gon $S$ respectively.
\begin{lemma}\label{lem:characterization-cone}
We assume the same situation as in Lemma \ref{lem:unique-variable-in-Dn}. 
The complex 
\begin{align*}
    \join(\{x^*\},\lk_{\Delta(\xx,B)}(x^*)\cap\st_{\Delta(\xx,B)}(I'))
\end{align*}
    is a full subcomplex of $\Delta(\xx,B)$ the vertex set of which is
\begin{align}\label{eq:setV}
    V=\{y\in\Xcal(\xx,B)\mid y \notin I' \text{ and $y$ is compatible with both $x'_{n-1}$ and $x'_{n}$}\},
\end{align}
where $\Xcal(\xx,B)$ is the vertex set of $\Delta(\xx,B)$.
\end{lemma}
\begin{proof}
 We abbreviate $\lk_{\Delta(\xx,B)}(x^*)\cap\st_{\Delta(\xx,B)}(I')$ to $\lk\cap\st$. We denote by $\Delta_V$ a full subcomplex of $\Delta(\xx,B)$ in which the vertex set is $V$. 
 
 First, we prove $\join(\{x^*\},\lk\cap\st)\subset\Delta_V$. Since $\Delta_V$ is a full subcomplex, it suffices to show that $\Delta_V$ contains all $0$-dimensional simplices of $\join(\{x^*\},\lk\cap\st)$. We fix any vertex $\{y\}$ of $\join(\{x^*\},\lk\cap\st)$. If $y=x^{*}$, then $\Delta_V$ clearly contains $\{y\}$ since $x^*$ corresponds to the orbit of diagonals $\{(v_1,v_4),(v_2,v_3)\}$ of $S$ by the proof of Lemma \ref{lem:unique-variable-in-Dn}. We assume $y\neq x^{*}$ and show that $\{y\}\in\Delta_V$. By this assumption, $\{y\}$ is a $0$-dimensional simplex of $\lk\cap\st$. Since $\{x^*\}\notin \st_{\Delta(\xx,B)}(I')$ and $\{y\}\in\lk_{\Delta(\xx,B)}(x^*)$, we have $y\notin I'$. We assume that $y$ is not compatible with $x'_{n-1}$ or $x'_{n}$. Then, the corresponding orbit of diagonal $\bar{\ell}_y$ must cross $(v_3,v_4)$ at interior of $S$. Moreover, since $\{y\}\in\lk_{\Delta(\xx,B)}(x^*)$, $\bar{\ell_y}$ is compatible with $\{(v_1,v_4),(v_2,v_3)\}$. An orbit satisfying this condition has a diameter of only $(v_1,v_2)$. Therefore, $\bar{\ell}_y$ is a diameter $(v_1,v_2)$ in which the diameter label is 1 or 2. In either case, we have $y \in I$. This contradicts $\{y\}\in\st_{\Delta(\xx,B)}(I')$. 
 
 Next, we show converse inclusion. The complex $\join(\{x^*\},\lk\cap\st)$ is also a full subcomplex since $\lk_{\Delta(\xx,B)}(x^*)$ is full by Lemma \ref{lem:fullsub}, $\lk_{\Delta(\xx,B)}(I')$ is also full by Lemma \ref{cor:boundary-closedstar}, and we have
 \begin{align}\label{eq:lk-st}
\lk_{\Delta(\xx,B)}(x^*)\cap\st_{\Delta(\xx,B)}(I')=\lk_{\Delta(\xx,B)}(x^*)\cap\lk_{\Delta(\xx,B)}(I').
\end{align}
Therefore, it suffices to show that $\join(\{x^*\},\lk\cap\st)$ contains all $0$-dimensional simplices of $\Delta_V$. We fix any $0$-dimensional simplex $\{y\}\in\Delta_V$. If $y=x^*$, then $\{y\}$ is in $\join(\{x^*\},\lk\cap\st)$ clearly. We assume $y\neq x^*$, and show that $\{y\}\in\lk\cap\st$. First, cluster variables which are compatible with both $x'_{n-1}$ and $x'_n$, and not with $x^*$ are only $x_1',\dots,x'_{n-2}$. Since $y\notin I'$, we have $\{y\}\in\lk_{\Delta(\xx,B)}(x^*)$. We prove $\{y\}\in\st_{\Delta(\xx,B)}(I')$. We assume $y$ is not compatible with any elements in $I'$. Then, we have $y\in\{x_1,x_2,\dots,x_n,x^*\}$, but since $y\neq x^*$, we have $y\in\{x_1,x_2,\dots,x_n\}$. Any $x_i$ are not compatible with $x'_{n-1}$ or $x'_{n}$. This is a contradiction.
\end{proof}
The following lemma determines the shape of the cone.
\begin{lemma}\label{lem:description-join}
We assume the same situation as Lemma \ref{lem:unique-variable-in-Dn}. The complex $\lk_{\Delta(\xx,B)}(x^*)\cap\st_{\Delta(\xx,B)}(I')$ is isomorphic to $\join(\Delta_1,\Delta(A_{n-3}))$.
\end{lemma}
\begin{proof}
 Let $\bar{\ell}_{x^*}$ be an orbit corresponding to $x^*$. Let $E$ be the set of diagonals of $S$ corresponding to a vertex set of $\lk_{\Delta(\xx,B)}(x^*)\cap\st_{\Delta(\xx,B)}(I')$ (as $D$ in the proof of Lemma \ref{cor:boundary-closedstar}, we distinguish two labeled diameters with the same edge). By the characterization of Lemma \ref{lem:characterization-cone}, $E$ consists of two labeled diameters obtained by an edge $(v_3,v_4)$, all diagonals of the $(n-3)$-gon cut out by $(v_1,v_4)$, and all diagonals of the ($n-3)$-gon cut out by $(v_2,v_3)$ (see the subfigure shown at left in Figure \ref{fig:sujection-D-to-D'Dn-link}).
By \eqref{eq:lk-st}, $E$ is a subset of $D$ in Lemma \ref{cor:boundary-closedstar}, and $E/\Theta$ is a subset of $D/\Theta$. We consider a restriction $\psi|_{E}$ of $\psi$ to $E$, where $\psi$ is the same function as that in Lemma \ref{cor:boundary-closedstar}. The map $\psi|_{E}$ induces an injection $\bar{\psi}|_{E}\colon E/\Theta \to D'/\Theta$, and by the fullness of $\lk_{\Delta(\xx,B)}(x^*)\cap\st_{\Delta(\xx,B)}(I')$, it is lifted to isomorphism between $\lk_{\Delta(\xx,B)}(x)\cap\st_{\Delta(\xx,B)}(I')$ and a full subcomplex of $\Delta(D_{n-1})$ (see Figure \ref{fig:sujection-D-to-D'Dn-link}).
\begin{figure}[ht]
\caption{Restriction $\psi$ to $E$ ($n=4$) }\label{fig:sujection-D-to-D'Dn-link}
\[
\begin{tikzpicture}[baseline=0mm]
 \coordinate (u1) at(67.5:1.5); \coordinate (u2) at(112.5:1.5); \coordinate (lu) at(157.5:1.5);
 \coordinate (ld) at(-157.5:1.5); \coordinate (ru) at(22.5:1.5); \coordinate (rd) at(-22.5:1.5);  \coordinate (d1) at(-67.5:1.5); \coordinate (d2) at(-112.5:1.5);
 \draw (u1)--(u2)--(lu)--(ld)--(d2)--(d1)--(rd)--(ru)--(u1);
 \draw [blue,thick](u1)--(rd);
 \draw [blue,thick](u2)--(ld);
 \draw [blue,thick](u2)--(ld);
 \draw [blue,thick](d1)--(ru);
 \draw [blue,thick](d2)--(lu);
 \draw [blue,thick](d1)--(ru);
 \draw [blue,thick](d2)--(lu);
 \draw [blue,thick](u2)--(d1);
 \node at (67.5:1.8) {$v_1$};
 \node at (-112.5:1.8) {$v_2$};
 \node at (112.5:1.8) {$v_3$};
 \node at (-67.5:1.8) {$v_4$};
 \node at (22.5:1.8) {$v_5$};
 \node at (157.5:1.8) {$v_8$};
 \node at (-157.5:1.8) {$v_7$};
 \node at (-22.5:1.8) {$v_6$};
 \fill [white](0,0) circle (2mm);
\node at (0,0){\textcolor{blue}{12}};
\end{tikzpicture}
\hspace{20pt}\to \hspace{20pt}
\begin{tikzpicture}[baseline=0mm]
 \coordinate (u) at(90:1.5);  \coordinate (lu) at(145:1);
 \coordinate (ld) at(-145:1); \coordinate (ru) at(35:1); \coordinate (rd) at(-35:1);  \coordinate (d) at(-90:1.5);
 \draw (u)--(lu)--(ld)--(d)--(rd)--(ru)--(u);
 \draw [blue,thick](u)--(d);
 \draw [blue,thick](u)--(rd);
 \draw [blue,thick](u)--(ld);
 \draw [blue,thick](u)--(ld);
 \draw [blue,thick](d)--(ru);
 \draw [blue,thick](d)--(lu);
 \draw [blue,thick](d)--(ru);
 \draw [blue,thick](d)--(lu);
 \draw [blue,dashed](lu)--(ru);
 \draw [blue,dashed](ld)--(rd);
 \draw [blue,dashed](ld)--(ru);
 \draw [blue,dashed](rd)--(lu);
 \node at (90:1.8) {$v_1=v_3$};
 \node at (-90:1.8) {$v_2=v_4$};
 \node at (35:1.3) {$v_5$};
 \node at (145:1.3) {$v_8$};
 \node at (-35:1.3) {$v_6$};
 \node at (-145.5:1.3) {$v_7$};
 \fill [white](0,0) circle (2mm);
\node at (0,0){\textcolor{blue}{12}};
\end{tikzpicture}
\]
\end{figure}
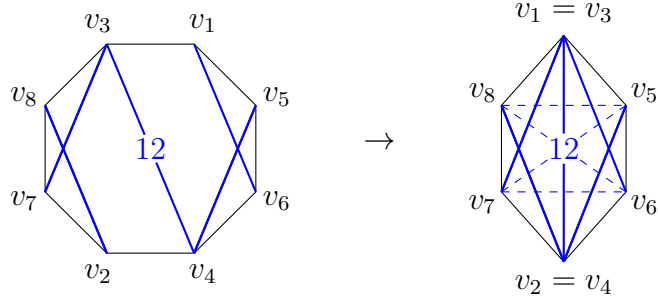
Furthermore, this subcomplex is isomorphic to $\join(\Delta_1,\Delta(A_{n-3}))$. We will confirm this fact. We denote by $E'=\psi(E)$ the image of $\psi|_E$. Let $S''$ be an $n$-gon obtained by cutting out $S'$ at a diameter $(v_1,v_2)$. By symmetry, we can assume $S''$ has $v_{5}$ as a vertex without loss of generality. Let $E'_\circ$ be a subset of $E'$ consists of all diagonals of $S$ except for two labeled diameters obtained by an edge $(v_1, v_2)$. Let $E''$ be all diagonals of $S''$. We consider a map $\eta:E''\to E'_\circ/\Theta, (v_i,v_j)\mapsto\overline{(v_i,v_j)}$, which is a bijection by the construction of $E$, which is lifted to the isomorphism of complexes. Moreover, two labeled diameters obtained by an edge $(v_1,v_2)$ are compatible with all orbits of $E'_{\circ}/\Theta$.
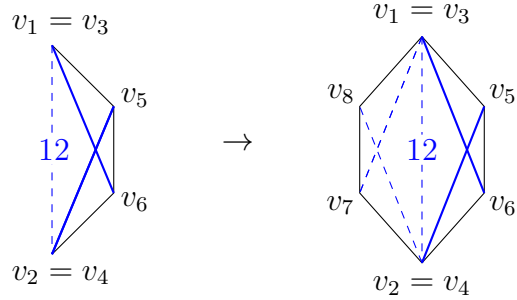
\begin{figure}[ht]
\caption{Map $\eta$ ($n=4$) }\label{fig:sujection-D-to-D'Dn-link2}
\[
\begin{tikzpicture}[baseline=0mm]
 \coordinate (u1) at(67.5:1.5);   \coordinate (ru) at(22.5:1.5); \coordinate (rd) at(-22.5:1.5);  \coordinate (d1) at(-67.5:1.5);
 \draw (d1)--(rd)--(ru)--(u1);
 \draw [blue,thick](u1)--(rd);
 \draw [blue,thick](d1)--(ru);
 \draw [blue,thick](d1)--(ru);
 \draw [blue,dashed](u1)--(d1);
 \node at (67.5:1.8) {$v_1=v_3$};
 \node at (-67.5:1.8) {$v_2=v_4$};
 \node at (22.5:1.8) {$v_5$};
 \node at (-22.5:1.8) {$v_6$};
 \fill [white](0.6,0) circle (2mm);
\node at (0.6,0){\textcolor{blue}{12}};
\end{tikzpicture}
\hspace{20pt}\to \hspace{20pt}
\begin{tikzpicture}[baseline=0mm]
 \coordinate (u) at(90:1.5);  \coordinate (lu) at(145:1);
 \coordinate (ld) at(-145:1); \coordinate (ru) at(35:1); \coordinate (rd) at(-35:1);  \coordinate (d) at(-90:1.5);
 \draw (u)--(lu)--(ld)--(d)--(rd)--(ru)--(u);
 \draw [blue,dashed](u)--(d);
 \draw [blue,thick](u)--(rd);
 \draw [blue,dashed](u)--(ld);
 \draw [blue,dashed](u)--(ld);
 \draw [blue,thick](d)--(ru);
 \draw [blue,dashed](d)--(lu);
 \node at (90:1.8) {$v_1=v_3$};
 \node at (-90:1.8) {$v_2=v_4$};
 \node at (35:1.3) {$v_5$};
 \node at (145:1.3) {$v_8$};
 \node at (-35:1.3) {$v_6$};
 \node at (-145.5:1.3) {$v_7$};
 \fill [white](0,0) circle (2mm);
\node at (0,0){\textcolor{blue}{12}};
\end{tikzpicture}
\]
\end{figure}
Since two labeled diameters are compatible with each other, the full subcomplex the vertex set of which is $E'$ is isomorphic to $\join(\Delta_1,\Delta(A_{n-3}))$.
\end{proof}
\begin{proof}[Proof of Theorem \ref{thm:positive-simplex-D_n}]
 First, we prove the equation
 \[
 \Delta^+(\xx,B)=\join(\{x^*\},\lk_{\Delta(\xx,B)}(x^*)\cap\st_{\Delta(\xx,B)}(I'))\cup \st_{\Delta(\xx,B)}(I'),
 \]
We abbreviate $\join(\{x^*\},\lk_{\Delta(\xx,B)}(x^*)\cap\st_{\Delta(\xx,B)}(I'))$ to $\join(\{x^*\},\lk\cap\st)$ again. The right-hand side is included in the left. We prove the converse inclusion. We denote by $J$ a vertex set of $\st_{\Delta(\xx,B)}(I')$. Let $F$ be a simplex of $\Delta^+(\xx,B)$. Since $F$ consists of cluster variables which are not initial variables, $F$ is a subset of $\{x^*\}\cup J$ by Lemma \ref{lem:unique-variable-in-Dn}. Since $\lk_{\Delta(\xx,B)}(I')$ is a full subcomplex of $\Delta(\xx,B)$, $\st_{\Delta(\xx,B)}(I')$ is also a full subcomplex of $\Delta(\xx,B)$. Therefore, if $F$ does not contain $x^*$, then $F$ is a simplex of $\st_{\Delta(\xx,B)}(I')$. If $F$ contains $x^*$, then each $0$-dimensional simplices of $F$ are simplices of $\join(\{x^*\},\lk\cap\st)$. By the fullness of $\join(\{x^*\},\lk\cap\st)$, $F$ belongs to $\join(\{x^*\},\lk\cap\st)$.
 To complete the proof of Theorem \ref{thm:positive-simplex-D_n}, we verify the following three.
 \begin{itemize}
     \item[(1)] $\lk\cap\st\cong\join(\Delta_1,\Delta(A_{n-3}))$.
     \item[(2)] $\join(\{x^*\},\lk\cap\st)\cong\join(\Delta_0,\Delta_1,\Delta(A_{n-3}))$,
     \item[(3)] there exists an isomorphism  $\st_{\Delta(\xx,B)}(I')\cong\st_{\Delta(\xx,B)}(I\setminus\{x_{n-1},x_n\})$ such that it induces $\lk\cap\st\cong\Delta_{\mathrm{glue}}$.
 \end{itemize}
  The statement (1) follows from Lemma \ref{lem:description-join} and (2) follows from (1).
  We show (3). The existence of an isomorphism follows from Lemma \ref{lem:description-positiveconeDn}. We show that it induces $\lk\cap\st\cong\Delta_{\mathrm{glue}}$. We can assume $\sigma$ in the proof of Lemma \ref{lem:description-positiveconeDn} is the identity without loss of generality. Owing to $\lk\cap\st=\join(\{x^*\},\lk\cap\st)\cap\st_{\Delta(\xx,B)}(I')$ and Lemma \ref{lem:characterization-cone}, $\lk\cap\st$ is a full subcomplex the vertex set of which is $V\cap\st_{\Delta(\xx,B)}(I')$, where $V$ is that of \eqref{eq:setV}. Since $x'_i$ corresponds to $x_i$ by the isomorphism in the proof of Lemma \ref{lem:description-positiveconeDn}, $V\cap\st_{\Delta(\xx,B)}(I')$ corresponds to $V'\cap\st_{\Delta(\xx,B)}(I\setminus\{x_{n-1},x_n\})$, where
  \begin{align*}
    V'=\{y\in\Xcal(\xx,B)\mid y \notin I \text{ and $y$ is compatible with both $x_{n-1}$ and $x_{n}$}\}.
\end{align*}
Since $\Delta_{\mathrm{glue}}$ is a full subcomplex of $\Delta(\xx,B)$ in which the vertex set is $V'\cap\st_{\Delta(\xx,B)}(I\setminus\{x_{n-1},x_n\})$, this completes the proof.
\end{proof}
Similar to $A_n,B_n$ and $C_n$ type, we give a conceptual proof of the following corollary given by Krattenthaler.
\begin{corollary}[\cite{krat}*{Theorem FD}]\label{cor:positive-face-from-all-face-D_nD_n-1}
Let $B$ be an exchange matrix whose quiver $Q_B$ has a diagram of $D_n$ type in Figure \ref{fig:valued-quiver} as its underlying graph. Let $f_{{D_n},k}$ be the number of $k$-dimensional simplices of $\Delta(D_{n})$, $f_{{A_n},k}$ be the number of $k$-dimensional simplices of $\Delta(A_{n})$ and $f^+_{n,k}$ the number of $k$-dimensional simplices of $\Delta^+(\xx,B)$. If $k= -1$, then we have $f^+_{n,-1}=1$. If $k\neq -1$, then we obtain
\begin{align}
    f^+_{n,k}&=
    \dfrac{1}{2}(f_{{D_n},k}+f_{D_{n-1},k}+f_{A_{n-3},k-1}+f_{A_{n-3},k-2})-f_{A_{n-2},k-1}-f_{A_{n-2},k-1}\label{eq:positive-facevector-D_n}\\\nonumber
    &=\begin{pmatrix}n\\ k+1\end{pmatrix}\begin{pmatrix}n+k-1\\ k+1\end{pmatrix}+\begin{pmatrix}n-1\\ k\end{pmatrix}\begin{pmatrix}n-k-2\\ k\end{pmatrix}-\dfrac{1}{n-1}\begin{pmatrix}n-1\\ k\end{pmatrix}\begin{pmatrix}n+k-1\\ k+1\end{pmatrix}.
\end{align}
\end{corollary}
\begin{proof}
When $k=-1$, the claim is clear. We consider the case of $k\neq -1$. By Corollary \ref{cor:independent-orientation}, it suffices to consider the linearly oriented case. First, we show the following equation:
\begin{align}\label{facevectorDn}
f(\Delta^+(\xx,B))&=\dfrac{1}{2}f(\Delta(D_n))+\dfrac{1}{2}f(\Delta(D_{n-1}))-[f(\Delta(A_{n-2}))]_1-[f(\Delta(A_{n-2}))]_2\\
&+\dfrac{1}{2}[f(\Delta(A_{n-3}))]_1+\dfrac{1}{2}[f(\Delta(A_{n-3}))]_2.\nonumber
\end{align}
We note that
\begin{align*}
&\st_{\Delta(\xx,B)}(I-\{x_{n-1},x_n\})\\
&=\widehat{\Delta_\xx}\setminus(\underline{\Delta_{\{x_{n-1}\}\subset \xx}}\cup\underline{\Delta_{\{x_n\}\subset \xx}}\cup\underline{\Delta_{\{x_{n-1},x_n\}\subset \xx}})\cup \lk_{\Delta(\xx,B)}(I-\{x_{n-1},x_n\}),
\end{align*}
and by Lemmas \ref{lem:description-positiveconeDn} and \ref{cor:boundary-closedstar}, we have
$\lk_{\Delta(\xx,B)}(I-\{x_{n-1},x_n\})\cong \lk_{\Delta(\xx,B)}(I')\cong \Delta(D_{n-1})$.
By Theorem \ref{thm:positive-simplex-D_n} and the above argument, we obtain
\begin{align*}
    f(\Delta(D_n))&=f(\widehat{\Delta_\xx})+f(\Delta^+(\xx,B))\\
    &=f(\widehat{\Delta_\xx})+f(\st_{\Delta(\xx,B)}(I-\{x_{n-1},x_n\}))+[f(\join(\Delta_1,\Delta(A_3)))]_1\\
     &=2f(\widehat{\Delta_\xx})-f(\underline{\Delta_{\{x_{n-1}\}\subset \xx}})-f(\underline{\Delta_{\{x_n\}\subset \xx}})-f(\underline{\Delta_{\{x_{n-1},x_n\}\subset \xx}})\\
     &+f(\Delta(D_{n-1}))+[f(\join(\Delta_1,\Delta(A_3)))]_1,
\end{align*}
where the notation is the same in the proof of Theorem \ref{thm:mutation-positivecomplex}.
Then, we have
\begin{align*}
    f(\underline{\Delta_{\{x_{n-1}\}\subset \xx}})&=f(\underline{\Delta_{\{x_{n}\}\subset \xx}})=[f(\Delta^+(\xx-\{x_n\},B_{\backslash\{x_n\}}))]_1,\\
     f(\underline{\Delta_{\{x_{n-1},x_{n}\}\subset \xx}})&=[f(\Delta^+(\xx-\{x_{n-1},x_n\},B_{\backslash\{x_{n-1},x_n\}}))]_2
\end{align*}
by using the discussion in the proof of Theorem \ref{thm:mutation-positivecomplex} and Lemma \ref{cor:generalized-bijection-between-B-and-Bx}. Furthermore, since $B_{\backslash\{x_{n-1}\}}$, $B_{\backslash\{x_{n}\}}$, $B_{\backslash\{x_{n-1},x_n\}}$ is of tree $A_n$ type, we have
\begin{align*}
   [f(\Delta^+(\xx-\{x_n\},B_{\backslash\{x_n\}}))]_1&=[f(\cone(\Delta({A_{n-2})})]_1\\&=[f(\Delta({A_{n-2})})]_1+[f(\Delta({A_{n-2})})]_2,\\
    [f(\Delta^+(\xx-\{x_{n-1},x_n\},B_{\backslash\{x_{n-1},x_n\}}))]_2&=[f(\cone(\Delta({A_{n-3})}))]_2\\&=[f(\Delta({A_{n-3})})]_2+[f(\Delta({A_{n-3})})]_3
\end{align*}
by Corollary \ref{cor:positive-face-from-all-face-n-1}.
Finally, we have
\begin{align*}
[f(\join(\Delta_1,\Delta(A_{n-3})))]_1=[f(\Delta(A_{n-3}))]_1+2[f(\Delta(A_{n-3}))]_2+[f(\Delta(A_{n-3}))]_3
\end{align*}
by \eqref{eq:join-facevector}.
By substituting and simplifying the above equations, we have \eqref{facevectorDn}.
The combinatorial formula of $f_{D_n,k}$ has already given by \cite{oeis}*{A080721}, and it is
\begin{align}\label{eq:formula-faces-D_n}
    f_{D_n,k}=\begin{pmatrix}n\\ k+1\end{pmatrix}\left(\begin{pmatrix}n+k+1\\ k+1\end{pmatrix}-\begin{pmatrix}n+k-1\\ k\end{pmatrix}\right).
\end{align}
By \eqref{facevectorDn} and \eqref{eq:formula-faces-D_n}, if $k\neq0$, we have
\begin{align}
    f^+_{D_n,k}&=
    \dfrac{1}{2}(f_{{D_n},k}+f_{D_{n-1},k}+f_{A_{n-3},k-1}+f_{A_{n-3},k-2})-f_{A_{n-2},k-1}-f_{A_{n-2},k-2}\\
    &=\dfrac{1}{2}\begin{pmatrix}n\\ k+1\end{pmatrix}\left(\begin{pmatrix}n+k+1\\ k+1\end{pmatrix}-\begin{pmatrix}n+k-1\\ k\end{pmatrix}\right)+\dfrac{1}{2}\begin{pmatrix}n-1\\ k+1\end{pmatrix}\left(\begin{pmatrix}n+k\\ k+1\end{pmatrix}-\begin{pmatrix}n+k-2\\ k\end{pmatrix}\right)\nonumber\\
    &-\dfrac{1}{k+1}\begin{pmatrix}n-2\\ k\end{pmatrix}\begin{pmatrix}n+k\\ k\end{pmatrix}-\dfrac{1}{k}\begin{pmatrix}n-2\\ k-1\end{pmatrix}\begin{pmatrix}n+k-1\\ k-1\end{pmatrix}+\dfrac{1}{2k+2}\begin{pmatrix}n-3\\ k\end{pmatrix}\begin{pmatrix}n+k-1\\ k\end{pmatrix}\nonumber\\
    &+\dfrac{1}{2k}\begin{pmatrix}n-3\\ k-1\end{pmatrix}\begin{pmatrix}n+k-2\\ k-1\end{pmatrix}\nonumber.
\end{align}
If $k=0$, we have $f^+_{D_n,0}=n^2-n$.
By organizing \eqref{eq:formula-faces-D_n}, we obtain the desired equation \eqref{eq:positive-facevector-D_n}. This equation is also correct in the case that $k=0$. 
\end{proof}
\begin{remark}
The formula is given more generally in \cite{krat}*{Theorem FD}. The equation \eqref{eq:positive-facevector-D_n} corresponds with \cite{krat}*{Theorem FD} in the case that $m = 1,\ell = 0$. 
\end{remark}
\begin{example}
Let $(\xx,B)$ be a seed, which is the same as in Example \ref{ex:correspondence-arc-variable-D3}. Then, $\Delta^+(\xx,B)$ is a domain which is filled with gray in Figure \ref{positivecomplexD_n}.
\begin{figure}[ht]
\caption{positive cluster complex of $D_3$ type \label{positivecomplexD_n}}
\vspace{5pt}
\begin{center}
\scalebox{0.8}{
\begin{tikzpicture}
\coordinate (0) at (0,0);
\coordinate (u*) at (90:5.33);
\coordinate (u) at (90:4);
\coordinate (ul) at (150:4);
\coordinate (ur) at (30:4);
\coordinate (uml) at (150:2);
\coordinate (umr) at (30:2);
\coordinate (dmc) at (-90:2);
\coordinate (dl) at (-150:4);
 \coordinate (dl*) at (-150:5.33);
\coordinate (dr) at (-30:4);
\coordinate (dr*) at (-30:5.33);
\coordinate (d) at (-90:4);
\draw (u) to (ul);
\draw (u) to (ur);
\draw (ul) to (uml);
\draw (ur) to (umr);
\draw (umr) to (uml);
\draw (u) to (uml);
\draw (u) to (umr);
\draw (uml) to (dl);
\draw (umr) to (dr);
\draw (ul) to (dl);
\draw (ur) to (dr);
\draw (uml) to (dmc);
\draw (umr) to (dmc);
\draw (dl) to (dmc);
\draw (dr) to (dmc);
\draw (dl) to (d);
\draw (dmc) to (d);
\draw (dr) to (d);
\draw(ul) [out=90,in=180]to (u*);
\draw(ur) [out=90,in=0]to (u*);
\draw(ul) [out=210,in=120]to (dl*);
\draw(d) [out=210,in=300]to (dl*);
\draw(d) [out=330,in=240]to (dr*);
\draw(ur) [out=330,in=60]to (dr*);
\fill [white](u) circle (0.9cm);
\fill[white] (ul) circle (0.9cm);
\fill[white] (ur) circle (0.9cm);
\fill[white] (uml) circle (0.9cm);
\fill[white] (umr) circle (0.9cm);
\fill[white] (dmc) circle (0.9cm);
\fill[white] (dl) circle (0.9cm);
\fill[white] (dr) circle (0.9cm);
\fill[white] (d) circle (0.9cm);
\draw (u)++(210:0.7cm) to ++(90:0.7cm);
\draw (u)++(150:0.7cm) to ++(30:0.7cm);
\draw (u)++(90:0.7cm) to ++(330:0.7cm);
\draw (u)++(30:0.7cm) to ++(270:0.7cm);
\draw (u)++(330:0.7cm) to ++(210:0.7cm);
\draw (u)++(270:0.7cm) to ++(150:0.7cm);
\draw (ur)++(210:0.7cm) to ++(90:0.7cm);
\draw (ur)++(150:0.7cm) to ++(30:0.7cm);
\draw (ur)++(90:0.7cm) to ++(330:0.7cm);
\draw (ur)++(30:0.7cm) to ++(270:0.7cm);
\draw (ur)++(330:0.7cm) to ++(210:0.7cm);
\draw (ur)++(270:0.7cm) to ++(150:0.7cm);
\draw (ul)++(210:0.7cm) to ++(90:0.7cm);
\draw (ul)++(150:0.7cm) to ++(30:0.7cm);
\draw (ul)++(90:0.7cm) to ++(330:0.7cm);
\draw (ul)++(30:0.7cm) to ++(270:0.7cm);
\draw (ul)++(330:0.7cm) to ++(210:0.7cm);
\draw (ul)++(270:0.7cm) to ++(150:0.7cm);
\draw (umr)++(210:0.7cm) to ++(90:0.7cm);
\draw (umr)++(150:0.7cm) to ++(30:0.7cm);
\draw (umr)++(90:0.7cm) to ++(330:0.7cm);
\draw (umr)++(30:0.7cm) to ++(270:0.7cm);
\draw (umr)++(330:0.7cm) to ++(210:0.7cm);
\draw (umr)++(270:0.7cm) to ++(150:0.7cm);
\draw (uml)++(210:0.7cm) to ++(90:0.7cm);
\draw (uml)++(150:0.7cm) to ++(30:0.7cm);
\draw (uml)++(90:0.7cm) to ++(330:0.7cm);
\draw (uml)++(30:0.7cm) to ++(270:0.7cm);
\draw (uml)++(330:0.7cm) to ++(210:0.7cm);
\draw (uml)++(270:0.7cm) to ++(150:0.7cm);
\draw (dl)++(210:0.7cm) to ++(90:0.7cm);
\draw (dl)++(150:0.7cm) to ++(30:0.7cm);
\draw (dl)++(90:0.7cm) to ++(330:0.7cm);
\draw (dl)++(30:0.7cm) to ++(270:0.7cm);
\draw (dl)++(330:0.7cm) to ++(210:0.7cm);
\draw (dl)++(270:0.7cm) to ++(150:0.7cm);
\draw (dr)++(210:0.7cm) to ++(90:0.7cm);
\draw (dr)++(150:0.7cm) to ++(30:0.7cm);
\draw (dr)++(90:0.7cm) to ++(330:0.7cm);
\draw (dr)++(30:0.7cm) to ++(270:0.7cm);
\draw (dr)++(330:0.7cm) to ++(210:0.7cm);
\draw (dr)++(270:0.7cm) to ++(150:0.7cm);
\draw (dmc)++(210:0.7cm) to ++(90:0.7cm);
\draw (dmc)++(150:0.7cm) to ++(30:0.7cm);
\draw (dmc)++(90:0.7cm) to ++(330:0.7cm);
\draw (dmc)++(30:0.7cm) to ++(270:0.7cm);
\draw (dmc)++(330:0.7cm) to ++(210:0.7cm);
\draw (dmc)++(270:0.7cm) to ++(150:0.7cm);
\draw (d)++(210:0.7cm) to ++(90:0.7cm);
\draw (d)++(150:0.7cm) to ++(30:0.7cm);
\draw (d)++(90:0.7cm) to ++(330:0.7cm);
\draw (d)++(30:0.7cm) to ++(270:0.7cm);
\draw (d)++(330:0.7cm) to ++(210:0.7cm);
\draw (d)++(270:0.7cm) to ++(150:0.7cm);
\draw [thick, blue](uml)++(150:0.7cm) to ++(-30:1.39cm);
\draw [thick, blue](umr)++(30:0.7cm) to ++(210:1.39cm);
\draw [thick, blue](dmc)++(90:0.7cm) to ++(270:1.39cm);
\draw [thick, blue](dl)++(90:0.7cm) to ++(-60:1.21cm);
\draw [thick, blue](dl)++(150:0.7cm) to ++(-60:1.21cm);
\draw [thick, blue](dr)++(90:0.7cm) to ++(240:1.21cm);
\draw [thick, blue](dr)++(30:0.7cm) to ++(240:1.21cm);
\draw [thick, blue](u)++(150:0.7cm) to ++(0:1.21cm);
\draw [thick, blue](u)++(210:0.7cm) to ++(0:1.21cm);
\draw [thick, blue](ul)++(150:0.7cm) to ++(-30:1.39cm);
\draw [thick, blue](d)++(90:0.7cm) to ++(270:1.39cm);
\draw [thick, blue](ur)++(30:0.7cm) to ++(210:1.39cm);
\fill [white](umr) circle (2mm);
\node at (umr){\textcolor{blue}{1}};
\fill [white](uml) circle (2mm);
\node at (uml){\textcolor{blue}{1}};
\fill [white](dmc) circle (2mm);
\node at (dmc){\textcolor{blue}{1}};
\fill [white](ur) circle (2mm);
\node at (ur){\textcolor{blue}{2}};
\fill [white](ul) circle (2mm);
\node at (ul){\textcolor{blue}{2}};
\fill [white](d) circle (2mm);
\node at (d){\textcolor{blue}{2}};
\fill [gray,opacity=.4] (ul)[out=90,in=180] to (u*)[out=0,in=90] to (ur)--(umr)--(uml)--(dl)--(ul);
\end{tikzpicture}}
\end{center}
\end{figure}
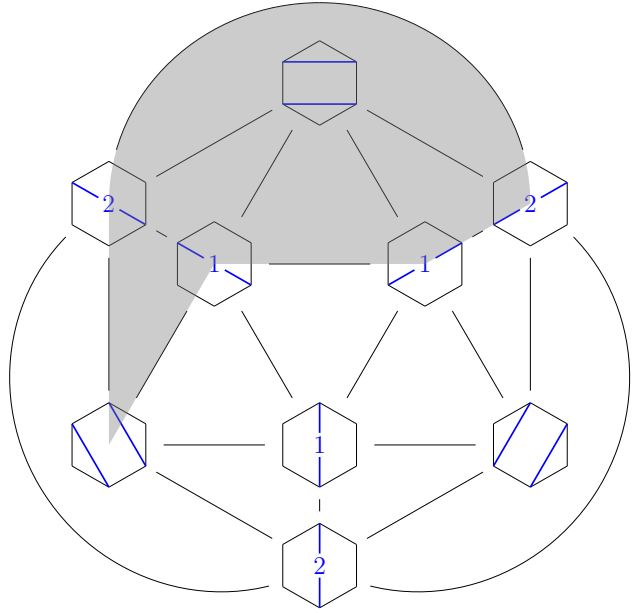
\end{example}
\section{Algorithm to calculate face vectors of positive cluster complexes}
In this section, we provide an algorithm to calculate face vectors of positive cluster complexes along with a concrete example. First, we give face vectors of positive cluster complexes of tree Dynkin type under a specific condition. We have already obtained the formula of face vectors for the classical type in Section 4.
\begin{theorem}\label{thm:count1}
Let $B$ be an exchange matrix of tree $A_n,B_n,C_n,\text{ or }D_n$ type. Then, $f(\Delta^+(\xx,B))$ is given by the following list.
\begin{itemize}
    \item[($A_n$)] $\dfrac{1}{k+2}\begin{pmatrix}n\\ k+1\end{pmatrix}\begin{pmatrix}n+k+1\\ k+1\end{pmatrix}$,
    \vspace{3mm}
    \item[($B_n$)] $ \begin{pmatrix}n\\ k+1\end{pmatrix}\begin{pmatrix}n+k\\ k+1\end{pmatrix}$,
      \vspace{3mm}
    \item[($C_n$)] same as $B_n$ type,
      \vspace{3mm}
    \item[($D_n$)] $\begin{pmatrix}n\\ k+1\end{pmatrix}\begin{pmatrix}n+k-1\\ k+1\end{pmatrix}+\begin{pmatrix}n-1\\ k\end{pmatrix}\begin{pmatrix}n-k-2\\ k\end{pmatrix}-\dfrac{1}{n-1}\begin{pmatrix}n-1\\ k\end{pmatrix}\begin{pmatrix}n+k-1\\ k+1\end{pmatrix}$.
  \end{itemize}
\end{theorem}
\begin{proof}
It follows from Corollaries \ref{cor:positive-face-from-all-face-n-1}, \ref{cor:positive-face-from-all-face-B_nB_n-1}, and \ref{cor:positive-face-from-all-face-D_nD_n-1}. 
\end{proof}
For exceptional type, we calculate face vectors using \texttt{SageMath} and the software package \texttt{ClusterSeed}.
\begin{theorem}\label{thm:count2}
Let $B$ be an exchange matrix of tree $E_6,E_7,E_8,F_4,\text{ or }G_2$ type. Then, $f(\Delta^+(\xx,B))$ is given by the following list.
  \begin{itemize}
    \item[($E_6$)] $(1,36,300,1035,1720,1368,418)$,
    \item[($E_7$)] $(1,63,777,3927,9933,13299,9009,2431)$,
    \item[($E_8$)] $(1,120,2135,15120,54327,108360,121555,71760,17342)$,
    \item[($F_4$)] $(1, 24, 101, 144, 66)$,
    \item[($G_2$)] $(1, 6, 5)$.
  \end{itemize}
\end{theorem}
Combining Theorem \ref{thm:count1} with Theorem \ref{thm:count2}, we have the following corollary.
\begin{corollary}\label{cor:invariant-orientation-Sec5}
Let $B$ be an exchange matrix of tree Dynkin type. Then, $f(\Delta^+(\xx,B))$ is given by the list in Theorem \ref{thm:count1} and \ref{thm:count2}.
\end{corollary}
By using these facts and Theorem \ref{thm:mutation-positivecomplex}, we calculate Example \ref{ex:not-equal-simplex} again.

Let $(\xx,B)$ and $(\xx'',B'')$ be the same as in Example \ref{ex:not-equal-simplex}. The vector $f(\Delta^+(\xx,B))-f(\Delta^+(\xx'',B''))$ is given by the following steps:
\begin{itemize}
    \item [(1)] Since $(\xx'',B'')=\mu_1(\xx,B)$, we consider full subquivers of $Q_B$ and $Q_{B''}$ excluded vertices labeled by 1. We denote these quivers as $\hat{Q}_B$ and $\hat{Q}_{B''}$, and they are as follows.
    \begin{align*}
        \hat{Q}_B=\begin{xy} (10,0)*+{2}="J", (20,0)*+{{3}}="K" \end{xy},\quad
        \hat{Q}_{B''}=\begin{xy}(10,0)*+{2}="J", (20,0)*+{{3}}="K" \ar@{->}"K";"J" \end{xy}.
    \end{align*}
    We note that they are products of quivers that are mutation equivalent to Dynkin quivers.
    \item[(2)] Let $\hat{B}$ and $\hat{B''}$ be the corresponding exchange matrices of $\hat{Q}_B$ and $\hat{Q}_{B''}$. We calculate $f(\Delta^+((x_2,x_3),\hat{B}))$ and $f(\Delta^+((x_2,x_3),\hat{B''}))$. Since $\hat{Q}_B$ is Dynkin $A_1\times A_1$ diagram and $\hat{Q}_{B''}$ is Dynkin $A_2$ diagram, we have
    \begin{align*}
        f(\Delta^+((x_2,x_3),\hat{B}))&=(1,2,1)\\
        f(\Delta^+((x_2,x_3),\hat{B''}))&=(1,3,2).
    \end{align*}
    by \eqref{eq:positive-join}, \eqref{eq:join-facevector} and Theorem \ref{thm:count1}.
    \item[(3)] We calculate the shift to the right by 1 of the difference of $f(\Delta^+((x_2,x_3),\hat{B}))$ and $f(\Delta^+((x_2,x_3),\hat{B''})$; that is,
        \begin{align*}
           [
        f(\Delta^+((x_2,x_3),\hat{B''}))]_1- [f(\Delta^+((x_2,x_3),\hat{B}))]_1 =(0,0,1,1).
        \end{align*} By Theorem \ref{thm:mutation-positivecomplex}, it corresponds with $f(\Delta^+(\xx,B))-f(\Delta^+(\xx'',B''))$.
\end{itemize}
Furthermore, we can calculate $f(\Delta^+(B,\xx))$ using Theorem \ref{thm:count1} (and Corollary \ref{cor:invariant-orientation-Sec5}), and the value is $(1,6,10,5)$. Therefore, we have
\begin{align*}
    f(\Delta^+(\xx'',B''))=(1,6,10,5)+(0,0,-1,-1)=(1,6,9,4).
\end{align*}
By repeating this algorithm, we can calculate face vectors of all positive cluster complexes of $A_n$ type.
In the case of other type, we can also calculate the same way as the above. If a non-Dynkin quiver $Q'$ appears when step (1) is complete, we need to calculate the face vector of a positive cluster complex associated with $Q'$ by applying this algorithm to $Q'$.

\section{Application to $\tau$-tilting theory of cluster-tilted algebras}\label{applicationtotautilting}
\subsection{$\tau$-tilting simplicial complexes}
In this section, we describe applications to the representation theory of algebras. We begin by introducing some notations and definitions. Through this section, $K$ is an algebraically closed field. Let $\Lambda$ be a finite dimensional basic $K$-algebra. We denote by $\mod \Lambda$ a category of finitely generated modules of $\Lambda$, and by $\proj\Lambda$ the full subcategory of $\mod\Lambda$ the objects of which are projective modules. We consider that a module $M\in\mod \Lambda$ is a \emph{$\tau$-rigid} module if $\Hom_{\Ccal}(M,\tau M)=0$, where $\tau$ is the AR-translation in $\Lambda$. Let $|M|$ be the number of non-isomorphic indecomposable direct summands of $M$. If $M$ is a $\tau$-rigid module and satisfies $|M|=|\Lambda|$, we say that $M$ is a \emph{$\tau$-tilting} module. 
We denote by $\tau\textrm{-tilt}\Lambda$ the set of isomorphism classes of basic $\tau$-tilting modules of $\Lambda$.

Next, we introduce a generalization of $\tau$-rigid, or $\tau$-tilting modules given by \cite{air}.
\begin{definition}
Let $(M,P)$ be a pair with $M\in\mod\Lambda$ and $P\in\proj\Lambda$.
\begin{itemize}
    \item [(1)] We say that $(M,P)$ is a \emph{$\tau$-rigid pair} of $\Lambda$ if $M$ is $\tau$-rigid and $\Hom_{\Lambda}(P,M)=0$,
     \item [(2)] We say that $(M,P)$ is a \emph{$\tau$-tilting pair} of $\Lambda$ if $(M,P)$ is a $\tau$-rigid pair and $|M|+|P|=|\Lambda|$.
\end{itemize}
\end{definition}
We say that a $\tau$-rigid pair $(M,P)$ is \emph{indecomposable} (resp. \emph{basic}) if $M\oplus P$ is indecomposable (resp. basic). Note that, if $(M,P)$ is indecomposable, then either $M$ or $P$ is 0. We consider the following simplicial complexes induced by $\tau$-rigid pairs and $\tau$-rigid modules.
\begin{definition}
\noindent
\begin{itemize}
    \item [(1)]
We define a \emph{support $\tau$-tilting simplicial complex} $\Delta(s\Lambda)$ of $\Lambda$ as a complex the vertex set of which comprises isomorphism classes of indecomposable $\tau$-rigid pairs and simplices of which are isomorphism classes of basic $\tau$-rigid pairs of $\Lambda$.
\item[(2)]
We define a \emph{$\tau$-tilting simplicial complex} $\Delta(\Lambda)$ of $\Lambda$ as a complex the vertex set of which is of isomorphism classes of indecomposable $\tau$-rigid modules and simplices of which are isomorphism classes of basic $\tau$-rigid modules of $\Lambda$. 
\end{itemize}
\end{definition}

For a finite dimensional basic $K$-algebra $\Lambda$, we denote by $Q_\Lambda$ a quiver, which satisfies $\Lambda\cong KQ_{\Lambda}/I$, where $I$ is an admissible ideal. 
In parallel with seed mutation, we define a $\tau$-tilting mutation. 
\begin{defprop}[\cite{air}*{Theorem 2.18}]
Let $\Lambda$ be a finite dimensional $K$-algebra and $(M\oplus N,P\oplus Q)$ a $\tau$-tilting pair of $\Lambda$, where $(N,Q)$ is an indecomposable $\tau$-rigid pair. There exists a unique indecomposable $\tau$-rigid pair $(N',Q')$ such that $(N,Q)\neq (N',Q')$ and $(M\oplus N',P\oplus Q')$ is a $\tau$-tilting pair. We say that $(M\oplus N',P\oplus Q')$ is the \emph{$\tau$-tilting mutation} of $(M,P)$ in direction $(N,Q)$, and we denote by $(M\oplus N',P\oplus Q')=\mu_{(N,P)}(M\oplus N,P\oplus Q)$.
\end{defprop}

On $\Delta(s\Lambda)$ or $\Delta(\Lambda)$, the mutation is the operation of moving from a maximal simplex to a neighboring maximal simplex. We also introduce the support $\tau$-tilting modules.

\begin{definition}
Let $M\in\mod \Lambda$. We say that $M$ is a \emph{support $\tau$-tilting module} if there exists an idempotent $e\in \Lambda$ such that $M$ is a $\tau$-tilting $\Lambda/\Lambda e \Lambda$-module.
\end{definition}

Let $M\in\mod \Lambda$. If $M$ is a $\tau$-tilting $\Lambda/\Lambda e\Lambda$-module, then $(M,\Lambda e)$ is a $\tau$-tilting pair of $\Lambda$. Conversely, if $(M,\Lambda e)$ is a $\tau$-tilting pair of $\Lambda$, then $M$ is a $\tau$-tilting $\Lambda/\Lambda e\Lambda$-module. We can identify the support $\tau$-tilting modules with $\tau$-tilting pairs through this correspondence. When $\Lambda$ has finitely many isomorphism classes of basic $\tau$-tilting modules, we say that $\Lambda$ is of \emph{finite $\tau$-tilting type} or \emph{$\tau$-tilting-finite algebra}. 

\subsection{Cluster categories and cluster-tilted algebras}
Let $Q$ be an acyclic (non-valued) quiver and $\Ccal_Q=\Dcal^b(KQ)/\tau^{-1}[1]$ a cluster category of $Q$.
For an additive category $\Ccal$, we say that $T\in\Ccal$ is a \emph{cluster tilting object} of $\Ccal$ if $T$ satisfies the following two conditions:
\begin{itemize}
    \item $\Hom_{\Ccal}(T,T[1])=0$,
    \item $\Hom_{\Ccal}(T,Y[1])=0$ implies $Y\in\add T$,
\end{itemize}
where $\add T$ is the full subcategory of $\Ccal$ whose objects are direct summands of direct sums of $T$.
We define a cluster tilting mutation parallel to a seed mutation or a $\tau$-tilting mutation.
\begin{defprop}[\cite{IY}*{Theorem 5.3}]
Let $\Ccal_Q$ be a cluster category and $T=U\oplus X$ a cluster tilting object in $\Ccal_Q$, where $X$ is an indecomposable object. There is a unique indecomposable $X'$ such that $X\neq X'$ holds and $T=U\oplus X'$ is cluster tilting objects. We say that $T\oplus U'$ is the \emph{cluster tilting mutation} of $T\oplus U$ in direction $U$, and we denote by $T\oplus U'=\mu_{U}(T\oplus U)$.
\end{defprop}

We consider algebras with a connection with cluster categories.
We say that $\Lambda$ is a \emph{cluster-tilted algebra} if there exists an acyclic quiver $Q$ and a cluster tilting object $T$ in $\Ccal_Q$ such that $\Lambda\cong (\End_{\Ccal_Q}T)^{\text{op}}$. 
Let $T,T'$ be a cluster tilting objects and $\Lambda=(\End_{\Ccal_Q}T)^{\text{op}},\ \Lambda'=(\End_{\Ccal_Q}T')^{\text{op}}$. We note that if $T'$ is obtained from $T$ by a cluster tilting mutation, then $Q_{\Lambda'}$ is obtained from $Q_\Lambda$ by a quiver mutation (\cite{birs}*{Theorem I.1,6}). The object $T=KQ$ is a cluster tilting object in $\Ccal_Q$, and then we have $Q_\Lambda=Q$. Therefore, for any cluster-tilted algebra $\Lambda$ of $\Ccal_Q$, $Q_\Lambda$ is mutation equivalent to $Q$. Moreover, for any quiver $Q$ which is mutation equivalent to an acyclic quiver, the corresponding cluster-tilted algebra $\Lambda=KQ/I$ is uniquely determined (see \cite{birsmith}*{Theorem 2.3}). 

We can identify $\tau$-tilting simplicial complexes with positive cluster complexes by the following theorem. 
\begin{theorem}\label{thm:iso-tilting-cluster}
Let $Q_\Lambda$ be a quiver which is mutation equivalent to a Dynkin quiver, and $B_\Lambda$ a skew-symmetric matrix whose associated quiver is $Q_\Lambda$. Then, we have $\Delta(s\Lambda)\cong\Delta(\xx,B_\Lambda)$ and its restriction gives $\Delta(\Lambda)\cong\Delta^+(\xx,B_\Lambda)$.
\end{theorem}
\begin{proof}
By \cite{air}*{Theorem 4.1}, we have a bijection between indecomposable rigid objects in $\Ccal_{Q}$ and indecomposable $\tau$-rigid pairs of $\Lambda$, which is lifted to a bijection between basic rigid objects in $\Ccal_{Q}$ (resp. basic rigid objects in $\Ccal_{Q}$, which do not have direct summand of non-zero direct summands in $\add T[1]$) and summands of basic $\tau$-tilting pairs (resp. basic $\tau$-rigid modules) of $\Lambda$. In contrast, by \cite{fk}*{Proposition 2.3} and \cite{cklp}*{Corollary 3.5}, the compositions of shift function and the inverse map of the cluster character $[1]\circ CC_T^{-1}$ gives a bijection between cluster variables in $P(\xx,B)$ and indecomposable objects in $\Ccal_{Q}$. It is lifted to a bijection between subsets of clusters in $P(\xx,B)$ (resp. subsets of clusters in $P(\xx,B)$ not containing initial variables) and rigid objects in $\Ccal_{Q}$ (resp. rigid objects in $\Ccal_{Q}$ which do not have direct summand of non-zero direct summands in $\add T[1]$). Therefore, we have the desired isomorphism.
\end{proof}
We remark that the isomorphism in Theorem \ref{thm:iso-tilting-cluster} is compatible with $\tau$-tilting/seed mutations. Therefore, we obtain the following proposition.
\begin{proposition}
A cluster-tilted algebra $\Lambda$ is finite representation type if and only if $\Delta(\xx,B_{\Lambda})$ is of finite type.
\end{proposition}
\begin{proof}
From the above discussion, $\Lambda$ is finite $\tau$-tilting type if and only if $\Delta(\xx,B_{\Lambda})$ is of finite type. Moreover, we have the equivalence of finite $\tau$-tilting type and finite representation type by \cite{zito}*{Theorem 3.1}.
\end{proof}
By Theorem \ref{thm:iso-tilting-cluster}, $\Delta(s\Lambda)$ and $\Delta(\Lambda)$ depend only on their quiver $Q_{\Lambda}$. In particular, $\Delta(s\Lambda)$ depends only on its mutation equivalence class. 
\subsection{Corollaries in representation theory}
Using an isomorphism in Theorem \ref{thm:iso-tilting-cluster} and results from Section 3, we have some corollaries on the cluster-tilted algebra of finite representation type.  Let $T=U\oplus X$ be a cluster tilting object in $\Ccal_Q$, where $X$ is indecomposable.

We denote by $e_X\in \Lambda$ an idempotent corresponding $X$ and $Q_\Lambda\setminus X$ a full subquiver of $Q$ consisting of all vertices of $Q$ except that corresponding to $X$. The following corollaries are analogs of Theorem \ref{thm:mutation-positivecomplex} and its corollaries.
\begin{corollary}\label{cor:mutation-tau-tilting-complex} 
Let $\Lambda=(\End_{\Ccal_Q}T)^{\mathrm{op}},\ \Lambda'=(\End_{\Ccal_Q}T')^{\mathrm{op}}$ be cluster-tilted algebras of finite representation type. We assume that $T'$ is obtained from $T$ by a mutation in direction $X$. Conversely, we assume that $T$ is obtained from $T'$ by a mutation in direction $X'$. We denote by $\Lambda_{-}=\Lambda/\Lambda e_X\Lambda$ and $\Lambda'_{-}=\Lambda/\Lambda e_{X'}\Lambda$. Then, we have
\begin{align}
   f(\Delta(\Lambda))-f(\Delta(\Lambda'))
   =\left[f(\Delta(\Lambda'_-))\right]_1-\left[f(\Delta(\Lambda_-))\right]_1.
\end{align}
\end{corollary}
\begin{proof}
By \cite{bmr08}*{Theorem 2.13}, $\Lambda/\Lambda e_X\Lambda$ is the cluster-tilted algebra associated with a quiver $Q_\Lambda\setminus X$. Therefore, the statement follows from Theorems \ref{thm:mutation-positivecomplex} and \ref{thm:iso-tilting-cluster}.
\end{proof}
\begin{corollary}\label{cor:invariant-sink/source-tilting}
 Let $\Lambda$ and $\Lambda'$ be cluster-tilted algebras of finite representation type. If $Q_{\Lambda'}$ is obtained from $Q_{\Lambda}$ by a sink or source quiver mutation, then we have a bijection between the simplicial set of $\Delta(\Lambda)$ and that of $\Delta(\Lambda')$. Furthermore, this bijection induces a bijection between simplices of $\Delta(\Lambda)$ and $\Delta(\Lambda')$ in each dimension. In particular, $\Lambda$ and $\Lambda'$ have the same numbers of isomorphism classes of basic $\tau$-rigid modules.
\end{corollary}
\begin{proof}
It follows Corollary \ref{thm:invariant-sink/source} and Theorem \ref{thm:iso-tilting-cluster}.
\end{proof}
Particularly, the following corollary has been given by \cite{mrz}*{Proposition 6.1} and \cite{eno}*{Theorem A.3} in other manners.
\begin{corollary}\label{cor:independent-orientation-tilting}
 Let $\Lambda$ and $\Lambda'$ be finite dimensional hereditary algebras. If $Q_\Lambda$ and $Q_{\Lambda'}$ have the same underlying graph, then we have a bijection between the simplicial set of $\Delta(\Lambda)$ and that of $\Delta(\Lambda')$. Furthermore, this bijection induces a bijection between simplices of $\Delta(\Lambda)$ and $\Delta(\Lambda')$ in each dimension.
\end{corollary}
\begin{proof}
The above follows from Corollary \ref{cor:independent-orientation} and Theorem \ref{thm:iso-tilting-cluster}.
\end{proof}
Using the result in Section 4, we have explicit descriptions of $\Delta(\Lambda)$ in some special cases. The following corollary is an explicit description of a ($\tau$-)tilting simplicial complex of special $A_n$ type.

\begin{corollary}\label{cor:explicit-descriptiton-An}
Let $\Lambda_{A_n}$ be a $n$-dimensional hereditary algebra whose quiver is given by \eqref{assumption:An}.
Then $\Delta(\Lambda_{A_n})$ is isomorphic to a cone of $\Delta(s\Lambda_{A_{n-1}})$.
\end{corollary}
\begin{proof}
The statement follows from Theorems \ref{thm:positive-simplex-A_n} and \ref{thm:iso-tilting-cluster}.
\end{proof}
\begin{remark}
We also provide a proof of Corollary \ref{cor:explicit-descriptiton-An} using $\tau$-tilting reduction. (See Appendix \ref{anotherproof}.)
\end{remark}
Before describing a ($\tau$-)tilting simplicial complex of special $D_n$ type, we define the compatibility of indecomposable modules.
\begin{definition}
For indecomposable modules $T,S$ of $\Lambda$, we say that $T$ is \emph{compatible} with $S$ if $T\oplus S$ is $\tau$-rigid.
\end{definition}
\begin{corollary}\label{cor:explicit-descriptiton-Dn}
Let $\Lambda_{D_n}$ be a $n$-dimensional hereditary algebra whose quiver is given by \eqref{assumption:Dn1} or \eqref{assumption:Dn2}. Let $I=\{T_1,\dots,T_n\}$ be the set of indecomposable summands of $T$ and $\Delta_{\mathrm{glue}}$ a full subcomplex of $\Delta(s\Lambda_{D_n})$ whose vertex set consists of all $0$-dimensional simplices of $\st_{\Delta(s\Lambda_{D_n})}(I\setminus\{T_{n-1},T_{n}\})$, which are not in $\{T_1,\dots,T_{n-2}\}$ and compatible with both $T_{n-1}$ and $T_n$.
Then, $\Delta(\Lambda_{D_n})$ is isomorphic to a complex obtained by gluing $\join(\Delta_0,\Delta_1,\Delta(s\Lambda_{A_{n-3}}))$ and $\st_{\Delta(s\Lambda_{D_n})}(I\setminus\{T_{n-1},T_{n}\})$ together along $\join (\Delta_1,\Delta(s\Lambda_{A_{n-3}}))$ and $\Delta_{\mathrm{glue}}$, where $\Delta_i$ is the $i$-dimensional simplex.
\end{corollary}
\begin{proof}
The statement follows from Theorems \ref{thm:positive-simplex-D_n} and \ref{thm:iso-tilting-cluster}.
\end{proof}
\begin{example}
Let $Q$ be a quiver $Q= \begin{xy}(0,0)*+{1}="I",(10,0)*+{2}="J", (20,0)*+{{3}}="K" \ar@{->}"I";"J"  \ar@/_4mm/"K";"I" \end{xy}$. We consider a cluster category $\Ccal_Q$ and its cluster tilting object $T=\substack{1\\2}\oplus 2\oplus\substack{3\\1\\2}$. Then $\Lambda:=(\End_{\Ccal_Q}T)^\mathrm{op}\cong KQ$ and $\Delta(s\Lambda)$ and $\Delta(\Lambda)$ of $\Lambda$ are shown in Figure \ref{fig:support-tau-tilting}. 
\begin{figure}[ht]
\caption{Support $\tau$-tilting simplicial complex and $\tau$-tilting simplicial complex of $\Lambda$}\label{fig:support-tau-tilting}
\vspace{5pt}
\begin{center}
\scalebox{0.8}{
\begin{tikzpicture}
\coordinate (0) at (0,0);
\coordinate (u*) at (90:5.33);
\coordinate (u) at (90:4);
\coordinate (ul) at (150:4);
\coordinate (ur) at (30:4);
\coordinate (uml) at (150:2);
\coordinate (umr) at (30:2);
\coordinate (dmc) at (-90:2);
\coordinate (dl) at (-150:4);
 \coordinate (dl*) at (-150:5.33);
\coordinate (dr) at (-30:4);
\coordinate (dr*) at (-30:5.33);
\coordinate (d) at (-90:4);
\draw (u) to (ul);
\draw (u) to (ur);
\draw (ul) to (uml);
\draw (ur) to (umr);
\draw (umr) to (uml);
\draw (u) to (uml);
\draw (u) to (umr);
\draw (uml) to (dl);
\draw (umr) to (dr);
\draw (ul) to (dl);
\draw (ur) to (dr);
\draw (uml) to (dmc);
\draw (umr) to (dmc);
\draw (dl) to (dmc);
\draw (dr) to (dmc);
\draw (dl) to (d);
\draw (dmc) to (d);
\draw (dr) to (d);
\draw(ul) [out=90,in=180]to (u*);
\draw(ur) [out=90,in=0]to (u*);
\draw(ul) [out=210,in=120]to (dl*);
\draw(d) [out=210,in=300]to (dl*);
\draw(d) [out=330,in=240]to (dr*);
\draw(ur) [out=330,in=60]to (dr*);
\fill [white](u) circle (0.8cm);
\fill[white] (ul) circle (0.9cm);
\fill[white] (ur) circle (0.7cm);
\fill[white] (uml) circle (0.7cm);
\fill[white] (umr) circle (0.7cm);
\fill[white] (dmc) circle (0.9cm);
\fill[white] (dl) circle (0.8cm);
\fill[white] (dr) circle (0.8cm);
\fill[white] (d) circle (0.7cm);
\node at (dmc) {$\left(0,\ \substack{3\\1\\2}\right)$};
\node at (umr) {{\Large$\left(0,\ 2\right)$}};
\node at (dr) {{\large$\left(0,\ \substack{1\\2}\right)$}};
\node at (ur) {{\Large$\left(3,\ 0\right)$}};
\node at (d) {{\Large$\left(2,\ 0\right)$}};
\node at (uml) {{\Large$\left(1,\ 0\right)$}};
\node at (ul) {$\left(\substack{3\\1\\2},\ 0\right)$};
\node at (dl) {{\large$\left(\substack{1\\2},\ 0\right)$}};
\node at (u) {{\large$\left(\substack{3\\1},\ 0\right)$}};
\end{tikzpicture}
\begin{tikzpicture}[baseline=-5cm]
 \coordinate (0) at (0,0);
 \coordinate (1) at (18:3);
 \coordinate (2) at (90:3);
 \coordinate (3) at (162:3);
 \coordinate (4) at (234:3);
 \coordinate (5) at (306:3);
 \draw(0) to (1);
 \draw(0) to (2);
 \draw(0) to (3);
 \draw(0) to (4);
 \draw(0) to (5);
 \draw(1) to (2);
 \draw(2) to (3);
 \draw(3) to (4);
 \draw(4) to (5);
 \draw(5) to (1);
\fill[white](0) circle (0.8cm);
\fill[white] (1) circle (0.7cm);
\fill[white] (2) circle (0.8cm);
\fill[white] (3) circle (0.7cm);
\fill[white] (4) circle (0.7cm);
\fill[white] (5) circle (0.8cm);
\node at (0) {$\substack{3\\1\\2}$};
\node at (1) {{\LARGE $1$}};
\node at (2) {{\Large$\substack{3\\1}$}};
\node at (3) {{\LARGE $3$}};
\node at (4) {{\LARGE $2$}};
\node at (5) {{\Large $\substack{1\\2}$}};
\end{tikzpicture}
}
\end{center}
\end{figure}
Let $T'$ be a cluster tilting object obtained from $T$ by a mutation in direction $\substack{3\\1\\2}$. Then, we have $T'=\substack{1\\2}\oplus 2\oplus\substack{3\\1\\2}\ [1]$ and $\Lambda':=(\End_{\Ccal_Q}T')^{\mathrm{op}}\cong KQ'$, where $Q'= \begin{xy}(0,0)*+{1}="I",(10,0)*+{2}="J", (20,0)*+{{3}}="K" \ar@{->}"I";"J"  \ar@/^4mm/"I";"K" \end{xy}$. Complexes $\Delta(s\Lambda')$ and $\Delta(\Lambda')$ of $\Lambda'$ are in Figure \ref{fig:support-tau-tilting2}. 
\begin{figure}[ht]
\caption{Support $\tau$-tilting simplicial complex and $\tau$-tilting simplicial complex of $\Lambda'$}\label{fig:support-tau-tilting2}
\vspace{5pt}
\begin{center}
\scalebox{0.8}{
\begin{tikzpicture}
\coordinate (0) at (0,0);
\coordinate (u*) at (90:5.33);
\coordinate (u) at (90:4);
\coordinate (ul) at (150:4);
\coordinate (ur) at (30:4);
\coordinate (uml) at (150:2);
\coordinate (umr) at (30:2);
\coordinate (dmc) at (-90:2);
\coordinate (dl) at (-150:4);
 \coordinate (dl*) at (-150:5.33);
\coordinate (dr) at (-30:4);
\coordinate (dr*) at (-30:5.33);
\coordinate (d) at (-90:4);
\draw (u) to (ul);
\draw (u) to (ur);
\draw (ul) to (uml);
\draw (ur) to (umr);
\draw (umr) to (uml);
\draw (u) to (uml);
\draw (u) to (umr);
\draw (uml) to (dl);
\draw (umr) to (dr);
\draw (ul) to (dl);
\draw (ur) to (dr);
\draw (uml) to (dmc);
\draw (umr) to (dmc);
\draw (dl) to (dmc);
\draw (dr) to (dmc);
\draw (dl) to (d);
\draw (dmc) to (d);
\draw (dr) to (d);
\draw(ul) [out=90,in=180]to (u*);
\draw(ur) [out=90,in=0]to (u*);
\draw(ul) [out=210,in=120]to (dl*);
\draw(d) [out=210,in=300]to (dl*);
\draw(d) [out=330,in=240]to (dr*);
\draw(ur) [out=330,in=60]to (dr*);
\fill [white](u) circle (0.7cm);
\fill[white] (ul) circle (0.8cm);
\fill[white] (ur) circle (0.7cm);
\fill[white] (uml) circle (0.8cm);
\fill[white] (umr) circle (0.7cm);
\fill[white] (dmc) circle (0.7cm);
\fill[white] (dl) circle (0.9cm);
\fill[white] (dr) circle (0.9cm);
\fill[white] (d) circle (0.7cm);
\node at (dmc) {{\Large $\left(3,\ 0\right)$}};
\node at (umr) {{\Large$\left(0,\ 2\right)$}};
\node at (dr) {{\large$\left(0,\ \substack{1\\2\ 3}\right)$}};
\node at (ur) {{\Large$\left(0,\ 3\right)$}};
\node at (d) {{\Large$\left(2,\ 0\right)$}};
\node at (uml){{\large$\left(\substack{1\\3},\ 0\right)$}};
\node at (ul){{\large$\left(\substack{1\\2},\ 0\right)$}};
\node at (dl) {{\large$\left(\substack{1\\2\ 3},\ 0\right)$}};
\node at (u) {{\Large$\left(1,\ 0\right)$}};
\end{tikzpicture}
\begin{tikzpicture}[baseline=-50mm]
 \coordinate (0) at (306:3);
 \coordinate (1) at (18:3);
 \coordinate (2) at (90:3);
 \coordinate (3) at (0,0);
 \coordinate (4) at (234:3);
 \coordinate (5) at (306:6);
 \draw(0) to (1);
 \draw(1) to (3);
 \draw(0) to (3);
 \draw(0) to (4);
 \draw(0) to (5);
 \draw(1) to (2);
 \draw(2) to (3);
 \draw(3) to (4);
 \draw(4) to (5);
 \draw(5) to (1);
\fill[white](0) circle (0.8cm);
\fill[white] (1) circle (0.8cm);
\fill[white] (2) circle (0.7cm);
\fill[white] (3) circle (0.8cm);
\fill[white] (4) circle (0.7cm);
\fill[white] (5) circle (0.7cm);
\node at (0) {{\Large$\substack{1\\2\ 3}$}};
\node at (1) {{\Large$\substack{1\\3}$}};
\node at (2) {{\LARGE$1$}};
\node at (3) {{\Large$\substack{1\\2}$}};
\node at (4) {{\LARGE$2$}};
\node at (5) {{\LARGE$3$}};
\end{tikzpicture}
}
\end{center}
\end{figure}
On the other hand, let $T''$ be a cluster tilting object obtained from $T$ by a mutation in direction $\substack{1\\2}$. Then, we have $T''=3\oplus 2\oplus\substack{3\\1\\2}$ and $\Lambda':=(\End_{\Ccal_Q}T'')^\mathrm{op}\cong KQ''/I$, where $Q''= \begin{xy}(0,0)*+{1}="I",(10,0)*+{2}="J", (20,0)*+{{3}}="K" \ar_{\beta}@{<-}"I";"J" \ar^{\alpha}@/^4mm/"I";"K" \ar_{\gamma}@{<-}"J";"K"\end{xy}$ and $I=\langle\alpha\beta,\beta\gamma,\gamma\alpha\rangle$. The complexes $\Delta(s\Lambda'')$ and $\Delta(\Lambda'')$ of $\Lambda'$ are in Figure \ref{fig:support-tau-tilting3}. 
\begin{figure}[ht]
\caption{Support $\tau$-tilting simplicial complex and $\tau$-tilting simplicial complex of $\Lambda''$}\label{fig:support-tau-tilting3}
\vspace{5pt}
\begin{center}
\scalebox{0.8}{
\begin{tikzpicture}
\coordinate (0) at (0,0);
\coordinate (u*) at (90:5.33);
\coordinate (u) at (90:4);
\coordinate (ul) at (150:4);
\coordinate (ur) at (30:4);
\coordinate (uml) at (150:2);
\coordinate (umr) at (30:2);
\coordinate (dmc) at (-90:2);
\coordinate (dl) at (-150:4);
 \coordinate (dl*) at (-150:5.33);
\coordinate (dr) at (-30:4);
\coordinate (dr*) at (-30:5.33);
\coordinate (d) at (-90:4);
\draw (u) to (ul);
\draw (u) to (ur);
\draw (ul) to (uml);
\draw (ur) to (umr);
\draw (umr) to (uml);
\draw (u) to (uml);
\draw (u) to (umr);
\draw (uml) to (dl);
\draw (umr) to (dr);
\draw (ul) to (dl);
\draw (ur) to (dr);
\draw (uml) to (dmc);
\draw (umr) to (dmc);
\draw (dl) to (dmc);
\draw (dr) to (dmc);
\draw (dl) to (d);
\draw (dmc) to (d);
\draw (dr) to (d);
\draw(ul) [out=90,in=180]to (u*);
\draw(ur) [out=90,in=0]to (u*);
\draw(ul) [out=210,in=120]to (dl*);
\draw(d) [out=210,in=300]to (dl*);
\draw(d) [out=330,in=240]to (dr*);
\draw(ur) [out=330,in=60]to (dr*);
\fill [white](u) circle (0.7cm);
\fill[white] (ul) circle (0.8cm);
\fill[white] (ur) circle (0.8cm);
\fill[white] (uml) circle (0.8cm);
\fill[white] (umr) circle (0.8cm);
\fill[white] (dmc) circle (0.8cm);
\fill[white] (dl) circle (0.7cm);
\fill[white] (dr) circle (0.7cm);
\fill[white] (d) circle (0.8cm);
\node at (dmc) {{\large $\left(0,\ \substack{3\\2}\right)$}};
\node at (umr) {{\large$\left(0,\ \substack{2\\1}\right)$}};
\node at (dr) {{\Large$\left(1,\ 0\right)$}};
\node at (ur) {{\large$\left(\substack{1\\3},\ 0\right)$}};
\node at (d) {{\large$\left(\substack{2\\1},\ 0\right)$}};
\node at (uml){{\large$\left(0,\ \substack{1\\3}\right)$}};
\node at (ul){{\large$\left(\substack{3\\2},\ 0\right)$}};
\node at (dl) {{\Large$\left(2,\ 0\right)$}};
\node at (u) {{\Large$\left(3,\ 0\right)$}};
\end{tikzpicture}
\begin{tikzpicture}[baseline=-50mm]
 \coordinate (0) at (306:3);
 \coordinate (1) at (198:5);
 \coordinate (2) at (90:3);
 \coordinate (3) at (0,0);
 \coordinate (4) at (234:3);
 \coordinate (5) at (162:3);
 \draw(0) to (3);
 \draw(0) to (4);
 \draw(2) to (3);
 \draw(3) to (4);
 \draw(3) to (5);
 \draw(2) to (5);
 \draw(4) to (5);
 \draw(1) to (4);
 \draw(1) to (5);
\fill[white](0) circle (0.7cm);
\fill[white] (1) circle (0.7cm);
\fill[white] (2) circle (0.7cm);
\fill[white] (3) circle (0.8cm);
\fill[white] (4) circle (0.8cm);
\fill[white] (5) circle (0.8cm);
\node at (2) {{\LARGE 3}};
\node at (3) {{\Large$\substack{3\\2}$}};
\node at (5) {{\Large$\substack{1\\3}$}};
\node at (4) {{\Large$\substack{2\\1}$}};
\node at (0) {{\LARGE 2}};
\node at (1) {{\LARGE 1}};
\end{tikzpicture}
}
\end{center}
\end{figure}
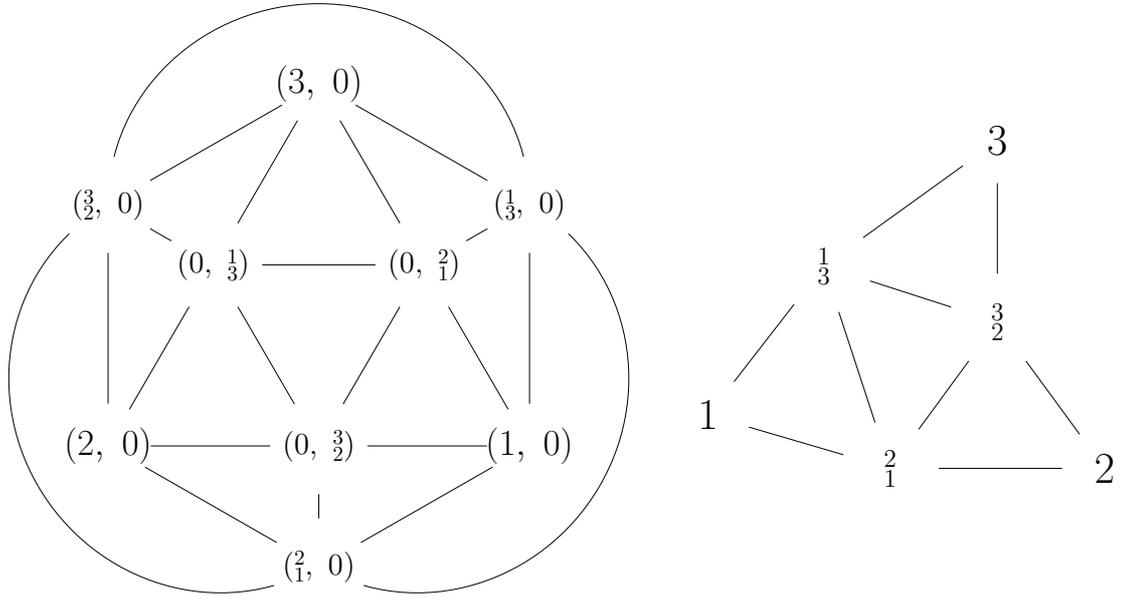
\end{example}
By Theorem \ref{thm:iso-tilting-cluster}, we have
$\Delta(s\Lambda)\cong\Delta(s\Lambda')\cong\Delta(s\Lambda'')\cong\Delta(\xx,B)$,
$\Delta(\Lambda)\cong\Delta^+(\xx,B)$, $\Delta(\Lambda')\cong\Delta^+(\xx',B')$, and 
$\Delta(\Lambda'')\cong\Delta^+(\xx'',B'')$, where $\Delta^+(\xx,B)$, $\Delta(\xx,B)$, $\Delta^+(\xx',B'')$, and $\Delta^+(\xx'',B'')$ is the same notation as those in Examples \ref{ex:not-iso-positive-complex} and \ref{ex:not-equal-simplex}.

\subsection{$h$-vectors and Hasse quivers}\label{applicationtotautilting2}

We consider the application of $h$-vectors of $\tau$-tilting simplicial complexes.

\begin{definition}
Let $K$ be a pure simplicial complex and $K_{\max}$ the set of all maximal simplices in $K$. We say that $K$ is \emph{shellable} if there exists a total order $\preceq$ in $K_{\max}$ such that $\left(\cup_{S'\prec S}S'\right)\cap S$ is a pure proper maximal simplicial complex for all $S\in K_{\max}$.  
\end{definition}

We prove the shellability of $\tau$-tilting simplicial complexes. 

\begin{theorem}\label{shellability}
For a $\tau$-tilting-finite algebra $\Lambda$, $\Delta(\Lambda)$ is shellable.
\end{theorem}

Let $s\tau\textrm{-}\mathrm{tilt}\Lambda$ be the set of all isomorphism classes of basic support $\tau$-tilting modules in $\mod \Lambda$. To prove Theorem \ref{shellability}, we introduce the partial order $\preceq$ in $s\tau\textrm{-}\mathrm{tilt}\Lambda$. For $M,N\in s\tau\textrm{-}\mathrm{tilt}\Lambda$, we define $M \preceq N$ if $\mathsf{Fac} M\subset \mathsf{Fac} N$. The \emph{Hasse quiver} $\mathsf{Hasse}(s\tau\textrm{-}\mathrm{tilt}\Lambda)$ of $\Lambda$ is the following quiver:
\begin{itemize}
    \item The vertex set is $s\tau\textrm{-}\mathrm{tilt}\Lambda$.
    \item Let $M,N$ be support $\tau$-tilting modules in $\Lambda$. We draw an arrow $M\to N$ if $M\succ N$ and there is no support $\tau$-tilting module $L$ such that $M\succ L \succ N$.
    \end{itemize} 
We denote by $\mathsf{Hasse}(\tau\textrm{-}\mathrm{tilt}\Lambda)$ the full subquiver of $\mathsf{Hasse}(s\tau\textrm{-}\mathrm{tilt}\Lambda)$ whose vertex set consists of $\tau$-tilting module of $\Lambda$.  

\begin{example}
Let $\Lambda_{A_3}$ be a cluster-tilted algebra associated with the linearly oriented quiver of $A_3$ type. Then, $\mathsf{Hasse}(s\tau\textrm{-}\mathrm{tilt}\Lambda)$ and $\mathsf{Hasse}(\tau\textrm{-}\mathrm{tilt}\Lambda)$ are as in Figure \ref{hasseA3}. 
 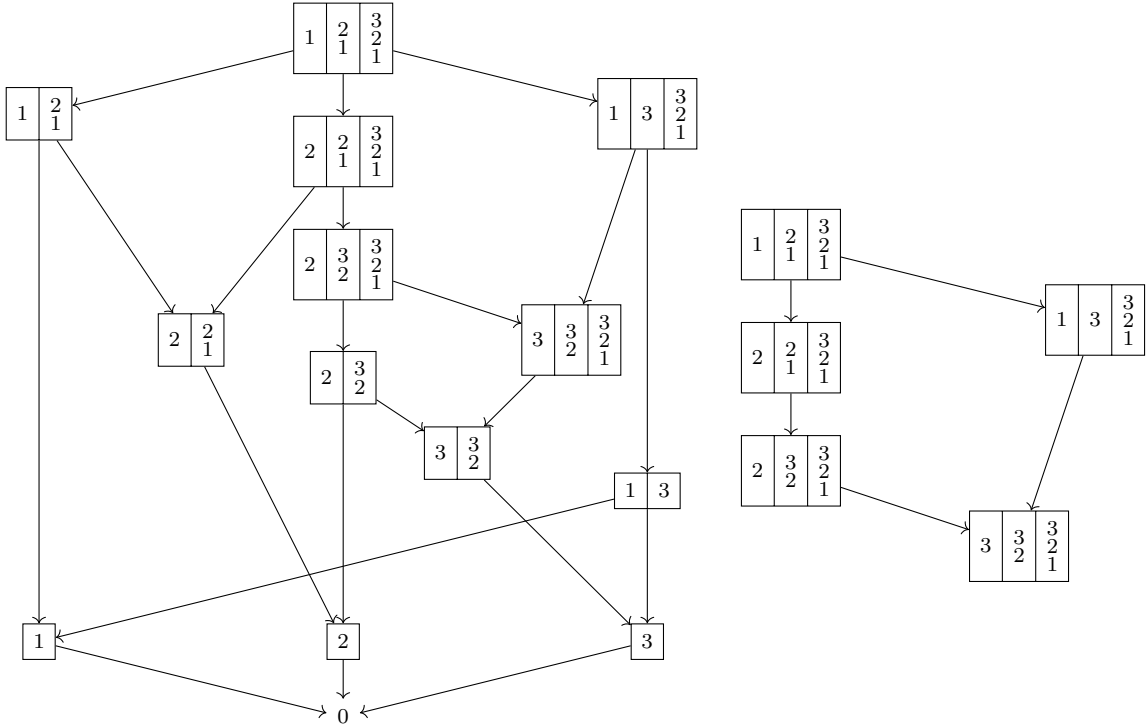
\begin{figure}[htp]
  \caption{$\mathsf{Hasse}(s\tau\text{-tilt} \Lambda_{A_3})$ and $\mathsf{Hasse}(\tau\text{-tilt} \Lambda_{A_3})$}
    \label{hasseA3}
\vspace{3mm}    
    \begin{tikzpicture}
      \node[block=3] (121321) at (0,10) {\nodepart{one}$\sst{1}$\nodepart{two}$\sst{2 \\ 1}$\nodepart{three}$\sst{3 \\ 2 \\ 1}$};
      \node[block=2] (121) at (-4,9) {\nodepart{one}$\sst{1}$\nodepart{two}$\sst{2 \\ 1}$};
      \node[block=3] (221321) at (0,8.5) {\nodepart{one}$\sst{2}$\nodepart{two}$\sst{2 \\ 1}$\nodepart{three}$\sst{3 \\ 2 \\ 1}$};
      \node[block=3] (13321) at (4,9) {\nodepart{one}$\sst{1}$\nodepart{two}$\sst{3}$\nodepart{three}$\sst{3 \\ 2 \\ 1}$};
      \node[block=2] (221) at (-2,6) {\nodepart{one}$\sst{2}$\nodepart{two}$\sst{2 \\ 1}$};
      \node[block=3] (232321) at (0,7) {\nodepart{one}$\sst{2}$\nodepart{two}$\sst{3 \\ 2}$\nodepart{three}$\sst{3 \\ 2 \\ 1}$};
      \node[block=2] (232) at (0,5.5) {\nodepart{one}$\sst{2}$\nodepart{two}$\sst{3 \\ 2}$};
      \node[block=3] (332321) at (3,6) {\nodepart{one}$\sst{3}$\nodepart{two}$\sst{3 \\ 2}$\nodepart{three}$\sst{3 \\ 2 \\ 1}$};
      \node[block=2] (13) at (4,4) {\nodepart{one}$\sst{1}$\nodepart{two}$\sst{3}$};
      \node[rectangle,draw] (2) at (0,2) {\nodepart{one}$\sst{2}$};
      \node[block=2] (332) at (1.5,4.5) {\nodepart{one}$\sst{3}$\nodepart{two}$\sst{3 \\ 2}$};
      \node[rectangle,draw] (3) at (4,2) {\nodepart{one}$\sst{3}$};
      \node[rectangle,draw] (1) at (-4,2) {\nodepart{one}$\sst{1}$};
      \node (0) at (0,1) {$\sst{0}$};

      \draw[->] (121321) -- (121);
      \draw[->] (121321) -- (221321);
      \draw[->] (121321) -- (13321);
      \draw[->] (121) -- (221);
      \draw[->] (121) -- (1);
      \draw[->] (221321) -- (221);
      \draw[->] (221321) -- (232321);
      \draw[->] (13321) -- (332321);
      \draw[->] (13321) --  (13);
      \draw[->] (221) -- (2);
      \draw[->] (232321) -- (232);
      \draw[->] (232321) -- (332321);
      \draw[->] (232) -- (332);
      \draw[->] (232) -- (2);
      \draw[->] (332321) -- (332);
      \draw[->] (13) --  (1);
      \draw[->] (13) -- (3);
      \draw[->] (332) -- (3);
      \draw[->] (1) -- (0);
      \draw[->] (2) -- (0);
      \draw[->] (3) -- (0);
    \end{tikzpicture}
    \hspace{3mm}
    \begin{tikzpicture}[baseline=35mm]
      \node[block=3] (121321) at (0,10) {\nodepart{one}$\sst{1}$\nodepart{two}$\sst{2 \\ 1}$\nodepart{three}$\sst{3 \\ 2 \\ 1}$};
      \node[block=3] (221321) at (0,8.5) {\nodepart{one}$\sst{2}$\nodepart{two}$\sst{2 \\ 1}$\nodepart{three}$\sst{3 \\ 2 \\ 1}$};
      \node[block=3] (13321) at (4,9) {\nodepart{one}$\sst{1}$\nodepart{two}$\sst{3}$\nodepart{three}$\sst{3 \\ 2 \\ 1}$};
      \node[block=3] (232321) at (0,7) {\nodepart{one}$\sst{2}$\nodepart{two}$\sst{3 \\ 2}$\nodepart{three}$\sst{3 \\ 2 \\ 1}$};
      \node[block=3] (332321) at (3,6) {\nodepart{one}$\sst{3}$\nodepart{two}$\sst{3 \\ 2}$\nodepart{three}$\sst{3 \\ 2 \\ 1}$};
      \draw[->] (121321) -- (221321);
      \draw[->] (121321) -- (13321);
      \draw[->] (221321) -- (232321);
      \draw[->] (13321) -- (332321);
      \draw[->] (232321) -- (332321);
    \end{tikzpicture}
  \end{figure}
\end{example}
We use the following characterization of the Hasse quiver of $s\tau\textrm{-}\mathrm{tilt}\Lambda$.
\begin{theorem}[\cite{air}*{Theorem 2.33}]\label{mutation-charactorization}
For $\tau$-tilting pairs $(M,P)$ and $(M',P')$, the following conditions are equivalent.
\begin{itemize}
    \item[(i)] $(M,P)$ is a mutation of $(M',P')$,
    \item[(ii)] $M$ is connected to $M'$ by an edge in $\mathsf{Hasse}(s\tau\textrm{-}\mathrm{tilt}\Lambda)$.
\end{itemize}
\end{theorem}

\begin{lemma}\label{tau-tiltlarger}
Let $\Lambda$ be a $\tau$-tilting-finite algebra. For any basic support $\tau$-tilting $\Lambda$-modules $M,N$, if there is an arrow $M\to N$ in $\mathsf{Hasse}(s\tau\text{-}\mathrm{tilt}\Lambda)$, then we have $|M|\geq|N|$. 
\end{lemma}
\begin{proof}
Let $(M,P)$ and $(N,Q)$ be the corresponding $\tau$-tilting pair of $M$ and $N$. By Theorem \ref{mutation-charactorization}, there exist indecomposable modules $X$ and $X'$ such that $((M\setminus X)\oplus X',P)=(N, Q)$, $(M\setminus X,P\oplus X')=(N, Q)$, or $(M\oplus X',P\setminus X)=(N, Q)$. If $(M\oplus X',P\setminus X)=(N, Q)$, then we have $\textsf{Fac}M\subset\textsf{Fac}N$. This contradicts $M\to N$. Therefore, we have $((M\setminus X)\oplus X',P)=(N, Q)$ or $(M\setminus X,P\oplus X')=(N, Q)$ and $|M|\geq|N|$.
\end{proof}

For support $\tau$-tilting simplicial complexes, the following theorem is known.
\begin{theorem}[\cite{dij}*{Theorem 5.4}]\label{shellability-support}
Let $\Lambda$ be a $\tau$-tilting-finite algebra. Then, $\Delta(s\Lambda)$ is shellable for a total order preserving the partial orders $\preceq$ or $\preceq^{\mathrm{op}}$.
\end{theorem}

\begin{proof}[Proof of Theorem \ref{shellability}]
We set $|\Lambda|=n$. It suffices to construct a total order preserving the partial order $\preceq^\mathrm{op}$ such that every $\tau$-tilting module is smaller than every non-$\tau$-tilting support $\tau$-tilting module. We prove the dual; that is, there is a total order preserving the partial order $\preceq$ such that every $\tau$-tilting module is larger than all non-$\tau$-tilting support $\tau$-tilting modules. Choose the maximal element in $\Delta(s\Lambda)_{\max}$ in the partial order $\preceq$, remove the chosen element from the whole, and then choose one of the maximal elements for $\preceq$ from the set, and so on, until the set is empty. In this case, for any two support $\tau$-tilting modules, if we determine the order under the rule that the one which is chosen earlier is greater, then this order is the total order such that the original partial order is preserved. Now, we can make the order of choosing support $\tau$-tilting modules such that all $\tau$-tilting modules are chosen before choosing the others. Assuming that this is impossible, there will be arrows from supporting $\tau$-tilting modules with less than $n$ direct commutations to support $\tau$-tilting modules with $n$, which contradicts Lemma \ref{tau-tiltlarger}. 
\end{proof}

Let $K$ be a shellable simplicial complex for the order $\preceq$ of $K_{\max}$. We set $\{F_1,\dots,F_m\}=K_{\max}$, where $F_i\preceq F_j$ if and only if $i\leq j$. Let $N(F_j)$ be the number of vertices $x$ of $F_j$ such that at least one face of codimension 1 of $F_{j}$ not containing $x$ is contained in $F_1, F_2,\dots, F_{j-1}$. Let $h=(h_0,h_1,\dots,h_n)$ be the $h$-vector of $K$. Then, we have
\begin{align*}
    h_i=\#\{1\leq j \leq m \mid N(F_j)=i\}.
\end{align*}
For further details, see \cite{zie}*{Section 8.3}. Let $\Lambda$ be a $\tau$-tilting-finite algebra. Since every codimension 1 simplex in $\Delta(\Lambda)$ is contained 1 or 2 maximal simplices, $N(F_j)$ coincides with the number of $F_k$ such that $k<j$ and the codimension of $F_k\cap F_j$ is 1. Furthermore, this is the number of arrows the terminals of which are the support tilting module corresponding to $F_j$ in $\mathsf{Hasse}(\tau\textrm{-tilt}\Lambda)$. For $M\in\tau\textrm{-tilt}\Lambda$, we denote by $T(M)$ the number of arrows whose terminal is $M$ in $\mathsf{Hasse}(\tau\textrm{-tilt}\Lambda)$. Set 
\begin{align}\label{anotherreph-vector}
    h_i(\Lambda)=\#\{M\in\tau\textrm{-tilt}\Lambda\mid T(M)=i\}.
\end{align}
By Theorem \ref{shellability}, $(h_0(\Lambda),h_1(\Lambda),\dots,h_n(\Lambda))$ is the $h$-vector of $\Delta(\Lambda)$. Therefore, when $\Delta(\Lambda)$ is a cluster-tilted algebra of finite representation type, we have the following corollaries by Corollaries \ref{thm:h-invariant-sink/source}, \ref{cor:h-independent-orientation} and Theorem \ref{thm:iso-tilting-cluster}.

\begin{corollary}
 Let $\Lambda$ and $\Lambda'$ be cluster-tilted algebras of the finite representation type. If $Q_{\Lambda'}$ is obtained from $Q_{\Lambda}$ by a sink or source quiver mutation, then $h_i(\Lambda)$ coincide with $h_i(\Lambda')$ for any $i$.
\end{corollary}

\begin{corollary}
  Let $\Lambda$ and $\Lambda'$ be finite dimensional hereditary algebras. If $Q_\Lambda$ and $Q_{\Lambda'}$ have the same underlying graph, then $h_i(\Lambda)$ coincide with $h_i(\Lambda')$ for any $i$.
\end{corollary}

\begin{example}
Let $Q= \begin{xy}(0,0)*+{1}="I",(10,0)*+{2}="J", (20,0)*+{{3}}="K",(30,0)*+{{4}}="L" \ar@{->}"J";"I"  \ar@{->}"K";"J" \ar@{->}"L";"K"  \end{xy}$ and $Q'= \begin{xy}(0,0)*+{1}="I",(10,0)*+{2}="J", (20,0)*+{{3}}="K", (30,0)*+{{4}}="L" \ar@{->}"J";"I"  \ar@{->}"K";"J",\ar@{->}"K";"L"\end{xy}$.
Let $\Lambda,\Lambda'$ be cluster-tilted algebras associated with $Q$ and $Q'$, respectively. The $h$-vector of both $\Delta(\Lambda)$ and $\Delta(\Lambda')$ are $h = (1,6,6,1)$ (see Figure \ref{hasselambda}).

 \begin{figure}[htp]
  \caption{$\mathsf{Hasse}(\tau\text{-tilt} \Lambda)$ and $\mathsf{Hasse}(\tau\text{-tilt} \Lambda')$}\label{hasselambda}
  \vspace{4mm}
    \begin{tikzpicture}
      \node[block=4] (121321) at (0,10) {\nodepart{one}$\sst{1}$\nodepart{two}$\sst{2 \\ 1}$\nodepart{three}$\sst{3 \\ 2 \\ 1}$\nodepart{four}$\sst{4\\ 3 \\ 2 \\ 1}$};
      \node[block=4] (121) at (-4,9) {\nodepart{one}$\sst{1}$\nodepart{two}$\sst{2 \\ 1}$\nodepart{three}$\sst{4}$\nodepart{four}$\sst{4\\ 3 \\ 2 \\ 1}$};
      \node[block=4] (221321) at (0,8.5) {\nodepart{one}$\sst{2}$\nodepart{two}$\sst{2 \\ 1}$\nodepart{three}$\sst{3 \\ 2 \\ 1}$\nodepart{four}$\sst{4\\ 3 \\ 2 \\ 1}$};
      \node[block=4] (13321) at (4,9) {\nodepart{one}$\sst{1}$\nodepart{two}$\sst{3}$\nodepart{three}$\sst{3 \\ 2 \\ 1}$\nodepart{four}$\sst{4\\ 3 \\ 2 \\ 1}$};
      \node[block=4] (221) at (-2.5,6.5) {\nodepart{one}$\sst{2}$\nodepart{two}$\sst{2 \\ 1}$\nodepart{three}$\sst{4}$\nodepart{four}$\sst{4\\ 3 \\ 2 \\ 1}$};
      \node[block=4] (232321) at (0,7) {\nodepart{one}$\sst{2}$\nodepart{two}$\sst{3 \\ 2}$\nodepart{three}$\sst{3 \\ 2 \\ 1}$\nodepart{four}$\sst{4\\ 3 \\ 2 \\ 1}$};
      \node[block=4] (232) at (0,5.5) {\nodepart{one}$\sst{2}$\nodepart{two}$\sst{3 \\ 2}$\nodepart{three}$\sst{4\\ 3 \\ 2}$\nodepart{four}$\sst{4\\ 3 \\ 2 \\ 1}$};
      \node[block=4] (332321) at (3,6) {\nodepart{one}$\sst{3}$\nodepart{two}$\sst{3 \\ 2}$\nodepart{three}$\sst{3 \\ 2 \\ 1}$\nodepart{four}$\sst{4\\ 3 \\ 2 \\ 1}$};
      \node[block=4] (13) at (4,4) {\nodepart{one}$\sst{1}$\nodepart{two}$\sst{3}$\nodepart{three}$\sst{4\\ 3}$\nodepart{four}$\sst{4\\ 3 \\ 2 \\ 1}$};
      \node[block=4] (2) at (0,2) {\nodepart{one}$\sst{2}$\nodepart{two}$\sst{4}$\nodepart{three}$\sst{4\\ 3 \\ 2}$\nodepart{four}$\sst{4\\ 3 \\ 2 \\ 1}$};
      \node[block=4] (332) at (2,4.5) {\nodepart{one}$\sst{3}$\nodepart{two}$\sst{3 \\ 2}$\nodepart{three}$\sst{4\\ 3 \\ 2 }$\nodepart{four}$\sst{4\\ 3 \\ 2 \\ 1}$};
      \node[block=4] (3) at (4,2) {\nodepart{one}$\sst{3}$\nodepart{two}$\sst{4\\ 3}$\nodepart{three}$\sst{4\\ 3 \\ 2 }$\nodepart{four}$\sst{4\\ 3 \\ 2 \\ 1}$};
      \node [block=4](1) at (-4,2) {\nodepart{one}$\sst{1}$\nodepart{two}$\sst{4}$\nodepart{three}$\sst{4\\ 3}$\nodepart{four}$\sst{4\\ 3 \\ 2 \\ 1}$};
      \node [block=4](0) at (0,0.5) {\nodepart{one}$\sst{4}$\nodepart{two}$\sst{4\\3}$\nodepart{three}$\sst{4\\ 3\\2}$\nodepart{four}$\sst{4\\ 3 \\ 2 \\ 1}$};

      \draw[->] (121321) -- (121);
      \draw[->] (121321) -- (221321);
      \draw[->] (121321) -- (13321);
      \draw[->] (121) -- (221);
      \draw[->] (121) -- (1);
      \draw[->] (221321) -- (221);
      \draw[->] (221321) -- (232321);
      \draw[->] (13321) -- (332321);
      \draw[->] (13321) --  (13);
      \draw[->] (221) -- (2);
      \draw[->] (232321) -- (232);
      \draw[->] (232321) -- (332321);
      \draw[->] (232) -- (332);
      \draw[->] (232) -- (2);
      \draw[->] (332321) -- (332);
      \draw[->] (13) --  (1);
      \draw[->] (13) -- (3);
      \draw[->] (332) -- (3);
      \draw[->] (1) -- (0);
      \draw[->] (2) -- (0);
      \draw[->] (3) -- (0);
    \end{tikzpicture}
    \vspace{8mm}
    \\
   \begin{tikzpicture}
      \node[block=4] (14213241) at (0,10) {\nodepart{one}$\sst{1}$\nodepart{two}$\sst{4}$\nodepart{three}$\sst{2 \\ 1}$\nodepart{four}$\sst{3\\ 24 \\ 1\ }$};
      \node[block=4] (1213213241) at (-4,9) {\nodepart{one}$\sst{1}$\nodepart{two}$\sst{2 \\ 1}$\nodepart{three}$\sst{3\\2\\1}$\nodepart{four}$\sst{3\\24\\1\ }$};
      \node[block=4] (14343241) at (0,8.5) {\nodepart{one}$\sst{1}$\nodepart{two}$\sst{4}$\nodepart{three}$\sst{3 \\ 4}$\nodepart{four}$\sst{3\\ 24 \\ 1\ }$};
      \node[block=4] (24213241) at (4,9) {\nodepart{one}$\sst{2}$\nodepart{two}$\sst{4}$\nodepart{three}$\sst{2\\1}$\nodepart{four}$\sst{3\\ 24 \\ 1\ }$};
      \node[block=4] (1343213241) at (-4,6) {\nodepart{one}$\sst{1}$\nodepart{two}$\sst{3\\4}$\nodepart{three}$\sst{3\\2\\1}$\nodepart{four}$\sst{3\\ 24 \\ 1\ }$};
      \node[block=4] (2213213241) at (0,7) {\nodepart{one}$\sst{2}$\nodepart{two}$\sst{2\\1}$\nodepart{three}$\sst{3 \\ 2 \\ 1}$\nodepart{four}$\sst{3\\ 24 \\ 1\ }$};
      \node[block=4] (4343243241) at (4,5) {\nodepart{one}$\sst{4}$\nodepart{two}$\sst{3 \\ 4}$\nodepart{three}$\sst{3 \\ 24}$\nodepart{four}$\sst{3\\ 24 \\ 1\ }$};
      \node[block=4] (243243241) at (4,7) {\nodepart{one}$\sst{2}$\nodepart{two}$\sst{4}$\nodepart{three}$\sst{3 \\ 24}$\nodepart{four}$\sst{3\\ 24 \\ 1\ }$};
      \node[block=4] (1334321) at (4,2) {\nodepart{one}$\sst{1}$\nodepart{two}$\sst{3}$\nodepart{three}$\sst{3\\ 4}$\nodepart{four}$\sst{3\\2\\1}$};
      \node[block=4] (343243213241) at (-4,4) {\nodepart{one}$\sst{3\\4}$\nodepart{two}$\sst{3\\24}$\nodepart{three}$\sst{3\\ 2 \\ 1}$\nodepart{four}$\sst{3\\ 24 \\ 1\ }$};
      \node[block=4] (23243213241) at (0,5.5) {\nodepart{one}$\sst{2}$\nodepart{two}$\sst{3 \\ 24}$\nodepart{three}$\sst{ 3 \\ 2 \\ 1 }$\nodepart{four}$\sst{3\\24\\1\ }$};
      \node[block=4] (232324321) at (0,2.5) {\nodepart{one}$\sst{2}$\nodepart{two}$\sst{3\\2}$\nodepart{three}$\sst{3 \\ 24 }$\nodepart{four}$\sst{3 \\ 2 \\ 1}$};
      \node [block=4](3234324321) at (-4,2) {\nodepart{one}$\sst{3\\2}$\nodepart{two}$\sst{3\\4}$\nodepart{three}$\sst{3\\ 24}$\nodepart{four}$\sst{3 \\ 2 \\ 1}$};
      \node [block=4](33234321) at (0,0.5) {\nodepart{one}$\sst{3}$\nodepart{two}$\sst{3\\2}$\nodepart{three}$\sst{3\\ 4}$\nodepart{four}$\sst{3 \\ 2 \\ 1}$};

      \draw[->] (14213241) -- (1213213241);
      \draw[->] (14213241) -- (14343241);
      \draw[->] (14213241) -- (24213241);
      \draw[->] (1213213241) -- (1343213241);
      \draw[->] (1213213241) -- (2213213241);
      \draw[->] (14343241) -- (4343243241);
      \draw[->] (24213241) -- (2213213241);
      \draw[->] (24213241) -- (243243241);
      \draw[->] (1343213241) --  (1334321);
      \draw[->] (1343213241) -- (343243213241);
      \draw[->] (2213213241) -- (23243213241);
      \draw[->] (4343243241) -- (343243213241);
      \draw[->] (243243241) -- (4343243241);
      \draw[->] (23243213241) -- (343243213241);
      \draw[->] (23243213241) -- (232324321);
      \draw[->] (14343241) --  (1343213241);
      \draw[->] (23243213241) -- (232324321);
      \draw[->] (1334321) -- (33234321);
      \draw[->] (343243213241) -- (3234324321);
      \draw[->] (232324321) -- (3234324321);
      \draw[->] (3234324321) -- (33234321);
    \end{tikzpicture}
    \end{figure}
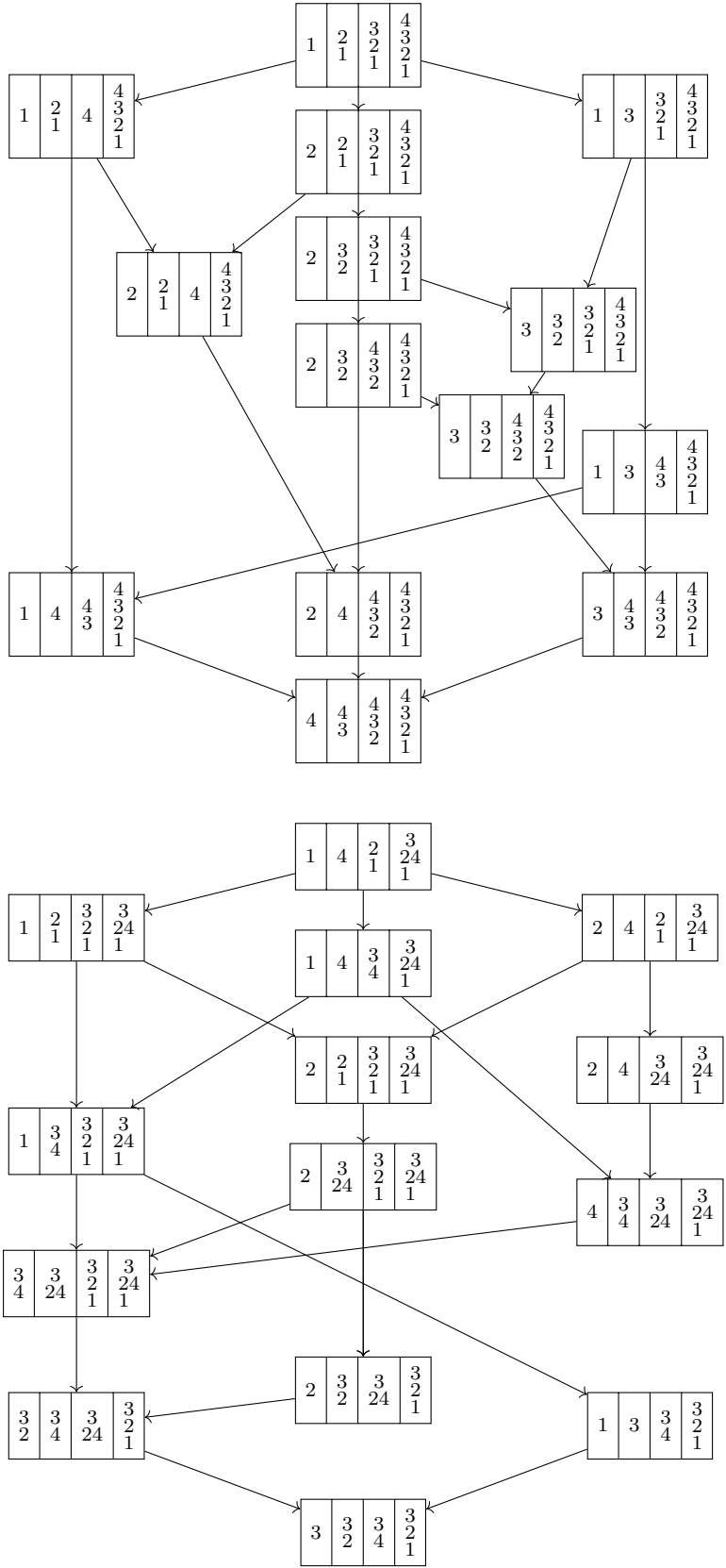
\end{example}

\appendix

\section{Proof of Corollary \ref{cor:explicit-descriptiton-An} using $\tau$-tilting reduction}\label{anotherproof}
In this section, we provide another proof of Corollary \ref{cor:explicit-descriptiton-An} without using cluster algebra theory. 
Through this section, we abbreviate $\Lambda_{A_n}$ to $\Lambda$ and $\Lambda_{A_{n-1}}$ to $\Lambda'$. Moreover, we denote by $P_1\oplus\cdots\oplus P_n$ the indecomposable decomposition of $_\Lambda\Lambda$, where $P_n$ is the unique indecomposable projective-injective module of $\Lambda$.  

First, we begin with the following proposition.

\begin{proposition}\label{delta=p}
$\Delta(\Lambda)=\st_{\Delta(s\Lambda)}(P_n,0)$.
\end{proposition}

\begin{proof}
First, we prove $\st_{\Delta(s\Lambda)}(P_n,0)\subset\Delta(\Lambda)$. Let $X=(M,P)\in \st_{\Delta(s\Lambda)}(P_n,0)$. It suffices to show that $(P_n\oplus M,P)$ is not a $\tau$-rigid pair if $P\neq 0$. We assume $P\neq 0$. Since there exists an injection from $P_i$ to $P_n$ for any $1\leq i\leq n$, we have $\Hom_\Lambda(P,P_n\oplus M)\neq 0$. Therefore, $(P_n\oplus M,P)$ is not $\tau$-rigid pair. Next, we prove $\Delta(\Lambda)\subset\st_{\Delta(s\Lambda)}(P_n,0)$. For any basic $\tau$-rigid module $M$, it suffices to show $\Hom_\Lambda(P_n\oplus M,\tau (P_n\oplus M))=0$. Since $P_n$ is projective, we have $\tau (P_n\oplus M)=\tau M$. By using the Auslander-Reiten formula, we have $\Hom_\Lambda(P_n,\tau M)=\Ext_\Lambda(M,P_n)=0$ since $P_n$ is injective. Therefore, we have
\begin{align*}
    \Hom_\Lambda(P_n\oplus M,\tau (P_n\oplus M))=\Hom_\Lambda(M,\tau M)=0
\end{align*}
Therefore, we have $\Delta(\Lambda)\subset\st_{\Delta(s\Lambda)}(P_n,0)$.
\end{proof}

We use \emph{$\tau$-tilting reduction} instead of Lemma \ref{lem:bijection-between-B-and-Bx}. We fix $A$ as a finite dimensional algebra, $U$ is a $\tau$-tilting module of $A$, and $T_U$ is the Bongarz completion of $U$. Let $B=\End_A T_U$ and $C=B/\langle e_U \rangle$, where $e_U$ is the idempotent corresponding $\Hom_A (T_U,U)$. For $M\in\mod A$, we denote by 
\begin{align*}
    0\to \mathrm{t}M \to M \to  \mathrm{f}M \to 0
\end{align*}
the canonical sequence associated with a torsion pair $(\textsf{Fac} U, U^{\perp})$. Then, we have the following theorem.

\begin{theorem}[\cite{jasso}*{Theorem 3.15, Corollary 3.16}]\label{jasso}
The map 
\begin{align*}
\Hom_\Lambda(T_U,\mathrm{f}(-))\colon s\tau\text{-}\mathrm{tilt}_UA\to s\tau\text{-}\mathrm{tilt}\ C 
\end{align*}
is a bijection. Moreover, this bijection is compatible with support $\tau$-tilting mutations.
\end{theorem}

\begin{proof}[Proof of Corollary \ref{cor:explicit-descriptiton-An}]
By Proposition \ref{delta=p} and Lemma \ref{lem.st-lk-relation}, it suffices to show that
\begin{align}\label{lk=slambda}
\lk_{\Delta(s\Lambda)}(P_n,0)\cong\Delta(s\Lambda').\end{align}
We set $A = \Lambda$ and $U = P_n$. Then, we have $T = \Lambda$, $B = \End_\Lambda \Lambda\cong \Lambda$ and $C\cong\Lambda'$. Therefore, by applying the above setting to Theorem \ref{jasso}, we have a bijection 
\begin{align*}
\mathrm{f}(-)\colon s\text{-}\mathrm{tilt}_{P_n}\Lambda\to s\text{-}\mathrm{tilt}\ \Lambda'. 
\end{align*}
Moreover, this bijection is compatible with support $\tau$-tilting mutations. Therefore, we obtain \eqref{lk=slambda} as desired.
\end{proof}

\bibliography{myrefs}
\end{document}